\documentclass[11pt, reqno]{amsart}
\usepackage{appendix}
\usepackage{xcolor}
\usepackage[utf8]{inputenc}
\usepackage[colorlinks=true,hyperindex,pagebackref,linktocpage=true]{hyperref}

\usepackage{amsfonts}
\usepackage{enumitem}
\usepackage{amssymb}
\usepackage{tikz-cd}
\usepackage{mathtools}

\hypersetup{
	colorlinks,
	linkcolor={violet!60!black},
	citecolor={cyan!60!black},
	urlcolor={orange!60!black}
}

\usepackage{stmaryrd}
\usepackage{mathrsfs}  
\usepackage{varwidth}

\usepackage{cleveref}

\theoremstyle{plain}
\newtheorem{theorem}{Theorem}[section]
\newtheorem{proposition}[theorem]{Proposition}
\newtheorem{lemma}[theorem]{Lemma}
\newtheorem{corollary}[theorem]{Corollary}
\newtheorem{conjecture}[theorem]{Conjecture}

\theoremstyle{definition}
\newtheorem{definition}[theorem]{Definition}
\newtheorem{example}[theorem]{Example}
\newtheorem{remark}[theorem]{Remark}

\newtheorem{assumption}[theorem]{Assumption}
\newtheorem{numberedparagraph}{}

\theoremstyle{definition}

\newcommand{\nc}{\newcommand}
\nc{\on}{\operatorname}

\nc{\Q}{\mathbb{Q}}
\nc{\Z}{\mathbb{Z}}
\nc{\cl}{\mathrm{cl}}

\nc{\fraka}{{\mathfrak a}} \nc{\bba}{{\mathbf a}}
\nc{\frakb}{{\mathfrak b}}
\nc{\frakc}{{\mathfrak c}}
\nc{\frakd}{{\mathfrak d}}
\nc{\frake}{{\mathfrak e}}
\nc{\frakf}{{\mathfrak f}}
\nc{\frakg}{{\mathfrak g}}
\nc{\frakh}{{\mathfrak h}}
\nc{\fraki}{{\mathfrak i}}
\nc{\frakj}{{\mathfrak j}}
\nc{\frakk}{{\mathfrak k}}
\nc{\frakl}{{\mathfrak l}}
\nc{\frakm}{{\mathfrak m}}
\nc{\frakn}{{\mathfrak n}}
\nc{\frako}{{\mathfrak o}}
\nc{\frakp}{{\mathfrak p}}
\nc{\frakq}{{\mathfrak q}}
\nc{\frakr}{{\mathfrak r}}
\nc{\fraks}{{\mathfrak s}}
\nc{\frakt}{{\mathfrak t}}
\nc{\fraku}{{\mathfrak u}}
\nc{\frakv}{{\mathfrak v}}
\nc{\frakw}{{\mathfrak w}}
\nc{\frakx}{{\mathfrak x}}
\nc{\fraky}{{\mathfrak y}}
\nc{\frakz}{{\mathfrak z}}
\nc{\frakA}{{\mathfrak A}}
\nc{\frakB}{{\mathfrak B}}
\nc{\frakC}{{\mathfrak C}}
\nc{\frakD}{{\mathfrak D}}
\nc{\frakE}{{\mathfrak E}}
\nc{\frakF}{{\mathfrak F}}
\nc{\frakG}{{\mathfrak G}}
\nc{\frakH}{{\mathfrak H}}
\nc{\frakI}{{\mathfrak I}}
\nc{\frakJ}{{\mathfrak J}}
\nc{\frakK}{{\mathfrak K}}
\nc{\frakL}{{\mathfrak L}}
\nc{\frakM}{{\mathfrak M}}
\nc{\frakN}{{\mathfrak N}}
\nc{\frakO}{{\mathfrak O}}
\nc{\frakP}{{\mathfrak P}}
\nc{\frakQ}{{\mathfrak Q}}
\nc{\frakR}{{\mathfrak R}}
\nc{\frakS}{{\mathfrak S}}
\nc{\frakT}{{\mathfrak T}}
\nc{\frakU}{{\mathfrak U}}
\nc{\frakV}{{\mathfrak V}}
\nc{\frakW}{{\mathfrak W}}
\nc{\frakX}{{\mathfrak X}}
\nc{\frakY}{{\mathfrak Y}}
\nc{\frakZ}{{\mathfrak Z}}
\nc{\bbA}{{\mathbb A}}
\nc{\bbB}{{\mathbb B}}
\nc{\bbC}{{\mathbb C}}
\nc{\bbD}{{\mathbb D}}
\nc{\bbE}{{\mathbb E}}
\nc{\bbF}{{\mathbb F}} 

\nc{\bbf}{{\mathbf f}}
\nc{\bbG}{{\mathbb G}}
\nc{\bbH}{{\mathbb H}}
\nc{\bbI}{{\mathbb I}}
\nc{\bbJ}{{\mathbb J}}
\nc{\bbK}{{\mathbb K}}
\nc{\bbL}{{\mathbb L}}
\nc{\bbM}{{\mathbb M}}
\nc{\bbN}{{\mathbb N}}
\nc{\bbO}{{\mathbb O}}
\nc{\bbP}{{\mathbb P}}
\nc{\bbQ}{{\mathbb Q}}
\nc{\bbR}{{\mathbb R}}
\nc{\bbS}{{\mathbb S}}
\nc{\bbT}{{\mathbb T}}
\nc{\bbU}{{\mathbb U}}
\nc{\bbV}{{\mathbb V}}
\nc{\bbW}{{\mathbb W}}
\nc{\bbX}{{\mathbb X}}
\nc{\bbY}{{\mathbb Y}}
\nc{\bbZ}{{\mathbb Z}}
\nc{\calA}{{\mathcal A}}
\nc{\calB}{{\mathcal B}}
\nc{\calC}{{\mathcal C}}
\nc{\calD}{{\mathcal D}}
\nc{\calE}{{\mathcal E}}
\nc{\calF}{{\mathcal F}}
\nc{\calG}{{\mathcal G}}
\nc{\calH}{{\mathcal H}}
\nc{\calI}{{\mathcal I}}
\nc{\calJ}{{\mathcal J}}
\nc{\calK}{{\mathcal K}}
\nc{\calL}{{\mathcal L}}
\nc{\calM}{{\mathcal M}}
\nc{\calN}{{\mathcal N}}
\nc{\calO}{{\mathcal O}}
\nc{\calP}{{\mathcal P}}
\nc{\calQ}{{\mathcal Q}}
\nc{\calR}{{\mathcal R}}
\nc{\calS}{{\mathcal S}}
\nc{\calT}{{\mathcal T}}
\nc{\calU}{{\mathcal U}}
\nc{\calV}{{\mathcal V}}
\nc{\calW}{{\mathcal W}}
\nc{\calX}{{\mathcal X}}
\nc{\calY}{{\mathcal Y}}
\nc{\calZ}{{\mathcal Z}}

\nc{\scrA}{{\mathscr A}}
\nc{\scrB}{{\mathscr B}}
\nc{\scrC}{{\mathscr C}}
\nc{\scrD}{{\mathscr D}}
\nc{\scrE}{{\mathscr E}}
\nc{\scrF}{{\mathscr F}}
\nc{\scrG}{{\mathscr G}}
\nc{\scrH}{{\mathscr H}}
\nc{\scrI}{{\mathscr J}}
\nc{\scrJ}{{\mathscr I}}
\nc{\scrK}{{\mathscr K}}
\nc{\scrL}{{\mathscr L}}
\nc{\scrM}{{\mathscr M}}
\nc{\scrN}{{\mathscr N}}
\nc{\scrO}{{\mathscr O}}
\nc{\scrP}{{\mathscr P}}
\nc{\scrQ}{{\mathscr Q}}
\nc{\scrR}{{\mathscr R}}
\nc{\scrS}{{\mathscr S}}
\nc{\scrT}{{\mathscr T}}

\nc{\bfG}{{\mathbf G}}
\nc{\bfK}{{\mathbf K}}

\nc{\D}{{\on{D}}}
\nc{\Div}{{\on{Div}}}
\nc{\Perv}{{\on{Perv}}}

\nc{\bnu}{{\bar{ \nu}}}
\nc{\olO}{\bar{\calO}}

\nc{\al}{{\alpha}} 
\nc{\be}{{\beta}}
\nc{\ga}{{\gamma}} \nc{\Ga}{{\Gamma}}
\nc{\hGa}{\hat{\Gamma}}
\nc{\ve}{{\varepsilon}} 
\nc{\la}{{\lambda}} \nc{\La}{{\Lambda}}
\nc{\om}{\omega} \nc{\Om}{\Omega} 
\nc{\sig}{{\sigma}} \nc{\Sig}{{\Sigma}}
\nc{\dR}{{\mathrm{dR}}}
\nc{\Perf}{{\mathrm{Perf}}}
\nc{\PSch}{{\mathrm{PSch}}}
\nc{\Gm}{{\mathbb{G}_m}}
\nc{\colim}{{\on{colim}}}

\nc{\brevI}{{\breve{\mathcal I}}}
\nc{\dg}{{\on{deg}}}
\nc{\Isoc}{{\on{Isoc}}}
\nc{\Inf}{{\on{inf}}}
\nc{\FilI}{\on{IsocFil}}
\nc{\Fil}{\on{Fil}}
\nc{\gFil}{\overline{{\on{Fil}}}}
\nc{\Irrep}{{{\on{Irrep}}}}
\nc{\wa}{{\on{wa}}}
\nc{\cris}{{\on{cris}}}
\nc{\MT}{{\on{MT}}}
\nc{\br}[1]{{\breve{#1}}}

\makeatletter
\DeclareFontEncoding{LS1}{}{}
\DeclareFontSubstitution{LS1}{stix}{m}{n}
\DeclareMathAlphabet{\rhomalpha}{LS1}{stixscr}{m}{n}
\makeatother

\nc{\Spa}{\on{{Spa}}}
\nc{\Spd}{\on{{Spd}}}
\nc{\tnb}{\psi_{\rm tame}}
\nc{\oM}{\overline{{M}}}
\nc{\op}{{\on{op}}}
\nc{\ad}{{\on{ad}}}
\nc{\ab}{{\on{ab}}}
\nc{\alg}{{\on{alg}}}
\nc{\Ad}{{\on{Ad}}}
\nc{\Adm}{{\on{Adm}}} \nc{\aff}{{\on{af}}}
\nc{\Aut}{{\on{Aut}}}
\nc{\Bun}{{\on{Bun}}}
\nc{\cha}{{\on{char}}}
\nc{\der}{{\on{der}}}
\nc{\Det}{{\on{det}}}
\nc{\Der}{{\on{Der}}}
\nc{\diag}{{\on{diag}}}
\nc{\End}{{\on{End}}}
\nc{\Fl}{{\calF\!\ell}}
\nc{\Tr}{{\on{Transp}}}
\nc{\TR}{{\calT\!\calR}}
\nc{\Gal}{{\on{Gal}}}
\nc{\Gr}{{\on{Gr}}}
\nc{\Hk}{{\on{Hk}}}
\nc{\rH}{{\on{H}}}
\nc{\Hom}{{\on{Hom}}}
\nc{\IC}{{\on{IC}}}
\nc{\id}{{\on{id}}}
\nc{\Id}{{\on{Id}}}
\nc{\ind}{{\on{ind}}}
\nc{\Ind}{{\on{Ind}}}
\nc{\Lie}{{\on{Lie}}}
\nc{\Pic}{{\on{Pic}}}
\nc{\pr}{{\on{pr}}}
\nc{\Res}{{\on{Res}}}
\nc{\res}{{\on{res}}} \nc{\Sat}{{\on{Sat}}}
\nc{\spc}{{\on{sc}}}
\nc{\drv}{{\on{der}}}
\nc{\sgn}{{\on{sgn}}}
\nc{\Spec}{{\on{Spec} \hspace{0,5mm}}}
\nc{\Spf}{\on{Spf}} 
\nc{\Sph}{\on{Sph}}
\nc{\St}{{\on{St}}}
\nc{\tr}{{\on{tr}}}
\nc{\Mod}{{\mathrm{-Mod}}}
\nc{\Hilb}{{\on{Hilb}}} 
\nc{\Ext}{{\on{Ext}}} 
\nc{\vs}{{\on{Vec}}}
\nc{\ev}{{\on{ev}}}
\nc{\nO}{{\breve{\calO}}}
\nc{\tS}{{\tilde{S}}}
\nc{\spe}{{\on{sp}}}
\nc{\loc}{{\on{loc}}}
\nc{\pre}{{\on{pre}}}
\nc{\bfmu}{{\boldsymbol{\mu}}}
\nc{\bfb}{{\mathbf{b}}}
\nc{\bfnu}{{\boldsymbol{\nu}_\bfb}}
\nc{\bfnuH}{{\boldsymbol{\nu}_{\bfb_H}}}

\nc{\co}{\colon}
\nc{\dia}{{\diamondsuit}}

\nc{\nscrR}{{\mathscr{R}^{\on{nr}}}}

\nc{\GL}{{\on{GL}}}

\nc{\Gl}{\on{Gl}} 
\nc{\GSp}{{\on{GSp}}}
\nc{\gl}{{\frakg\frakl}}
\nc{\SL}{{\on{SL}}} 
\nc{\SU}{{\on{SU}}} 
\nc{\SO}{{\on{SO}}}
\nc{\PGL}{{\on{PGL}}}

\nc{\Conv}{{\on{Conv}}}
\nc{\Rep}{\on{Rep}}
\nc{\Dom}{{\on{Dom}}}
\nc{\red}{{\on{red}}}
\nc{\act}{{\on{act}}}
\nc{\nr}{{\on{nr}}}
\nc{\ctf}{{\on{ctf}}}

\nc{\str}{{\on{-}}} 
\nc{\os}{{\bar{s}}}
\nc{\oeta}{{\bar{\eta}}}

\nc{\et}{\textup{\'et}}

\nc{\hookto}{\hookrightarrow}
\nc{\longto}{\longrightarrow}
\nc{\leftto}{\leftarrow}
\nc{\onto}{\twoheadrightarrow}
\nc{\lonto}{\twoheadleftarrow}

\nc{\pot}[1]{ [\hspace{-0,5mm}[ {#1} ]\hspace{-0,5mm}] }
\nc{\rpot}[1]{ (\hspace{-0,7mm}( {#1} )\hspace{-0,7mm}) }

\nc{\Sht}{{\on{Sht}}}
\nc{\Gsc}{{G^{\on{sc}}}}
\nc{\Gab}{{G^{\on{ab}}}}
\nc{\Gder}{{G^{\on{der}}}}
\nc{\ov}[1]{\overline{#1}}

\numberwithin{equation}{section}

\begin{document}
	
	\title{The connected components of affine Deligne--Lusztig varieties}
	
	\author[I. Gleason, D.G. Lim, Y. Xu]{Ian Gleason, Dong Gyu Lim, Yujie Xu}

\address{Mathematisches Institut der Universit\"at Bonn, Endenicher Allee 60, 53115 Bonn, Germany}
\email{ianandre@math.uni-bonn.de}

\address{University of California, Evans Hall, Berkeley, CA 94720, USA}
\email{limath@math.berkeley.edu}

\address{Columbia University, Department of Mathematics, 2990 Broadway, New York, NY 10027, USA}
\email{xu.yujie@columbia.edu}

	\begin{abstract}
		We compute the connected components of arbitrary parahoric level affine Deligne--Lusztig varieties and local Shimura varieties, thus resolving a folklore conjecture raised in \cite{HeTalk2018, Zhou20} in full generality (even for non-quasisplit groups). We achieve this by relating them to the connected components of infinite level moduli spaces of $p$-adic shtukas, where we use v-sheaf-theoretic techniques such as the specialization map of \textit{kimberlites}. Along the way, we give a $p$-adic Hodge-theoretic characterization of HN-irreducibility. 
		
		As applications, we obtain many results on the geometry of integral models of Shimura varieties of Hodge type at arbitrary stabilizer-parahoric levels. In particular, we deduce new CM lifting results on integral models of Shimura varieties for quasisplit groups at parahoric levels that arise as stabilizer Bruhat--Tits group schemes. 
	\end{abstract}
	\date{}
	
	\maketitle
	\tableofcontents

	\section{Introduction}
	\subsection{Background}
	In \cite{Rap}, Rapoport introduced certain geometric objects called affine Deligne--Lusztig varieties (ADLVs), to study mod $p$ reduction of Shimura varieties. Since then, ADLVs have played a prominent role in the geometric study of Shimura varieties, Rapoport--Zink spaces, local Shimura varieties and moduli spaces of local shtukas. 
	Moreover, results on connected components of affine Deligne--Lusztig varieties have found remarkable applications to Kottwitz' conjecture and Langlands--Rapoport conjecture, which describe mod $p$ points of Shimura varieties in relation to $L$-functions, as part of the Langlands program (for more background on this, see for example \cite{Kisin-mod-p-points}). 

	Although there have been many successful approaches \cite{Vie08,CKV,Nie,He-Zhou,Ham20,Nie21} to computing connected components of ADLVs in the past decade, as far as the authors know, the current article is the first one that approaches the problem using $p$-adic analytic geometry \textit{\`a la Scholze}. As it turns out, the $p$-adic approach proves the most general case of Conjecture \ref{mainconj} and gives a new and uniform proof to all previously known cases. 
	More precisely, we use a combination of Scholze's theory of diamonds \cite{Et}, the theory of \textit{kimberlites} due to the first author \cite{Specializ}, the connectedness of $p$-adic period domains \cite{GL22Conn}, and the normality of the local models \cite{AGLR22,GL22} to compute the connected components of ADLVs. Just as diamonds are generalizations of rigid analytic spaces, kimberlites and prekimberlites are the v-sheaf-theoretic generalizations of formal schemes. Roughly speaking, they are diamondifications of formal schemes.

	As is well-known to experts, affine Deligne--Lusztig varieties parametrize (at-$p$) isogeny classes on integral models of Shimura varieties. As an application of our main theorems, we deduce the isogeny lifting property for integral models for Shimura varieties at parahoric levels constructed in \cite{Kisin-Pappas} (see also \cite{KPZ24}). 
	Moreover, using work of Zhou \cite{Zhou20}, we give a new CM lifting result on integral models for Shimura varieties--which is a generalization of the classical Honda-Tate theory--for quasisplit groups at $p$ at stabilizer-parahoric levels. This improves on previous CM lifting results, which were proved either assuming (1) $G_{\bbQ_p}$ residually split, or assuming (2) $G$ unramified, or assuming that (3) the parahoric level is very special. 

As a further application, we prove that the Newton strata of the integral models for Shimura varieties at parahoric level constructed in \cite{pappas2021padic} satisfy $p$-adic uniformization, and that the Rapoport--Zink spaces considered in \textit{loc.cit.} agree with the moduli spaces of $p$-adic shtukas of \cite{Ber} associated to the same data. 

\subsection{Notations}
	\label{introduction-notation-section}
	We use numbered paragraphs and sections. When referring to them, we use \P\, for numbered paragraphs and \S\, for sections. 
	To avoid overloading the introduction, we use common terminology whose rigorous definitions we postpone until later (in $\mathsection$\ref{preliminaries-section}). 
	
	We denote by $\varphi$ the lift of arithmetic Frobenius to $\br{\bbQ}_p$. 
	Let $\calI$ and $\calK_p$ be $\bbZ_p$-parahoric group schemes with common generic fiber a reductive group $G$.
	We let 
	$K_p=\calK_p(\bbZ_p)$, $\brevI=\calI(\br{\bbZ}_p)$ and $\br{K}_p:=\calK_p(\br{\bbZ}_p)$. 
	We require that ${\calI(\bbZ_p)}\subseteq K_p$ and that $\calI$ is an Iwahori subgroup of $G$.

	Fix $S\subseteq G$, a $\bbQ_p$-torus that is maximally split over $\br{\bbQ}_p$. 
	Let $T=Z_G(S)$ be the centralizer of $S$; by Steinberg's theorem it is a maximal $\bbQ_p$-torus. 
	Let $B\subseteq G_{\br{\bbQ}_p}$ be a Borel, defined over $\br{\bbQ}_p$, containing $T_{\br{\bbQ}_p}$. 
	Let $\mu\in X^+_*(T)$ be a $B$-dominant cocharacter, and let $b\in G(\br{\bbQ}_p)$. 
	Let $\widetilde{W}$ be the Iwahori--Weyl group of $G$ over $\br{\bbQ}_p$. 
	Let $\on{Adm}(\mu)\subseteq \widetilde{W}=\brevI\backslash G(\br{\bbQ}_p)/\brevI$ denote the $\mu$-admissible set of Kottwitz--Rapoport (see \cite{KR00} for the case of split groups and \cite[Definition 4.23]{PRS13} for the general case). 
	
	We call the triple $(G,b,\mu)$ a \textit{$p$-adic shtuka datum} (compare with \cite[Definition 5.1]{Towards}). 
	The (closed) affine Deligne--Lusztig variety\footnote{In the literature, the closed affine Deligne--Lusztig varieties attached to $b$ and $\mu$ are more commonly denoted by $X(\{\mu\},b)$.} associated to $(G,b,\mu)$, which we will denote as $X_\mu(b)$, is a locally perfectly finitely presented $\bar{\bbF}_p$-scheme (see \cite{Zhu}), with $\bar{\bbF}_p$-valued points given by:   
	\begin{equation}X_\mu(b)=\{g\brevI\mid g^{-1}b\varphi(g)\in \brevI \on{Adm}(\mu) \brevI\}. \end{equation}
	By definition, $X_\mu(b)$ embeds into the Witt vector affine flag variety $\Fl_\brevI$, whose $\bar{\bbF}_p$-valued points are the cosets $G(\br{\bbQ}_p)/\brevI$. 
	We also consider the $\calK_p$-version $X^{\calK_p}_\mu(b)$ with $\bar{\bbF}_p$-points: 
	\begin{equation}X^{\calK_p}_\mu(b)=\{g \breve{K}_p\mid g^{-1}b\varphi(g)\in \breve{K}_p \on{Adm}(\mu) \breve{K}_p\}. \end{equation}
	Let $\bfb\in B(G)$ be the $\varphi$-conjugacy class of $b$, and let $\boldsymbol{\mu}$ be the conjugacy class of $\mu$. 
	Assume $\bfb$ lies in the Kottwitz set $B(G,\bfmu)$. 
	Let $\mu^\diamond\in X_*(T)^+_\bbQ$ denote the ``twisted Galois average" of $\mu$ (see \eqref{defn-mu-diamond}), and let $\bfnu\in X_*(T)^+_\bbQ$ denote the dominant Newton point.
	Recall that $\bfb\in B(G,\bfmu)$ implies that $\mu^\diamond-\bfnu$ is a non-negative linear combination of simple positive coroots with rational coefficients. We say that $(\bfb,\bfmu)$ with $\bfb\in B(G,\bfmu)$ is \textit{Hodge--Newton irreducible} (HN-irreducible) if all simple positive coroots have non-zero coefficient in $\mu^\diamond-\bfnu$.

	Let $\Gamma$ and $I$ denote the Galois groups of $\bbQ_p$ and $\br{\bbQ}_p$ respectively. 
	Recall that the Kottwitz map \cite[7.4]{KottwitzII}  
	\begin{equation}
		\kappa_G:G(\br{\bbQ}_p)\to \pi_1(G)_I
	\end{equation}
	induces bijections $\pi_0(\Fl_\brevI)\cong \pi_0(\Fl_{\br{K}_p}) \cong \pi_1(G)_I$ (see \cite[Theorem 0.1]{PR08} for the function field case and \cite[Proposition 1.21]{Zhu} for the mixed-characteristic case). 
	Moreover, it is known that the map induced by $\kappa_G$ on connected components of ADLV, 
	\begin{equation}\label{Hamacher-map}
	\omega_G:\pi_0(X^{\calK_p}_\mu(b))\to \pi_1(G)_I,
	\end{equation}
	factors surjectively onto $c_{b,\mu}\pi_1(G)_I^\varphi\subseteq \pi_1(G)_I$ for a unique coset element $c_{b,\mu}\in \pi_1(G)_I/\pi_1(G)_I^\varphi$ (see for example \cite[Lemma 6.1]{He-Zhou}).

\subsection{Main Results}
In his ICM talk \cite{He}, X.~He highlights the study of connected components as an important open problem in the study of the geometric properties of ADLVs. 
Inspired by results of \cite{Vie08} and \cite{CKV}, He and Zhou point to the following folklore conjecture; see \cite{HeTalk2018} and \cite[Conjecture 5.4]{Zhou20}.
	\begin{conjecture}
		\label{mainconj}
		If $(\bfb,\bfmu)$ is HN-irreducible, the following map is bijective
		\begin{align*}
		\omega_G:\pi_0(X^{\calK_p}_\mu(b))\to c_{b,\mu}\pi_1(G)_I^\varphi 
		\end{align*}  
	\end{conjecture}
	Our main theorem is the following (see \Cref{mainthmmain}). 
	\begin{theorem}
		\label{mainthm} 
		For all $p$-adic shtuka datum $(G,b,\mu)$ and all parahoric subgroups $\calK_p\subseteq G(\bbQ_p)$, \Cref{mainconj} holds.
	\end{theorem}
	\begin{remark}
		\Cref{mainconj} can be formulated for arbitrary non-archimedean local fields $F$ of residue characteristic $p$. Our proof method for \Cref{mainthm} generalizes to the case in which $F$ is a finite extension of $\bbQ_p$ (see \Cref{generalmixedchar}). 
	\end{remark}

	To state the applications to the geometry of Shimura varieties, we shall also need the following notations. Let $(\bfG, X)$ be a Shimura datum of Hodge type. Suppose $\bfG$ splits over a tamely ramified extension. 
	We shall always assume $p>2$ and $p\nmid |\pi_1(\bfG^{\on{der}})|$. 
	Recall that to any facet in the (extended) Bruhat--Tits building of $\bfG_{\bbQ_p}$ one can attach three subgroups of $\bfG(\bbQ_p)$, namely the parahoric subgroup associated with the facet, the point-wise stabilizer of the facet, and the stabilizer (or fixer) of the facet.
	We say that a parahoric subgroup $K_p\subseteq \bfG(\bbQ_p)$, is a \textit{stabilizer-parahoric} if it can be obtained as the $\bbZ_p$-points of a parahoric group scheme that arises as the stabilizer group scheme of a point in the (extended) Bruhat--Tits building of $\bfG_{\bbQ_p}$. Equivalently, $K_p$ is a stabilizer-parahoric if it coincides with the point-wise stabilizer of a facet in the (extended) Bruhat--Tits building (see \cite{haines2009corrigendum} for a related discussion). 
	By \cite[\S 4.2]{Kisin-Pappas}\footnote{In \cite{Kisin-Pappas}, the authors construct parahoric integral models assuming that $\bfG_{\bbQ_p}$ splits over a tamely ramified extension. We expect that some of the technical conditions of our corollaries can be relaxed using the constructions in \cite{Kisin-Zhou} or in \cite{pappas2021padic}.} (see also \cite{KPZ24}), 
there is a normal integral model $\scrS_{\calK_p}(\bfG,X)$, for the Shimura variety $\mathrm{Sh}_{\calK_p}(\bfG,X)$. 
\begin{assumption}
	\label{global-on-going-assump}
	For the convenience of the reader we collect our assumptions in the following list.
	\begin{itemize}
		\item $(\bfG,X)$ is of Hodge type.
		\item $p\neq 2$. 
		\item $\bfG$ splits over a tamely ramified extension.
		\item $p\nmid |\pi_1(\bfG^{\on{der}}|$.
		\item $K_p\subseteq \bfG(\bbQ_p)$ is a stabilizer-parahoric.\footnote{As explained by the authors of \cite{DHKZ}, it is natural to expect that their methods combined with the present article could allow one to remove this condition.}
	\end{itemize}
	
\end{assumption}

\Cref{isog-lifting} below is the parahoric analogue of \cite[Proposition 1.4.4]{Kisin-mod-p-points} and can be obtained by combining our \Cref{mainthm} with \cite[Propositions 6.5, 7.8]{Zhou20}. 
For the convenience of the reader, we recall the setup \textit{loc.cit.}
Recall that by construction (and after fixing an embedding of Shimura data $(\bfG,X)\to (\mathbf{GSp},S^{\pm})$ compatible with maps of parahoric group schemes $\calK_p\to \calG\calS p$ as in \cite[\S 6.2, \S 7.3]{Zhou20}), the scheme $\scrS_{\calK_p}(\bfG,X)$ comes equipped with a finite morphism\footnote{In certain circumstances this map is a closed immersion (see \Cref{corollary-xu} below).} $\epsilon:\scrS_{\calK_p}(\bfG,X)\to \scrS_{\calG\calS p}(\mathbf{GSp},S^{\pm})$ to an integral model of the Siegel modular variety (or Siegel modular scheme).
Let $\scrS^{perf}_{\calK_p,\bar{\bbF}_p}(\bfG,X)$ and $\scrS^{perf}_{\calG\calS p,\bar{\bbF}_p}(\mathbf{GSp},S^{\pm})$ denote the perfection of the $\bar{\bbF}_p$-base change of $\scrS_{\calK_p}(\bfG,X)$ and $\scrS_{\calG\calS p}(\mathbf{GSp},S^{\pm})$ respectively.
Recall that to any $\bar{x}\in \scrS_{\calK_p}(\bfG,X)(\bar{\bbF}_p)$, one can associate an element $b\in G(\br{\bbQ}_p)$ as in \cite[Lemma 1.1.12]{Kisin-mod-p-points}. 
As in \cite[\S 6.7]{Zhou20} (see also \cite[\S 6]{RZ96}) for any such $\bar{x}\in \scrS_{\calK_p}(\bfG,X)(\bar{\bbF}_p)$ one has a map obtained from $p$-adic uniformization\footnote{The twist $\varphi\mu$ by Frobenius on \Cref{first-map-Zhou} arises from the conventions in Dieudonn\'e theory adopted in \cite{Zhou20}.} 
\begin{equation} \label{first-map-Zhou} \iota'_{\bar{x}}:X^{\calK_p}_{\varphi{\mu}}(b)\to \scrS^{perf}_{\calG\calS p,\bar{\bbF}_p}(\mathbf{GSp},S^{\pm}). \end{equation}
Further, fix $\bar{x}\in \scrS_{\calK_p}(\bfG,X)(\bar{\bbF}_p)$ and $y\in \scrS_{\calK_p}(\bfG,X)(O_K)$ an arbitrary lift of $\bar{x}$ defined over $O_K$ for $K$ some finite extension of $\bbQ_p$. 
Consider $\calA_y$, $\scrG_y$ and $\scrG_y^\vee$, which are respectively the abelian variety over $K$ attached to $y$, the $p$-divisible group attached to $\calA_y$ and the Cartier dual of $\scrG_y$.
The Tate module $T_p\scrG_y^\vee$ (which identifies canonically with $\on{H}^1_{\et}(\calA_y,\bbZ_p)$) comes equipped with \'etale tensors, and one can use the $p$-adic comparison theorem to equip $\bbD(\scrG_{\bar{x}})$ (the contravariant Dieudonn\'e module of the associated $p$-divisible group at $\epsilon(\bar{x})$) with crystalline tensors $s_{\alpha,0,\bar{x}}\in \bbD(\scrG_{\bar{x}})^\otimes$. Moreover, these tensors do not depend on the choice of $y$ (see \cite[\S 6.6]{Zhou20} for details, specifically the paragraph above \cite[Corollary 6.3]{Zhou20}).
Finally, as in \cite[\S 6.7]{Zhou20} $X^{\calK_p}_{\varphi\mu}(b)$ has an $r$-Frobenius action $\Phi$ with formula $\Phi(g):=(b\varphi)^r(g)$ for $g\in X^{\calK_p}_{\varphi\mu}(b)(\bar{\bbF}_p)$, where $r$ denotes the residue degree of the field of definition of $\mu$.

	\begin{corollary}\label{isog-lifting}
		Under our assumptions (\Cref{global-on-going-assump}), for any $\bar{x}\in \scrS_{\calK_p}(\bfG,X)(\bar{\bbF}_p)$, there exists a unique map of perfect schemes 
		\begin{equation}\label{isogeny-class-map-at-x}
		\iota_{\bar{x}}: X^{\calK_p}_{\varphi\mu}(b)\to \scrS^{perf}_{\calK_p,\bar{\bbF}_p}(\bfG,X)		\end{equation}
		lifting $\iota'_{\bar{x}}$ (as in \eqref{first-map-Zhou}) and such that for all $g\in X^{\calK_p}_{\varphi\mu}(b)(\bar{\bbF}_p)$ we have $s_{\alpha,0,\iota_{\bar{x}}(g)}=s_{\alpha,0,{\bar{x}}}$. Moreover, this map is $\Phi$-equivariant for the usual geometric $r$-Frobenius action on $\scrS^{perf}_{\calK_p,\bar{\bbF}_p}(\bfG,X)$.
	\end{corollary}
	The following \Cref{isog-inject}(1) (resp.~\Cref{isog-inject}(2)) is a parahoric analogue to \cite[Proposition 2.1.3]{Kisin-mod-p-points} (resp.~\cite[Theorem 2.2.3]{Kisin-mod-p-points}) when $\bfG_{\bbQ_p}$ is quasisplit, and can be obtained by combining our \Cref{mainthm} with \cite[Proposition 9.1]{Zhou20} (resp.~\cite[Theorem 9.4]{Zhou20}). We recall the setup. 
	Given $\bar{x}\in  \scrS_{\calK_p}(\bfG,X)(\bar{\bbF}_p)$ we obtain an abelian variety $\calA_{\bar{x}}$ with associated $p$-divisible group $\calG_{\bar{x}}$. These are equipped with $\ell$-adic tensors $s_{\alpha, \ell, \bar{x}}\in \on{H}^1_{\acute{e}t}(\calA_{\bar{x}}, \bbQ_\ell)^\otimes$ for all primes $\ell\neq p$ and with crystalline tensors $s_{\alpha,0,\bar{x}}\in \bbD(\calG_{\bar{x}})^\otimes$ as in \cite[\S 9.2]{Zhou20}. 
	Recall that associated to such an abelian variety we can consider an algebraic group $\on{Aut}_\bbQ(\calA_{\bar{x}})$, defined over $\bbQ$, parametrizing automorphisms of $\calA_{\bar{x}}$ treated as an abelian variety up-to-isogeny.
	This has $R$-points given by $\on{Aut}_\bbQ(\calA_{\bar{x}})(R)=(\on{End}(\calA_{\bar{x}})\otimes_{\bbZ} R)^\times$. 
	We consider the algebraic subgroup $I_{\bar{x}}\subseteq \on{Aut}_{\bbQ}(\calA_{\bar{x}})$ of those automorphisms of the abelian variety $\calA_{\bar{x}}$ up-to-isogeny that preserve the tensors $s_{\alpha,\ell,\bar{x}}$ for all $\ell\neq p$ and $s_{\alpha,0,\bar{x}}$. 
	\begin{corollary}\label{isog-inject}
		Let the assumptions be as in \Cref{global-on-going-assump}.
		Assume further that $\bfG$ is quasisplit at $p$. 
	\begin{enumerate}
		\item the map $\iota_{\bar{x}}$ in \eqref{isogeny-class-map-at-x} induces an injective map 
		\begin{equation}\label{injective-iota-x-map}
			\iota_{\bar{x}}: I_{\bar{x}}(\Q)\backslash (X_{\varphi\mu}^{\calK_p}(b)(\bar{\bbF}_p)\times \bfG(\bbA_f^p))\to\mathscr{S}_{\calK_p}(G,X)(\bar{\bbF}_p),
		\end{equation}
	\item  The isogeny class $\iota_{\bar{x}}(X_{\varphi\mu}^{\calK_p}(b)(\bar{\bbF}_p))\times \bfG(\bbA_f^p))$ contains a point which lifts to a special point on $\mathscr{S}_{\calK_p}(G,X)$.
	\end{enumerate}
	\end{corollary}

	\Cref{mainthm} together with \cite[Theorem 8.1(2)]{Zhou20} finishes the verification of the He--Rapoport axioms \cite{HR17} for integral models of Shimura varieties. 
	\begin{corollary}\label{HR-axiom-coro-nonempty}\
		Under our assumptions (\Cref{global-on-going-assump}).			The He--Rapoport axioms hold for $\scrS_{\calK_p}(\bfG,X)$. 
	\end{corollary}
Moreover, we also obtain the following corollary by combining \cite[Theorem 2]{HamKim} with \Cref{isog-lifting}, which allow us to verify Axiom A \textit{loc.cit.} 
	\begin{corollary}
Under our assumptions (\Cref{global-on-going-assump}).
		The ``almost product structure" of the Newton strata in $\mathscr{S}_{\calK_p,\bar{\bbF}_p}(G,X)$ holds.  
	\end{corollary}
We refer the reader to \cite[Theorem 2]{HamKim} for the precise formulation of the almost product structure of Newton strata. 

As yet another corollary, we remove a technical assumption from the following theorem originally due to the third author \cite[Main Theorem]{xu-normalization}\footnote{The original version of this theorem is stated assuming $\bfG_{\bbQ_p}$ residually split for integral models at parahoric levels; at hyperspecial levels, this assumption is not necessary. We are now able to relax ``$\bfG_{\bbQ_p}$ residually split'' to ``$\bfG_{\bbQ_p}$ quasisplit'' thanks to our main \cref{mainthm}. }.
Recall that the finite map $\epsilon:\scrS_{\calK_p}(\bfG,X)\to \scrS_{\calG\calS p}(\mathbf{GSp},S^{\pm})$ factors through a closed immersion commonly denoted $\scrS^{-}_{\calK_p}(\bfG,X)\subseteq \scrS_{\calG\calS p}(\mathbf{GSp},S^{\pm})$.
\begin{corollary}
	\label{corollary-xu}
Under our assumptions (\Cref{global-on-going-assump}) and assuming that $\bfG$ is quasisplit at $p$. 
    The normalization step in the construction of $\mathscr{S}_{\calK_p}(\bfG,X)$ is unnecessary, and the closure model $\mathscr{S}^-_{\calK_p}(\bfG,X)$ is already normal. 
    In particular, 
    \[\epsilon:\scrS_{\calK_p}(\bfG,X)\to \scrS_{\calG\calS p}(\mathbf{GSp},S^{\pm})\]
    is a closed immersion.
\end{corollary}
Also, we obtain the following corollary by combining \Cref{HR-axiom-coro-nonempty} with \cite[Theorem C]{SCZ21}. See \textit{loc.cit.} for the definition of EKOR strata. 
\begin{corollary}
	Every EKOR stratum in $\scrS_{\calK_p}(\bfG,X)_{\overline{\bbF}_p}$ is quasiaffine.	
\end{corollary}

\begin{remark}
	Our main \cref{mainthm} is independent of the integral models of Shimura varieties that one works with. 
	For this reason, we expect that our methods will be equally useful to study the integral models of Shimura varieties considered by Pappas--Rapoport \cite{pappas2021padic}. 
\end{remark}
As a final application, we deduce that the Newton strata of the integral models of Shimura varieties considered by Pappas--Rapoport \cite[Theorem 4.10.6]{pappas2021padic} satisfy $p$-adic uniformization with respect to the local Shimura varieties $\calM^{\on{int}}_{\calG,b,\mu}$ of \cite[Definition 25.1.1]{Ber}. 
Let $(p,\mathbf{G},X,\mathbf{K})$ be a tuple of global Hodge type \cite[\S 1.3]{pappas2021padic}, let $\scrS_{\bfK}$ denote the integral model of \cite[Theorem 1.3.2]{pappas2021padic}, let $k$ be an algebraically closed field in characteristic $p$ and let $x_0\in \scrS_{\bfK}(k)$.
Pappas and Rapoport consider a map of v-sheaves $c:\on{RZ}_{\calG,\mu,x_0}^\diamondsuit\to \calM^{\on{int}}_{\calG,b,\mu}$ \cite[Lemma 4.10.2]{pappas2021padic}, where the source is a Rapoport--Zink space as constructed in \cite[\S 4.10.1]{pappas2021padic}.
Let the notations be as in \cite[Theorem 4.10.6, \S 4.10.2]{pappas2021padic}. We verify Conjecture $(\mathrm{U}_{x_0})$ in \cite[\S 4.10.2]{pappas2021padic} and obtain \Cref{Pappas-Rapoport-uniformization-corollary} below. Let us first recall some of the notation.
Fix $x_0\in \scrS_{\bfK}(k)$, and recall that by construction $\on{RZ}_{\calG,\mu,x_0}$ is a formal scheme equipped with a uniformization map 
\[\Theta^{\on{RZ}}_{\calG,x_0}:\on{RZ}_{\calG,\mu,x_0}\to \hat{\scrS}_{\bfK}\]
of formal schemes over $\on{Spf} O_{\breve{E}}$ (see \cite[Equation 4.10.3]{pappas2021padic}).
Using the action of $\bfG(\bbA^f_p)$, we obtain a morphism 
\begin{equation}
	\label{unifo-equation}
	\Theta^{\on{RZ}}_{\calG,x_0}:\on{RZ}_{\calG,\mu,x_0}\times \bfG(\bbA_f^p)\to \hat{\scrS}_{\bfK}
\end{equation}
The image under this map $\calI(x_0)$ after taking $k$-points is called the isogeny class of $x_0$. We also consider the algebraic group $I_{x_0}\subseteq \on{Aut}_\bbQ(\calA_{x_0})$ discussed above.
\begin{corollary} (\Cref{Pappas-Rapoport-uniformization-corollary})
	The map $c:\on{RZ}_{\calG,\mu,x_0}^\diamondsuit\to \calM^{\on{int}}_{\calG,b,\mu}$ is an isomorphism. 
	Thus, $\calM^{\on{int}}_{\calG,b,\mu}$ is representable by a formal scheme $\scrM_{\calG,b,\mu}$,
	and composing this isomorphism with \ref{unifo-equation} we obtain a $p$-adic uniformization isomorphism of $O_{\br{E}}$-formal schemes 
	\begin{equation}
		\label{PR-equation}
		I_{x_0}(\bbQ)\backslash (\scrM_{\calG,b,\mu}\times \bfG(\bbA^p_f)/\bfK^p) \to (\widehat{\scrS_{\bfK}\otimes_{O_E}O_{\br{E}}})_{/\calI(x_0)}.
	\end{equation}
\end{corollary}
\begin{remark}
	As emphasized in \cite[Theorem 4.10.6]{pappas2021padic}, the isomorphism in \Cref{PR-equation} ought to be correctly interpreted (see \cite[Theorem 6.23]{RZ96} for the precise formulation in the PEL-type case).
\end{remark}

\subsection{Rough Sketch of the argument}
Many cases of \Cref{mainconj} have been proved in literature under various additional assumptions\footnote{When $G$ is split and $\calK_p$ is hyperspecial, \cite[Theorem 2]{Vie08} applies. When $G$ is unramified, $\calK_p$ is hyperspecial and $\bfmu$ is minuscule, \cite[Theorem 1.1]{CKV} applies. When $G$ is unramified, $\calK_p$ is hyperspecial and $\bfmu$ is general, \cite[Theorem 1.1]{Nie} applies. When $G$ is residually split or when $\bfb$ is basic \cite[Theorem 0.1]{He-Zhou} applies. When $G$ is quasisplit and $\calK_p$ is very special, \cite[Theorem 1.1(3)]{Ham20} applies. When $G$ is unramified and $\calK_p$ is arbitrary, \cite[Theorem 0.2]{Nie21} applies.}, see for example \cite[Theorem 2]{Vie08}, \cite[Theorem 1.1]{CKV}, \cite[Theorem 1.1]{Nie}, \cite[Theorem 0.1]{He-Zhou}, \cite[Theorem 1.1(3)]{Ham20}, \cite[Theorem 0.2]{Nie21}.

	Previous attempts in literature used characteristic $p$ perfect geometry and 
	combinatorial arguments to construct enough ``curves" connecting the components of the ADLV. 
	In our approach, we use the theory of \textit{kimberlites} and their specialization maps \cite{Specializ}, and the general kimberlite-theoretic unibranchness result for the local models considered by Scholze--Weinstein (see \cite[$\mathsection$ 21.4]{Ber}) recently established in \cite[Theorem 1.3]{GL22}, to turn the problem of computing $\pi_0(X^{\calK_p}_\mu(b))$ into the characteristic-zero question of computing $\pi_0(\Sht_{(G,b,\mu,\calK_p)})$, where $\Sht_{(G,b,\mu,K_p)}$ denotes the moduli spaces of $p$-adic shtukas of level $K_p=\calK_p(\bbZ_p)$ as considered in \cite[Lecture 23]{Ber}. 
	We remark that when $\mu$ is non-minuscule, diamond-theoretic considerations are necessary, since the spaces $\Sht_{(G,b,\mu,K_p)}$ are not rigid-analytic spaces. 
	Moreover, even in some cases in which $\mu$ is minuscule, the theory of kimberlites is necessary here because: although $\Sht_{(G,b,\mu,K_p)}$ is representable by a rigid-analytic space, its canonical integral model is not known to be representable by a formal scheme in full generality.

	Once in characteristic zero, we are now able to exploit Fontaine's classical $p$-adic Hodge theory. 
	In our approach, the role of ``connecting curves'' is played by ``\textit{generic} crystalline representations", inspired by the ideas in \cite{Chen} (see $\mathsection$\ref{generic-MTgp-section}). 
	Intuitively speaking, the action of Galois groups can be interpreted as ``analytic paths" in the moduli spaces of $p$-adic shtukas.

	More precisely, recall the infinite level moduli space $\Sht_{G,b,\mu,\infty}$ of $p$-adic shtukas \cite[\S 23.3]{Ber}.
	Recall that it comes equipped with a Grothendieck--Messing period morphism $\pi_{\on{GM}}:\Sht_{G,b,\mu,\infty}\to \Gr^b_{\mu}$ where we let $\Gr^b_{\mu}$ denote the so-called $b$-admissible locus in Scholze's $B_{\on{dR}}$-Grassmannian (see \cite[Proposition 23.3.3]{Ber} where $\Gr^b_{\mu}$ is denoted $\Gr^a_{G,\Spd \breve{E},\mu}$). 
	Recall that $\pi_{\on{GM}}$ may be realized as the moduli space of trivializations of the universal crystalline $G(\bbQ_p)$-torsor\footnote{For local Shimura varieties coming from Rapoport--Zink spaces, this torsor corresponds to the local system defined by the $p$-adic Tate module of the universal $p$-divisible group.} over $\Gr^b_{\mu}$ \cite[\S 23.5]{Ber}. 
	Intuitively speaking, rational points of $\Gr^b_\mu$ give rise to loops in $\Gr^b_\mu$, which produce ``connecting paths'' inside any covering space over $\Gr^b_\mu$ (in particular the covering space $\Sht_{G,b,\mu,\infty}$). 
	Thus it suffices to prove that the universal crystalline representation has enough monodromy to ``connect'' $\Sht_{G,b,\mu,\infty}$. 
	We can then 
	deduce our main \cref{mainthm} at finite level $\Sht_{G,b,\mu,\calK_p}$ from the analogous result at infinite level.

\subsection{More on the arguments}
We now dig in a bit deeper into the strategy for our main \cref{mainthm}, and sketch a few more results that led to our main theorem.

To each $(G,b,\mu,{\calI(\bbZ_p)})$, one can associate a diamond $\Sht_{(G,b,\mu,{\calI(\bbZ_p)})}$, which is the moduli space of $p$-adic shtukas at level ${\calI(\bbZ_p)}$ defined in \cite{Ber}. In \cite{Gle22a}, the first author constructed a specialization map  
	\begin{equation}
		\label{specializationmap}
		\on{sp}:|\Sht_{(G,b,\mu,{\calI(\bbZ_p)})}| \to |X_\mu(b)|.
	\end{equation}
	By the unibranchness result of the first author joint with Louren\c{c}o \cite[Theorem 1.3]{GL22}, and the construction of certain v-sheaf local model correspondence due to the first author \cite[Theorem 3]{Gle22a}, the specialization map induces an isomorphism of sets
	\begin{equation}
		\label{bijectivityofspec}
		\on{sp}:\pi_0(\Sht_{(G,b,\mu,{\calI(\bbZ_p)})}\times \Spd \bbC_p) \cong \pi_0(X_\mu(b)).
	\end{equation}
Therefore we have now turned the question on $\pi_0(X_{\mu}(b))$ into a characteristic zero question on the connected components $\pi_0(\Sht_{(G,b,\mu,{\calI(\bbZ_p)})})$ of the diamond $\Sht_{(G,b,\mu,{\calI(\bbZ_p)})}$, which we now compute. 

For this purpose, we make use of the infinite level moduli space $\Sht_{(G,b,\mu,\infty)}$ of $p$-adic shtukas. 
Recall that we have a constant group v-sheaf $\underline{\calI(\bbZ_p)}$ (see \Cref{v-sheaf-theory}). 
Since $\Sht_{(G,b,\mu,{\calI(\bbZ_p)})}=\Sht_{(G,b,\mu,\infty)}/\underline{\calI(\bbZ_p)}$, and passing to connected components commutes with colimits (see \Cref{general-lemma-profinite-qt}) we have  
	\begin{equation}
	\pi_0(\Sht_{(G,b,\mu,{\calI(\bbZ_p)})})=\pi_0(\Sht_{(G,b,\mu,\infty)})/{\calI(\bbZ_p)}.
	\end{equation}
	Let $G^\ad$ denote the adjoint group of $G$. 
	Our main \cref{mainthm} follows directly from the following \Cref{secondthm}.(b) whenever $G^\ad$ does not have anisotropic factors. When $G^\ad$ is anisotropic, we give a separate argument (see the proof of \Cref{mainthmmain}). 
	Let 
	$G^\circ:=G(\bbQ_p)/\on{Im}(G^{\on{sc}}(\bbQ_p))$ denote the maximal abelian quotient of $G(\bbQ_p)$.

\begin{theorem}
		\label{secondthm} (\Cref{secondthmmain})
		Suppose that $\bfb\in B(G,\bfmu)$.   
		\begin{enumerate}[label=\alph*)]
			\item If we assume that $G$ is quasisplit, then the following statements are equivalent: 
				\begin{enumerate}[label=(\arabic*)]
			\item The map $\omega_G:\pi_0(X_\mu(b))\to c_{b,\mu}\pi_1(G)_I^\varphi$ is bijective.  
			\item The map $\omega_G:\pi_0(X^{\calK_p}_\mu(b))\to c_{b,\mu}\pi_1(G)_I^\varphi$ is bijective.  
			\item The pair $(\bfb,\bfmu)$ is HN-irreducible.
			\item There exists a field extension $K$ of finite index over $\br{\bbQ}_p$, and a crystalline representation $\xi:\Gamma_K\to G(\bbQ_p)$ with invariants $(\bfb,\bfmu)$ for which $G^{\der}(\bbQ_p)\cap \xi(\Gamma_K)\subseteq G^{\der}(\bbQ_p)$ is open. 
			\item The action of $G(\bbQ_p)$ on $\Sht_{(G,b,\mu,\infty)}$ makes $\pi_0(\Sht_{(G,b,\mu,\infty)}\times \Spd \bbC_p)$ into a $G^\circ$-torsor.
\end{enumerate}
\item If we assume that $G^\ad$ does not have anisotropic factors, then the implications 
	\[(3)\implies (4)\implies (5) \implies (1) \implies (2)\]
	still hold.
	In particular, $(3)\implies (2)$ implies \Cref{mainthm} for the case that $G^\ad$ does not have anisotropic factors. 
\item If we assume that $G$ is arbitrary, then $(5)\implies (1) \implies (2)$ and $(3)\implies (4)$ still hold.
		\end{enumerate}

\end{theorem}
	\begin{remark}
		The implication $(3)\implies (5)$ of \Cref{secondthm}.(b) confirms almost all cases (excluding the anisotropic groups) of a conjecture of Rapoport--Viehmann \cite[Conjecture 4.30]{Towards}. Moreover, we generalize the statement to moduli spaces of $p$-adic shtukas, instead of only for local Shimura varieties as \textit{loc.cit}.
	\end{remark}
\begin{remark}
	Whenever $G$ is a $\bbQ_p$-simple anisotropic group, the implication $(2)\implies (3)$ of \Cref{secondthm} can be shown to fail.  	
In this degenerate case $\omega_G$ is always bijective, even when the pair $(\bfb,\bfmu)$ is not HN-irreducible. 
Indeed, this is what equation \eqref{equation-aniso} shows. 
More concretely, if $D$ is a division algebra over $\bbQ_p$, we can let $G=\on{GL}(D)$ so that $G(\bbQ_p)=D^\times$. 
This group has a unique Iwahori subgroup $\calI$, and $\calI(\bbZ_p)=O_D^\times$ (i.e. the group of the units in the maximal order $O_D$). 
Let $\bfb$ and $\bfmu$ be trivial (which corresponds to $1\in D^\times$).   
The trivial pair $(\bfb,\bfmu)$ is not HN-irreducible, but we will see $\omega_G$ is still bijective.
We note that $X_\mu(b)$ is $0$-dimensional and that the inclusion $X_\mu(b)(\bar{\bbF}_p)\subseteq \Fl_\brevI(\bar{\bbF}_p)$ identifies with the inclusion $D^\times/O_D^\times\subseteq G(\breve{\bbQ}_p)/\calI(\breve{\bbZ}_p)$. 
Moreover, $\pi_1(G)_I^\varphi= \pi_1(G)\simeq \bbZ$ and the map 
\[\omega_G:X_\mu(b)\to \pi_1(G)_I^\varphi \]
identifies with the $p$-adic value of the norm map  
\[D^\times/O_D^\times \xrightarrow{\on{Nm}} \bbQ_p^\times/\bbZ_p^\times \xrightarrow{{\on{val}_p}} \bbZ,\]
which is bijective.

It is reasonable to expect that $(3)$ and $(5)$ should still be equivalent for a general $G$, but we do not know how to show this general case.
\end{remark}

	\begin{remark}
		The implication $(3)\implies (5)$ of \Cref{secondthm} is a more general version of the main theorem of \cite{Gle22a}, where the first author proved the statement for unramified $G$, and computed the Weil group and $J_b(\bbQ_p)$-actions on $\pi_0(\Sht_{(G,b,\mu,\infty)}\times \Spd \bbC_p)$. One should be able to combine the methods of our current paper with those \textit{loc.cit.}~to compute the Weil group and $J_b(\bbQ_p)$-actions in the more general setup of \Cref{secondthm}.
	\end{remark}
\subsubsection{Loop of the argument for \Cref{secondthm}}
		\label{proof-2-t-3}
	We now discuss the proof of \Cref{secondthm}.
	Using ad-isomorphisms and z-extensions (see $\mathsection$\ref{ad-isom-z-extn-section}), we reduce all statements of \Cref{secondthm} to the case where $G^\der$--the derived subgroup of $G$--is simply connected (see \Cref{reductionsofsecondthm}). 
	In this case, $G^\circ=G^{\on{ab}}(\bbQ_p)$ and, using the appropriate hypothesis on $G$, we prove the implications 
	\begin{center}
		$(1)\implies (2)\overset{\on{q.split}}{\implies} (3)\implies (4)\overset{\on{iso. fact.}}{\implies} (5) \implies (1)$.
	\end{center}
	Let us explain the chain of implications. 
	The implication $(1)\implies (2)$ follows from \cite[Theorem 1.1]{He16}, which says that the map $X_\mu(b)\to X^{\calK_p}_\mu(b)$ is surjective. We give a new and simple proof of this result in \Cref{thmHe}, by observing that $\Sht_{(G,b,\mu,{\calI(\bbZ_p)})}\to \Sht_{(G,b,\mu,\calK_p)}$ is automatically surjective. This again exemplifies the advantage of working on the generic fiber (of the v-sheaf $\Sht^{\calK_p}_{\mu}(b)$). For more details, see $\mathsection$ \ref{specialization-map-Sht-section}. 
	The implication $(2)\implies (3)$ follows from the HN-decomposition (\Cref{GHNThmA}) and group-theoretic manipulations (\Cref{kappamapisoHNirrep}). 
	It is only in this step that one has to assume that $G$ is quasisplit to close the circle.

		The implication $(3)\implies (4)$ follows from an explicit construction that goes back to \cite[Th\'eor\`eme 5.0.6]{Chen} when $G$ is unramified (\Cref{cor:MainChen}).
	In $\mathsection$\ref{generic-MTgp-section}, we push the methods \textit{loc.cit.} and generalize the result to arbitrary reductive groups $G$ (see also $\mathsection$\ref{generic-MT-intro} in this introduction).

	The implication $(5)\implies (1)$ follows from \eqref{bijectivityofspec} (see also \Cref{adlvtosht}) and the identification 
	$\pi_0(\Sht_{(G,b,\mu,{\calI(\bbZ_p)})})=\pi_0(\Sht_{(G,b,\mu,\infty)})/{\calI(\bbZ_p)}$. Indeed, using the map $\det: G\to G^{\ab}:=G/G^{\on{der}}$, combined with Lang's theorem, we reduce $(5)\implies(1)$ to the tori case which can be handled directly (see $\mathsection$\ref{tori-case-section}). For more details, see  $\mathsection$\ref{secondthm-4-to-1}.

\subsubsection{Proof for $(4)\implies(5)$ in \Cref{secondthm}}\label{proof-3-t-4}	
The core of the argument lies in $(4)\implies (5)$, for our argument to work we assume that $G^{\on{ad}}$ only has isotropic $\bbQ_p$-simple factors. 
For simplicity, we only discuss the case where $G$ is semisimple and simply connected in the introduction (see $\mathsection$\ref{subsection3implies4} for the general argument). 
In this (simplified) case, $G^{\circ}$ is trivial, thus it suffices to show that $\Sht_{(G,b,\mu,\infty)}\times \Spd \bbC_p$ is connected. 
The first step is to prove that $G(\bbQ_p)$ acts transitively on $\pi_0(\Sht_{(G,b,\mu,\infty)}\times \Spd \bbC_p)$ (\Cref{corollary-badmconn}). This follows from the main theorem of \cite{GL22Conn} (see \Cref{badmconn}).\footnote{Before \cite{GL22Conn} was available, the argument for \Cref{secondthm} relied on the results of \cite{Ham20} which are only available when $G$ is quasisplit.}
Let $G_x$ denote the stabilizer of $x\in \pi_0(\Sht_{(G,b,\mu,\infty)}\times \Spd \bbC_p)$. 
	Since $G(\bbQ_p)$ acts transitively, it suffices to prove $G_x=G(\bbQ_p)$. 
	Further, one can write $G_x=\bigcap_{K} G_{x_K}$, where $G_{x_K}$ denotes the stabilizer of $x_K\in \pi_0(\Sht_{(G,b,\mu,\infty)}\times \Spd K)$ as $K$ ranges over finite extensions of $\bbQ_p$ contained in $\bbC_p$. 
	We show that each $G_{x_K}=G(\bbQ_p)$ and consequently $G_x=G(\bbQ_p)$.
	
	To show that $G_{x_K}=G(\bbQ_p)$ it suffices to show that: (i) $G_{x_K}$ is open (see \Cref{open-stab-lem}), this is where hypothesis $(4)$ enters into the argument, since one can relate $G^\der(\bbQ_p)\cap \xi(\Gamma_K)$ to $G_{x_K}$; (ii) the normalizer $N_{x_K}$ of $G_{x_K}$ in $G(\bbQ_p)$ is of finite index in $G(\bbQ_p)$ (see \Cref{normalizer-fin-index}). 
	Indeed, since we assumed that $G$ is semisimple and simply connected with only isotropic factors, a standard fact from \cite[Chapter II, Theorem 5.1]{Marg} shows that $G(\bbQ_p)$ does not have finite-index subgroups. 
	Thus (ii) allows us to conclude that $G_{x_K}$ is normal in $G(\bbQ_p)$.
	Moreover, the same standard fact \textit{loc.cit.} shows that $G(\bbQ_p)$ does not have non-trivial open normal subgroups, therefore (i) implies $G_{x_K}=G(\bbQ_p)$. 

	To relate $G^\der(\bbQ_p)\cap \xi(\Gamma_K)$ and $G_{x_K}$ in order to prove (i), we use the Fargues--Fontaine reinterpretation of crystalline $\Gamma_K$-representations in terms of modifications of vector bundles on the Fargues--Fontaine curve endowed with $\Gamma_K$-action.
	These modifications give rise to $K$-points in Scholze's $B_{\on{dR}}$-Grassmannian, which relates to $\Sht_{(G,b,\mu,\infty)}$ through the Grothendieck--Messing period morphism (see \S \ref{grothendieck-section} for details).
	Now, for (ii) we exploit that the actions of $J_b(\bbQ_p)$ and $G(\bbQ_p)$ commute. 
	In particular, for all $j\in J_b(\bbQ_p)$ we have that $G_{x_K}=G_{j\cdot x_K}$.
	We can then use elements in $J_b(\bbQ_p)$ to construct elements in $N_{x_K}$.
	This, together with the key bijection of \eqref{bijectivityofspec}, allow us to translate the general finiteness results of \cite{HV20} into the finiteness of $[G(\bbQ_p):N_{x_K}]$ (see \Cref{normalizer-fin-index} for details). 
\subsubsection{The Mumford-Tate group of ``generic crystalline representations''}\label{generic-MT-intro}
Let us give more detail on the construction used to prove the implication $(3)\implies (4)$ from \S\ref{proof-2-t-3}. Fix a finite extension $K/\br{\bbQ}_p$ with Galois group $\Gamma_K:=\on{Gal}(\ov{K}/K)$, and let $\xi:\Gamma_K\to G(\bbQ_p)$ denote a conjugacy class of $p$-adic Hodge--Tate representations.
\begin{definition}
	Let $\on{MT}_{\xi}$ denote the connected component of the Zariski closure of the image of $\xi$ in $G(\bbQ_p)$.
	This is the $p$-adic Mumford--Tate group attached to $\xi$ which is well-defined up to conjugation. 
\end{definition}
It follows from results of Serre \cite[Th\'eor\`eme 1]{Serre79} and Sen \cite[\S 4, Th\'eor\`eme 1]{Sen73} (see also \cite[Proposition 3.2.1]{Chen}) that $\xi(\Gamma_K)\cap \on{MT}_{\xi}(\bbQ_p)$ is open in $\on{MT}_{\xi}(\bbQ_p)$. 
Let $\mu^\eta:\bbG_m\to G_K$ be a cocharacter conjugate to $\mu$. 
Suppose that $(b,\mu^\eta)$ defines an admissible pair in the sense of \cite[Definition 1.18]{RZ96}. 
Since $\bfb\in B(G,\bfmu)$, it induces a conjugacy class of crystalline representations $\xi_{(b,\mu^\eta)}:\Gamma_K\to G(\bbQ_p)$, and a $p$-adic Mumford--Tate group $\on{MT}_{(b,\mu^\eta)}$ attached to $\xi_{(b,\mu^\eta)}$ (See \Cref{MT-defn}).

Let $\on{Fl}_\mu:=G/P_\mu$ denote the generalized flag variety. 
We say that $\mu^\eta$ is \textit{generic} if the map $\Spec(K)\to \on{Fl}_\mu$ induced by $\mu^\eta$ lies over the generic point.\footnote{this is possible since $\br{\bbQ}_p$ has infinite transcendence degree over $\bbQ_p$.}
Our third main theorem is the following generalization of 
\cite[Th\'eor\`eme 5.0.6]{Chen} to arbitrary reductive groups. 
\begin{theorem}
	\label{enhancedchen}(\Cref{enhancedChen-maintxt})
	Let $G$ be a reductive group over $\bbQ_p$. Let $b\in G(\br{\bbQ}_p)$ and $\mu^\eta:\bbG_m\to G_K$ as above.
	Suppose that $b$ is decent, that $\mu^\eta$ is generic and that $\bfb\in B(G,\bfmu)$.
	The following hold: 
	\begin{enumerate}
		\item $(b,\mu^\eta)$ is admissible.
		\item If $(\bfb,\bfmu)$ is HN-irreducible, then $\on{MT}_{(b,\mu^\eta)}$ contains $G^\der$. 
	\end{enumerate}
\end{theorem}
Now, \Cref{secondthm} has as corollary a partial converse to \Cref{enhancedchen}. 
The following gives a $p$-adic Hodge-theoretic characterization of HN-irreducibility. 
\begin{corollary}
	\label{conversechenintro}(\Cref{chenconverse})
		Assume that $G$ is quasisplit. If $\on{MT}_{(b,\mu^\eta)}$ contains $G^\der$, then $(\bfb,\bfmu)$ is HN-irreducible. 
\end{corollary}
\begin{remark}
	Our \Cref{conversechenintro} confirms the expectation in \cite[Remarque 5.0.5]{Chen} that, at least when $G$ is quasisplit, monodromy representations attached to data $(\bfb,\bfmu)$ that are HN-reducible must factor through a proper parabolic.
\end{remark}
\subsection{Organization}
Finally, let us describe the organization of the paper. 

\S\ref{preliminaries-section} is a preliminary section. We start by collecting general notation and standard definitions that we omitted in this introduction. 
We recall the group-theoretic setup of \cite{GHN19} necessary to discuss the Hodge-Newton decomposition for general reductive groups, and its relation to the connected components of affine Deligne--Lusztig varieties. 

In \S\ref{geometric-preps} we give a brief intuitive account of the theory of kimberlites. 
We also review the geometry of affine Deligne--Lusztig varieties and their relation to moduli spaces of $p$-adic shtukas. 
Moreover, we discuss ad-isomorphisms, z-extensions and compatibility with products (which will be used in \S\ref{proof-main-thm-section} to reduce the proofs of \Cref{mainthm} and \Cref{secondthm} to the key cases). 

In \S\ref{section-HN-decomposition} we discuss the Hodge-Newton decomposition and use it to prove the implication $(2)\implies (3)$ in \Cref{secondthm}.

In \S\ref{generic-MT-section}, we discuss Mumford--Tate groups. We review \cite{Chen} and discuss the modifications needed to prove \Cref{enhancedchen}. We deduce the implication $(3)\implies (4)$ in \Cref{secondthm}.

In \S\ref{proof-main-thm-section}, we give a new proof of \cite[Theorem 7.1]{He} (see \Cref{thmHe}), and complete proofs of our main results such as \Cref{mainthm} and \Cref{secondthm}. 

\subsection*{Acknowledgements}
This project started while I.G.~and D.L.~were PhD students at UC Berkeley and Y.X.~at Harvard. Part of this project was finished while I.G. was at Mathematisches Institut der Universit\"at Bonn. We would like to thank these institutions for their support.

I.G.~would like to thank Jo\~ao Louren\c{c}o heartily for the collaborative works \cite{GL22} and \cite{GL22Conn} which are key for the present article. 
Y.X.~would like to thank Rong Zhou for various clarifications on \cite{Kisin-Zhou}.

The authors would like to thank Miaofen Chen, David Hansen, Xuhua He, Pol van Hoften, Alexander Ivanov, Mark Kisin, Michael Rapoport, Peter Scholze, Sug Woo Shin, Mingjia Zhang, and Eva Viehmann for interesting and helpful discussions. We would also like to thank Peter Scholze and Rong Zhou for helpful comments on an earlier version of this paper. 
We thank the anonymous referee for very helpful feedback on the work.

Part of this project was done while I.G.~was supported by DFG via the Leibniz-Preis of Peter Scholze. 
Part of this project was done while Y.X.~was supported by an NSF postdoctoral research fellowship. 

\section{Group-theoretic setup}\label{preliminaries-section}
We fix a reductive group $G$ defined over $\bbQ_p$. We let $G^{\on{der}}$ denote its derived subgroup, $G^{\on{sc}}$ the simply connected cover of $G^{\on{der}}$, and $G^{\on{ab}}:=G/G^{\on{der}}$.
Since $G^{\on{ab}}$ is a torus, it admits a unique parahoric model denoted by $\calG^{\on{ab}}$ defined over $\bbZ_p$.

We continue the notation from $\mathsection$\ref{introduction-notation-section}. 
Recall that $S$ is a torus of $G$ defined over $\bbQ_p$ that is maximally split over $\breve{\bbQ}_p$. 
Let $\calN=N_G(S)$ be the normalizer of $S$ in $G$.
Let $W_0:=\calN(\br{\bbQ}_p)/T(\br{\bbQ}_p)$ be the relative Weyl group. Recall that $T=Z_G(S)$ is the centralizer of $S$. Let $\calT$ denote its unique parahoric model\footnote{This is the identity component of the locally of finite type N\'eron model of $T$.}. 
Denote by $\widetilde{W}$ the Iwahori-Weyl group $\calN(\br{\bbQ}_p)/\calT(\br{\bbZ}_p)$. 
There is a $\varphi$-equivariant exact sequence (\cite{HaiRap}):
\begin{equation}\label{Iwahori-Weyl-gp-sequence}
	0 \to X_*(T)_I\to \widetilde{W}\to W_0 \to 1
\end{equation}
Let $\calA$ denote the apartment in the Bruhat--Tits building of $G_{\br{\bbQ}_p}$ corresponding to $S$. 
Let $\mathbf{a}\subseteq \calA$ denote the $\varphi$-invariant alcove determined by ${\calI(\bbZ_p)}$.
We choose a special vertex $\mathbf{o}\in \mathbf{a}$, and identify $\calA$ with $X_*(T)^I\otimes \bbR=X_*(T)_I\otimes \bbR$ by sending the origin to $\mathbf{o}$. 
Let $B$ be the Borel subgroup attached to $\mathbf{a}$ and our choice of $\mathbf{o}$. Observe that the natural linear action of $\varphi$ on $X_*(T)^I$ is the gradient of the affine action of $\varphi$ on $\calA$.
Let $\Delta\subseteq \Phi^+\subseteq \Phi\subseteq X^*(T)$ denote the set of simple positive roots, positive roots and roots attached to $B$, respectively.

The choice of $\mathbf{o}$ defines a splitting $W_0\to\widetilde{W}$, which may not be $\varphi$-equivariant.
Let $\bar{\mu}$ denote the image of $\mu$ in $X_*(T)_I$. 
For every element $\lambda \in X_*(T)_I$, let $t_\lambda$ be its image in $\widetilde{W}$ under \eqref{Iwahori-Weyl-gp-sequence}. 
Let $\bbS$ be the set of reflections along the walls of $\mathbf{a}$. 
Let $W^{\on{a}}$ be the affine Weyl group generated by $\bbS$. It is a Coxeter group. 
There is a $\varphi$-equivariant exact sequence (\cite[Lemma 14]{HaiRap}):
\begin{equation}
	1 \to W^{\on{a}}\to \widetilde{W}\to \pi_1(G)_I\to 0
\end{equation}
This sequence splits and we can write $\widetilde{W}=W^{\on{a}}\rtimes \pi_1(G)_I$.
We can extend the Bruhat order $\preceq$ given on $W^{\on{a}}$ to the one on $\widetilde{W}$ as follows:
for elements $(w_i,\tau_i)\in \widetilde{W}$ with $i=1,2$, where $w_i\in W^{\on{a}}$ and $\tau_i\in \pi_1(G)_I$, we say 
\begin{equation}
    (w_1,\tau_1)\preceq (w_2,\tau_2)
\end{equation}
if $w_1\preceq w_2$ in $W^{\on{a}}$ and $\tau_1=\tau_2\in \pi_1(G)_I$.
By \cite[Theorem 4.2]{MR3772644}, we can define the Kottwitz--Rapoport admissible set as 
\begin{equation}\on{Adm}(\mu)=\{\tilde{w}\in \widetilde{W} \mid \tilde{w}\preceq t_\la \on{with } t_\la=t_{w(\bar{\mu})} \text{ for some } w\in W_0 \}.
\end{equation}

\begin{numberedparagraph}
\label{phi0action}
Let $\widetilde{W}^{\ad}$ denote the Iwahori--Weyl group of $G^{\ad}$. By \cite[Lemma 15]{HaiRap}\footnote{More precisely, $P^\vee$ \textit{loc.cit.} acts transitively on the set of special vertices and $\sigma$ sends a special vertex $\mathbf{o}$ to a special vertex. Thus $P^\vee$ and $W_0$ together make it possible to find this element $w^{\ad}$.}, there exists an element $w^\ad\in \widetilde{W}^\ad$ such that $w^\ad\cdot \varphi(\mathbf{o})=\mathbf{o}$ and $w^\ad\cdot\varphi(\mathbf{a})=\mathbf{a}$. 
Conjugation by a lift of $w^{\ad}$ to $G^{\on{ad}}(\br{\bbQ}_p)$ gives the quasisplit inner form of $G$, which we denote by $G^\ast$. 
This defines a second action $\varphi_0$ on $G(\br{\bbQ}_p)$ (called the $L$-action), whose fixed points are $G^\ast(\bbQ_p)$ and that satisfies $\varphi_0(\calA)=\calA$, $\varphi_0(\mathbf{o})=\mathbf{o}$, $\varphi_0(B)=B$. 
\end{numberedparagraph}
\begin{numberedparagraph}\label{notation-mu-sharp}
Let $\mu\in X_*(T)^+$ be a dominant cocharacter. Denote by $\mu^\natural\in \pi_1(G)_\Gamma$ the image of $\mu$ under the natural projection $X_*(T)\to \pi_1(G)_\Gamma$. 
As in \cite[(6.1.1)]{KottwitzII}, we define 
\begin{equation}
	\mu^\diamond \coloneqq \frac{1}{[\Gamma:\Gamma_\mu]}\sum_{\gamma\in \Gamma/\Gamma_\mu} \gamma(\mu)\in (X_*(T)^+_\bbQ)^\Gamma,
\end{equation}
where the Galois action on $X_*(T)$ is the one coming from $G^\ast$. 
Via the isomorphism $X_*(T)_I\otimes\mathbb{Q}\simeq (X_*(T)\otimes\mathbb{Q})^I$ given by $[\mu]\mapsto\frac{1}{[I:I_\mu]}\sum_{\gamma\in I/I_\mu}\gamma(\mu)$, we may write $\mu^\diamond$ as follows (see \cite[A.4]{HeNie}):
\begin{align}
	&\underline{\mu}\coloneqq\frac{1}{[I:I_\mu]}\sum_{\gamma\in I/I_\mu} \gamma(\mu) \\
\label{defn-mu-diamond}
	&\mu^\diamond=\frac{1}{N}\sum_{i=0}^{N-1}\varphi^i_0(\underline{\mu})
\end{align}
Here $N$ is any integer such that $\varphi_0^N(\underline{\mu})=\underline{\mu}$, and $I_\mu$ is the stabilizer of $\mu$ associated to the action by the inertia group.
Alternatively, \begin{equation}
    \mu^\diamond=\frac{1}{N}\sum_{i=0}^{N-1}\varphi^i(\mu)^{\on{dom}}.
\end{equation} 
Here $\lambda^{\on{dom}}$ denotes the unique $B$-dominant conjugate of $\lambda$ for $\lambda \in X_*(T)\otimes \bbQ$. 
\end{numberedparagraph}

\begin{numberedparagraph}\label{newtonpointdefi}
Recall that attached to $b$, there is a slope decomposition map 
\begin{equation}\label{slope-decomp-map}
\nu_b:\bbD \to G_{\br{\bbQ}_p},
\end{equation}
where $\bbD$ is the pro-torus with $X^*(\bbD)=\bbQ$. 
We let the \textit{Newton point}, denoted as $\bfnu$, be the unique conjugate in $X_*(T)_\bbQ^+$ of \eqref{slope-decomp-map}.  
Note that by \cite[Theorem 1.8.(c)]{RR96} $\bfnu$ is invariant with respect to the $L$-action and defines an element $\bfnu\in (X_*(T)_\bbQ^+)^\Gamma$.
Recall that there is a Kottwitz map $\kappa_G:B(G)\to \pi_1(G)_\Gamma$ \cite{Kottwitz, KottwitzII}, \cite[Theorem 1.15]{RR96}. 
\begin{definition}
	\label{HN-irred-Defn}
	Let $\bfb\in B(G)$. 
\begin{enumerate}
	\item We write $\bfb\in B(G,\bfmu)$ if $\mu^\natural=\kappa_G(\bfb)$ and $\mu^\diamond-\bfnu=\sum_{\alpha\in \Delta} c_\alpha \alpha^\vee$ with $c_\alpha\in \bbQ$ and $c_\alpha\geq 0$ (see \cite[\S 6.2]{KottwitzII}.)
	\item We say $(\bfb,\bfmu)$ is \textit{HN-irreducible (Hodge--Newton irreducible)} if $\bfb\in B(G,\bfmu)$ and $c_\alpha\neq 0$ for all $\alpha\in \Delta$. 
\end{enumerate}	
\end{definition}

\begin{remark}
	In the literature, one finds formulations of \Cref{HN-irred-Defn} in terms of coefficients $c_{\alpha^{\on{rel}}}$ as $\alpha^{\on{rel}}$ ranges over the relative coroots (see for example \cite[D\'efinition 5.0.4]{Chen}). 
	As explained in \cite[\S 6.2]{KottwitzII}, this leads to the same notion.   
	Indeed, $\mu^\diamond-\bfnu$ is already $\Gamma$-invariant for the $L$-action.
\end{remark}

\begin{definition}\label{decent-definition}
\cite[Definition 1.8]{RZ96}
	Let $s\in \bbN$. We say that $b\in G(\br{\bbQ}_p)$ is \textit{$s$-decent} if $s\cdot \nu_{b}$ factors through a map $\bbG_m\to G_{\br{\bbQ}_p}$, and the decency equation 
	$(b\varphi)^s=s\cdot \nu_b(p)\varphi^s$ 
	is satisfied in $G(\br{\bbQ}_p)\rtimes \langle \varphi \rangle$. If the context is clear, we say that $b$ is \textit{decent} if it is $s$-decent for some $s$. 
\end{definition}
\begin{numberedparagraph}
	\label{decencydominant}
If $b$ is $s$-decent, then $b\in G(\bbQ_{p^s})$ and $\nu_b$ is also defined over $\bbQ_{p^s}$, where $\bbQ_{p^s}$ is the degree $s$ unramified extension of $\bbQ_p$. 
Moreover, for all $\bfb\in B(G)$, there exists an $s\in \bbN$ and an $s$-decent representative $b\in G(\bbQ_{p^s})$ of $\bfb$, such that $\nu_b=\bfnu$. 
Indeed, by \cite[1.11]{RZ96}, every $\bfb$ has a decent representative. 
Moreover, we can choose $s$ large enough such that $G$ is quasisplit over $\bbQ_{p^s}$, and then take an arbitrary $s$-decent element. Now, replacing $b$ by a $\varphi$-conjugate in $G(\bbQ_{p^s})$ preserves decency and conjugates the map $\nu_b$, thus we can assume without loss of generality that $\nu_b$ is dominant. 
\end{numberedparagraph}
\end{numberedparagraph}

	\begin{remark}
		\label{generalmixedchar}
		One can define affine Deligne--Lusztig varieties 
		over any local field $F$, and the statement of \Cref{mainthm} is conjectured to hold in this generality. Our \Cref{mainthm} holds when $F$ is a finite extension of $\bbQ_p$, via a standard restriction of scalars argument (see for example \cite[\S 5\&\S 8]{DOR10}). It is not clear if our method goes through in the equal characteristic case.
	\end{remark}

\section{Geometric background}
\label{geometric-preps}
\subsection{v-sheaf-theoretic setup}
\label{v-sheaf-theory}
We work within Scholze's framework of diamonds and v-sheaves \cite{Et}. More precisely, we consider geometric objects that are functors 
\begin{equation}
	\calF:\on{Perf}_{\bbF_p}\to \on{Sets},	
\end{equation}
where $\on{Perf}_{\bbF_p}$ is the site of affinoid perfectoid spaces in characteristic $p$, endowed with the v-topology (see \cite[Definition 8.1]{Et}). 
Recall that given a topological space $T$, we can define a v-sheaf $\underline{T}$ whose value on $(R,R^+)$-points is the set of continuous maps $|\Spa(R,R^+)|\to T$.
We will mostly use this notation $\underline{T}$ for topological groups $T$.
\begin{example}
\label{exampleofvsheaf}
$\underline{{\calI(\bbZ_p)}}$ and $\underline{G(\bbQ_p)}$ are the v-sheaf group objects attached to the topological groups ${\calI(\bbZ_p)}$ and $G(\bbQ_p)$.
\end{example}
Conversely, to any diamond or small v-sheaf $\calF$, by \cite[Proposition 12.7]{Et}, one can attach an underlying topological space that we denote by $|\calF|$. 
We note that the pair $(\underline{(\cdot)}, |\cdot|)$ behaves as if it was an adjoint pair.\footnote{A subtlety of the theory is that, in general, $\underline{T}$ is not a small v-sheaf while $|\calF|$ can only be constructed if $\calF$ is small.}
Given $\calF$ a small v-sheaf, we define $\pi_0(\calF)$ to be the set of connected components of $|\calF|$ topologized with the quotient topology with respect to the surjection $|\calF|\to \pi_0(\calF)$.
We will constantly use the following lemma. 
\begin{lemma}
	\label{general-lemma-profinite-qt}
	Let $K$ be a locally pro-finite group. 
	Let $\calF$ and $\calG$ be small v-sheaves. 
	Suppose that $\underline{K}$ acts on $\calF$ and that $\calG=\calF/{\underline{K}}$.
	Then $|\calG|=|\calF|/K$ and $\pi_0(\calG)=\pi_0(\calF)/K$.
\end{lemma}
\begin{proof}
	We claim that $|\cdot|$ commutes with colimits.
	Indeed, it is not hard to show that it commutes with arbitrary coproducts and one can use \cite[Proposition 12.10]{Et} to show that it commutes with coequalizers.
	Recall that $\pi_0$ is the left-adjoint to the inclusion of totally disconnected spaces into the category of topological spaces. 
	In particular, $\pi_0$ also commutes with colimits. 

	Recall that $\calF/K$ is an abbreviation for the coequalizer 
\[
\text{coEqual.}\left(
  \underline{K} \times \mathcal{F}
  \rightrightarrows
  \mathcal{F}
\right).
\]
	So it suffices to show that $|\underline{K}\times \calF|=K\times |\calF|$ and that $\pi_0(\underline{K}\times \calF)=K\times \pi_0(\calF)$.
	Let $S\subseteq K$ be an open profinite subgroup. 
	Writing $K=\coprod_{K/S} S$ and using that $|\cdot|$ and $\pi_0$ commute with colimits we can reduce to the case where $K$ is profinite. 
	If $K$ is finite, one can easily show that $|\underline{K}\times \calF|=K\times |\calF|$ and that $\pi_0(\underline{K}\times \calF)=K\times \pi_0(\calF)$. 
	In the general profinite case, we can write $\underline{K}=\varprojlim_{i\in I} \underline{K_i}$ where $K_i$ ranges over finite quotients of $K$.
	The desired formulas follow from \cite[Lemma 12.17]{Et}, and the fact that limits commute with products.
\end{proof}

\begin{numberedparagraph}
Recall that in the more classical setup of Rapoport--Zink spaces \cite{RZ96}, affine Deligne--Lusztig varieties arise, via Dieudonne theory, as the perfection of special fibers of Rapoport-Zink spaces. 
Moreover, the rigid generic fiber of such a Rapoport-Zink space is a special case of the so called \textit{local Shimura varieties} \cite{Towards}. 
In this way, Rapoport-Zink spaces (formal schemes) interpolate between local Shimura varieties and their corresponding affine Deligne--Lusztig varieties. 
Or in other words, Rapoport-Zink spaces serve as \textit{integral models} of local Shimura varieties whose perfected special fibers are ADLVs. Moreover, by \cite{Ber}, the diamondification functor
\begin{align*}
\dia:\{\text{Adic Spaces}/\Spa \bbZ_p\} &\longrightarrow \{\text{v-sheaves}/\Spd \bbZ_p\} \\
X&\longmapsto X^\dia
\end{align*}
applied to a local Shimura variety is a locally spatial diamond that can be identified with a moduli space of $p$-adic shtukas (see $\mathsection$\ref{specialization-map-Sht-section}). 

Alternatively, one could consider the diamondification functor applied to the entire formal schemes (such as Rapoport-Zink spaces), rather than only their rigid generic fibres. The diamondification functor naturally takes values in v-sheaves, but contrary to the rigid-analytic case, these v-sheaves are no longer diamonds. Nevertheless, the v-sheaf associated to a formal scheme still has a lot of structure. Indeed, they are what the first author calls \textit{kimberlites} \cite[Definition 4.35]{Specializ}, i.e.~we have a commutative diagram  
\begin{equation}
\begin{tikzcd}
  \{\text{Adic Spaces}/\Spa \bbZ_p\}\arrow{r}{\dia}{}  & \{\text{v-sheaves}/\Spd \bbZ_p\}  \\
  \{\text{Formal Schemes}/\Spf \bbZ_p\} \arrow[hook]{u}{}{} \ar{r}{\dia} & \{\text{Kimberlites}/\Spd \bbZ_p\}\arrow[hook]{u}{}{}
\end{tikzcd}
\end{equation}
Kimberlites share with formal schemes many pleasant properties that general v-sheaves do not. Let us list the main ones. Let $\frakX$ be a kimberlite.
\begin{enumerate}
    \item Each kimberlite has an open analytic locus $\frakX^{\on{an}}$ (which is a locally spatial diamond by definition), and a reduced locus $\frakX^{\on{red}}$ (which is by definition a perfect scheme). 
    \item Each kimberlite has a continuous ``specialization map'' whose source is $|\frakX^{\on{an}}|$ and whose target is $|\frakX^{\on{red}}|$ (see \P\ref{specializationmapparagraph} for details).
    \item Kimberlites have a formal \'etale site and a formal nearby-cycles functor $R\Psi^{\on{for}}:D_{\et}(\frakX^{\on{an}},\Lambda)\to D_{\et}(\frakX^{\on{red}},\Lambda)$ \cite{GL22}.  
\end{enumerate}

Although we expect that every local Shimura variety admits a formal scheme ``integral model'' (see \cite{PapRap22} for the strongest result in this direction), this is not known in full generality. 
Nevertheless, as the first author proved, every local Shimura variety (even the more general moduli spaces of $p$-adic shtukas) is modeled by a \textit{prekimberlite}\footnote{In fact, we expect moduli spaces of $p$-adic shtukas to be modeled by kimberlites, but for our purposes this difference is minor, as the specialization map is defined for both kimberlites and prekimberlites.} whose perfected special fiber is the corresponding ADLV (see \Cref{Thmshtprekimb}). We shall return to this discussion in $\mathsection$\ref{specialization-map-Sht-section}.

\end{numberedparagraph}

\begin{numberedparagraph}\label{specializationmapparagraph}
Recall that given a formal scheme $\calX$, one can attach a specialization triple $(\calX_\eta,\calX^{\on{red}},\on{sp})$, where $\calX_{\eta}$ is a rigid analytic space (the Raynaud generic fiber), $\calX^{\on{red}}$ is a reduced scheme (the reduced special fiber) and 
\begin{equation}
	\label{specializationmapkimb}
	\on{sp}: | \calX_\eta |\to |\calX^{\on{red}}|	
\end{equation}
is a continuous map.

Analogously, to a prekimberlite $\frakX$ \cite[Definition 4.15]{Specializ} over $\Spd(\bbZ_p)$, one can attach a specialization triple $(\frakX_\eta,\frakX^{\on{red}},\on{sp})$ where
\begin{itemize}
    \item $\frakX_\eta$ is the generic fiber (which is an open subset of the analytic locus $\frakX^{\on{an}}$ \cite[Definition 4.15]{Specializ} of $\frakX$). 
    \item $\frakX^{\on{red}}$ is a perfect scheme over $\bbF_p$ (obtained via the reduction functor \cite[\S 3.2]{Specializ}) and 
\item $\on{sp}$ is a continuous map \cite[Proposition 4.14]{Specializ} analogous to \eqref{specializationmapkimb}.

\end{itemize}
For example, if $\frakX=\calX^\dia$ for a formal scheme $\calX$, then $\frakX$ is a kimberlite, and we have $\frakX_\eta=\calX_\eta^\dia$, $\frakX^{\on{red}}=(\calX^{\on{red}})^{\on{perf}}$ and the specialization maps attached to $\calX$ and $\frakX$ agree, i.e. ~we have the following commutative diagram:
\begin{equation}
    \begin{tikzcd}
\mid\calX_\eta\mid \ar{r}{\cong}\arrow[]{d}[swap]{\on{sp}} & \mid\frakX_\eta\mid \arrow[]{d}{\on{sp}} \\
    \mid\calX^{\on{red}}\mid \arrow[]{r}{\cong} & \mid\frakX^{\on{red}}\mid
    \end{tikzcd}
\end{equation}
\end{numberedparagraph}
\begin{numberedparagraph}
	A \textit{smelted kimberlite} is a pair $(\frakX,X)$ where $\frakX$ is a prekimberlite and $X\subseteq \frakX^{\on{an}}$ is an open subsheaf of the analytic locus, subject to some technical conditions (see \cite[Definitions 4.35, 4.30]{Specializ}). 
This is mainly used when $X=\frakX^{\on{an}}$ or when $X$ is the generic fiber of a map to $\Spd \bbZ_p$ that is not $p$-adic.  

Given a smelted kimberlite $(\frakX,X)$ and a closed point $x\in |\frakX^{\on{red}}|$, one can define the tubular neighborhood $X^\circledcirc_x$ (\cite[Definition 4.38]{Specializ}). It is an open subsheaf of $X$ which, roughly speaking, is given as the locus in $X$ of points that specialize to $x$.
\end{numberedparagraph}
\subsection{$B^+_{dR}$-Grassmannians and local models}\ 

Let $\Gr_G$ be the $B^+_{dR}$-Grassmannian attached to $G$ \cite[\S 19, 20]{Ber}. This is an ind-diamond over $\Spd \br{\bbQ}_p$. 
We omit $G$ from the notation from now on, and denote by $\Gr_\mu$ the Schubert variety \cite[Definition 20.1.3]{Ber} attached to $G$ and $\mu$. 
This is a spatial diamond over $\Spd \br{E}$ where $\br{E}=E\cdot \br{\bbQ}_p$ and $E$ is the field of definition of $\bfmu$. 
Now, $\Gr_\mu$ contains the Schubert cell attached to $\mu$, which we denote by $\Gr_\mu^\circ$. 
This is an open dense subdiamond of $\Gr_\mu$.

Let $\Gr_{\calK_p}$ be the Beilinson--Drinfeld Grassmannian attached to $\calK_p$. This is a v-sheaf that is ind-representable in diamonds over $\Spd \br{\bbZ}_p$, whose generic fiber is $\Gr_G$, and whose reduced special fiber is $\Fl_{\br{\calK}_p}$. 
Let $\calM_{\calK_p,\mu}$ be the local models first introduced in \cite[Definition 25.1.1]{Ber} for minuscule $\mu$ and later extended to non-minuscule $\mu$ in \cite[Definition 4.11]{AGLR22}.

A priori, these local models are defined only as v-sheaves over $\Spd O_{\br{E}}$, but when $\mu$ is minuscule, $\calM_{\calK_p,\mu}$ is representable by a normal scheme flat over $\Spec O_{\br{E}}$ by \cite[Theorem 1.1]{AGLR22} and \cite[Corollary 1.4]{GL22}\footnote{Representability is proved in full generality in \cite{AGLR22} and normality is proven when $p\geq 5$. In \cite{GL22} normality is proved even when $p<5$.}. Moreover, in the general case, i.e.~$\mu$ not necessarily minuscule,  $\calM_{\calK_p,\mu}$ is a kimberlite by \cite[Proposition 4.14]{AGLR22}, and it is unibranch by \cite[Theorem 1.3]{GL22}.
Let $\calA_{\calK_p,\mu}$ denote the $\mu$-admissible locus inside $\Fl_{\br{\calK}_p}$ (see for example \cite[Definition 3.11]{AGLR22}). This is the unique perfect closed subscheme of $\Fl_{\br{\calK}_p}$ whose $\bar{\bbF}_p$-valued points agree with $\br{K}_p\on{Adm}(\mu) \br{K}_p/\br{K}_p$. 
The generic fiber of $\calM_{\calK_p,\mu}$ is $\Gr_\mu$ and the reduced special fiber is $\calA_{\calK_p,\mu}$ by \cite[Theorem 1.5]{AGLR22}.

\subsection{Functoriality of affine Deligne--Lusztig varieties} 
The formation of affine Deligne--Lusztig varieties is functorial with respect to morphisms of tuples $(G_1,b_1,\mu_1,\calK_{1,p})\to (G_2,b_2,\mu_2,\calK_{2,p})$. 
More precisely, we have the following lemma.
\begin{lemma}\label{functoriality-Xmub}
	Let $f:\calK_{1,p}\to \calK_{2,p}$ be a group homomorphism such that $b_2=f(b_1)$, $\mu_2=f\circ \mu_1$. Then we have a map $X^{\calK_{1,p}}_{\mu_1}(b_1)\to X^{\calK_{2,p}}_{\mu_2}(b_2)$ that fits in the following commutative diagram:
\begin{equation}
	\label{flags-adlv-comm-diag}
	\begin{tikzcd}
		X^{\calK_{1,p}}_{\mu_1}(b_1)  \ar[r] \ar[d] & X^{\calK_{2,p}}_{\mu_2}(b_2) \ar[d]\\
		\Fl_{\br{\calK}_{1,p}} \ar[r]  &  \Fl_{\br{\calK}_{2,p}},
	\end{tikzcd}
\end{equation}
where the vertical maps are the canonical inclusions.
\end{lemma}
\begin{proof}
	Recall that we have closed immersions $X^{\calK_{i,p}}_{\mu_i}(b_i)\subseteq \Fl_{\br{K}_{i,p}}$ for $i\in \{1,2\}$ which on $\bar{\bbF}_p$-points induces the evident inclusion of sets
	\[\{g\cdot \breve{K}_{i,p}\mid g^{-1}b_i\varphi(g)\in \breve{K}_{i,p}\on{Adm}(\mu_1) \br{K}_{1,p}\}\subseteq G(\breve{\bbQ}_p)/\breve{K}_{i,p}.\]
	To show that the induced map $X^{\calK_{1,p}}_{\mu_1}(b_1)\to \Fl_{\br{K}_{2,p}}$ factors through $X^{\calK_{2,p}}_{\mu_2}(b_2)$ it suffices to argue on $\bar{\bbF}_p$-points.
This follows from \Cref{functorialityadlv} below.  
\end{proof}

\begin{lemma}
	\label{functorialityadlv}
Let the setup be as in \Cref{functoriality-Xmub}. Then 
\[f(\br{K}_{1,p} \on{Adm}(\mu_1) \br{K}_{1,p}) \subseteq  \br{K}_{2,p} \on{Adm}(\mu_2) \br{K}_{2,p}.\]
\end{lemma}
\begin{proof}
We give a geometric argument. Let $\calM_{\calK_{1,p},\mu_1}$ and $\calM_{\calK_{2,p},\mu_2}$ denote the v-sheaf local models in \cite[Definition 4.11]{AGLR22}.
By the functoriality result of v-sheaf local models \cite[Proposition 4.16]{AGLR22} with respect to the map $f:\calK_{1,p}\to \calK_{2,p}$, we obtain a morphism $\calM_{\calK_{1,p},\mu_1}\to \calM_{\calK_{2,p},\mu_2}$ of v-sheaves. 
Moreover, by \cite[Theorem 6.16]{AGLR22}, we know that $\calM_{\calK_{i,p},\mu_i,\bar{\bbF}_p}\subseteq \Fl_{\br{\calK}_{i,p}}$ consists of Schubert cells parametrized by $\on{Adm}(\mu_i)$.
More precisely,  
\[\calM_{\calK_{i,p},\mu_i,\bar{\bbF}_p}(\bar{\bbF}_p)=\br{K}_{i,p} \on{Adm}(\mu_i) \br{K}_{i,p}/\br{K}_{i,p}.\] Therefore, the existence of the map of perfect schemes  
\[\calM_{\calK_{1,p},\mu_1,\bar{\bbF}_p}\to \calM_{\calK_{2,p},\mu_2,\bar{\bbF}_p}\] immediately implies that 
$f(\br{K}_{1,p} \on{Adm}(\mu_1) \br{K}_{1,p}) \subseteq  \br{K}_{2,p} \on{Adm}(\mu_2) \br{K}_{2,p}$.
\end{proof}

\begin{remark}
	\label{maps-of-points-vs-groups}
	Given a morphism $f:G_1\to G_2$, it follows from \cite[Corollary 2.10.10]{KP23} that $f$ comes via base change from a map $f:\calK_{1,p}\to \calK_{2,p}$ if and only if $f(\br{K}_{1,p})\subseteq \br{K}_{2,p}$.
	We will use these two perspectives interchangeably without further notice.
\end{remark}

\Cref{functoriality-Xmub} is most relevant in the following situations:
\begin{enumerate}
	\item When $G_1=G_2$, $f_{|_{G_1}}=\on{id}$, and $\breve{K}_{1,p}\subseteq \breve{K}_{2,p}$. 
	\item When $G_2=G_1^{\on{ab}}$ and $\calK_{2,p}$ is the only parahoric of the torus $G_1^{\on{ab}}$.
	\item When $G_2=G_1/Z$, where $Z$ a central subgroup of $G_1$ and $\breve{K}_{2,p}=f(\breve{K}_{1,p})$.
\end{enumerate}

To simplify certain proofs, we will also need the following statement.
\begin{lemma}
	\label{productsopen}
	Suppose $G=G_1\times G_2$, $b=(b_1,b_2)$, $\mu=(\mu_1,\mu_2)$ and $\calK_p=\calK_p^1\times \calK_p^2$. Then 
		$X^{\calK_p}_\mu(b)=X^{\calK^1_p}_{\mu_1}(b_1)\times X^{\calK^2_p}_{\mu_2}(b_2)$.
\end{lemma}
\begin{proof}
This follows directly from the definition.
\end{proof}

\subsection{Moduli spaces of $p$-adic shtukas}\label{specialization-map-Sht-section}
\begin{numberedparagraph}
	Recall from \cite[$\mathsection$23]{Ber} that to each $(G,b,\mu)$ and an open compact subgroup $K\subseteq G(\bbQ_p)$, one can attach a locally spatial diamond $\Sht_{(G,b,\mu,K)}$ over $\Spd \breve{E}$, where $\br{E}=\br{\bbQ}_p\cdot E$ and $E$ is the reflex field of $\bfmu$, i.e.~$\Sht_{(G,b,\mu,K)}$ is the moduli space of $p$-adic shtukas with level $K$.\footnote{It is possible to extend the definition of $\Sht_{(G,b,\mu,K)}$ to subgroups $K\subseteq G(\bbQ_p)$ that are only assumed to be closed via the formula $\Sht_{(G,b,\mu,K)}:=\Sht_{(G,b,\mu,\infty)}/\underline{K}$. These still give rise to well-behaved locally spatial diamonds.}

This association is functorial in the tuple $(G,b,\mu,K)$, i.e.~if $f:G\to H$ is a morphism of groups, we let $b_H:=f(b)$, $\mu_H:=f\circ \mu$ and we assume $f(K)\subseteq K_H$, then we have a morphism of diamonds
\begin{equation}
	\label{functorialityshtukas}
	\Sht_{(G,b,\mu,K)}\to \Sht_{(H,b_H,\mu_H,K_H)}.	
\end{equation}
\begin{enumerate}
    \item When $H=\Gab$, $f=\on{det}:G\to \Gab$ is the natural quotient map, and $K_H=\on{det}(K)=:K^{\on{ab}}$, we let $b^{\on{ab}}:=\on{det}(b)$, $\mu^{\on{ab}}:=\on{det}\circ \mu$. In this case the morphism \eqref{functorialityshtukas} is called the ``determinant map''
\begin{equation}
	\label{thedeterminantmap}
	\on{det}:\Sht_{(G,b,\mu,K)}\to \Sht_{(\Gab,b^{\on{ab}},\mu^{\on{ab}},K^{\on{ab}})}.	
\end{equation}
\item When $H=G$, $f=\on{id}$, and the inclusion $K_1\subseteq K_2$ is proper, we have a change-of-level-structures map:
\begin{equation}
	\label{changeoflevelstructureshtukas}
	\Sht_{(G,b,\mu,K_1)}\to \Sht_{(G,b,\mu,K_2)}.	
\end{equation}
\end{enumerate}
\end{numberedparagraph}

\begin{numberedparagraph}
	For parahoric levels $K_p$, $\Sht_{(G,b,\mu,K_p)}$ is the generic fiber of a canonical\footnote{canonical in the sense that $\Sht_\mu^{\calK_p}(b)$ represents a moduli problem.} integral model, which is a v-sheaf $\Sht_\mu^{\calK_p}(b)$ over $\Spd {O}_{\breve{E}}$ defined in \cite[\S 25]{Ber} (see also \cite[Defintions 2.45, 2.51]{Gle22a}). 
In \cite{Gle22a}, the first author proved that $\Sht_\mu^{\calK_p}(b)$ is a well-behaved prekimberlite (see \cite[Definition 4.15]{Specializ}). 
Moreover, by  \cite[Proposition 2.61]{Gle22a}, its reduction (or its reduced special fiber in the sense of \cite[\S 3.2]{Specializ}) can be identified with $X^{\calK_p}_\mu(b)$.
Furthermore, the formalism of kimberlites developed in \cite{Specializ} gives a continuous specialization map 
which turns out to be surjective (on the underlying topological spaces). 
\begin{theorem}\cite[Theorem 2.76]{Gle22a}\label{Thmshtprekimb}
The pair $(\Sht_\mu^{\calK_p}(b),\Sht_{(G,b,\mu,K_p)})$ is a rich smelted kimberlite\footnote{The term ``rich'' refers to some technical finiteness assumption that ensures that the specialization map can be controlled by understanding the preimage of the closed points in the reduced special fiber.}. Moreover,  $\Sht_\mu^{\calK_p}(b)^{\on{red}}=X^{\calK_p}_\mu(b)$. In particular, we have a surjective and continuous specialization map.  
\begin{equation}\label{specialization-map-thm}
\on{sp}:|\Sht_{(G,b,\mu,K_p)}|\to |X^{\calK_p}_\mu(b)|.
\end{equation}
\end{theorem}
\begin{numberedparagraph}
	Now we recall the local model correspondence of \cite[Theorem 3]{Gle22a}\footnote{We warn the reader that this local model correspondence, although related to, does not agree with the more classical local model diagrams considered in the literature of Shimura varieties and Rapoport--Zink spaces. 
	Indeed, the classical local model diagram is deduced from the finite-dimensional map $\Sht^{\calK_p}_\mu(b) \to [\calK_p\backslash \calM_{\calK_p,\mu}]$. 
On the other hand, the correspondence in \eqref{local-model-correspondence} is more closely related to the map $\Sht^{\calK_p}_\mu(b)\to \Hk_{\calK_p,\mu}$ towards the local Hecke stack, and this map has infinite-dimensional fibers.}. 
	Fix an arbitrary extension $\br{\bbQ}_p\subseteq \breve{E}\subseteq  F$ of non-archimedean fields. Let $O_F\subseteq F$ the ring of integers with residue field $k_F$.
	The diagram has the form 
\begin{equation}
	\label{local-model-correspondence}
\begin{tikzcd}
   & X \ar{dr}{f} \ar{dl}[swap]{g} &  \\
  (\Sht^{\calK_p}_\mu(b)\times \Spd {O}_F)_x^\circledcirc & & (\calM_{\calK_p,\mu}\times \Spd O_F)^\circledcirc_y 
\end{tikzcd}
\end{equation}
where $x\in X^{\calK_p}_{\mu,k_F}(b)(k_F)$, $y\in \calA_{\calK_p,\mu,k_F}(k_F)$ are closed points and the maps $f$ and $g$ are $\widehat{L^+_{\bbW}G}$-torsors for a certain connected small v-sheaf in groups $\widehat{L^+_{\bbW}G}$ (see \cite[Definition 8.1]{Et} and \cite[Proposition 3.4.7]{pappas2021padic}). 
This correspondence induces a bijection
\begin{equation}\label{local-model-diagrams} 
	\pi_0((\Sht^{\calK_p}_\mu(b)\times \Spd {O}_F)_x^\circledcirc)\simeq \pi_0((\calM_{\calK_p,\mu}\times \Spd O_F)^\circledcirc_y)
\end{equation}
See \cite[\S 2.5]{Gle22a} for details.
Finally, recall the following theorem due to the first author in joint work with Louren\c{c}o.

\begin{theorem}\label{tubular-nbhd-conn-prop}\cite[Theorem 1.3]{GL22}
 For any parahoric $K_p\subseteq G(\bbQ_p)$ and any field extension $\br{E}\subseteq F \subseteq \bbC_p$, the tubular neighborhoods of $(\calM_{\calK_p,\mu}\times \Spd {O}_F, \Gr_\mu\times \Spd F)$ are connected.
\end{theorem}
Using \eqref{local-model-diagrams} and \Cref{tubular-nbhd-conn-prop}, one can show that the specialization map \eqref{specialization-map-thm} induces a bijection $\pi_0(\on{sp})$ on connected components.
\begin{theorem}\cite[Theorem 2]{Gle22a}
	\label{adlvtosht}
	For any parahoric $K_p\subseteq G(\bbQ_p)$ and any field extension $\br{E}\subseteq F \subseteq \bbC_p$, the map 
	
	\begin{equation}\pi_0(\on{sp}):\pi_0(\Sht_{(G,b,\mu,K_p)}\times \Spd F )\xrightarrow{\sim}  \pi_0(X^{\calK_p}_\mu(b))
	\end{equation}
	is bijective. 	
\end{theorem}
\begin{proof}
Recall that by \cite[Lemma 4.55]{Specializ}, whenever $(\frakX,X)$ is a rich smelted kimberlite, to prove that 
\begin{equation}
\pi_0(\on{sp}):\pi_0(X)\to \pi_0(\frakX^{\on{red}})
\end{equation}
is bijective, it suffices to prove that $(\frakX,X)$ is unibranch\footnote{The definition of unibranchness, or alternatively \textit{topological normality}, for smelted kimberlites is inspired by a useful criterion for the unibranchness of a scheme (see \cite[Proposition 2.38]{AGLR22}).} (in the sense of \cite[Definition 4.52]{Specializ}), i.e.~tubular neighborhoods are connected. 
By \Cref{Thmshtprekimb}, $(\Sht_\mu^{\calK_p}(b),\Sht_{(G,b,\mu,K_p)})$ is a rich smelted kimberlite, and thus it suffices to prove that $(\Sht_\mu^{\calK_p}(b) \times \Spd {O}_F, \Sht_{(G,b,\mu,K_p)}\times \Spd F)$ is unibranch, i.e. their tubular neighborhoods are connected. 

By \eqref{local-model-diagrams}, it suffices to prove that the tubular neighborhoods of $(\calM_{\calK_p,\mu}\times \Spd {O}_F, \Gr_\mu\times \Spd F)$ are connected, which follows from \Cref{tubular-nbhd-conn-prop}.   
\end{proof}
\end{numberedparagraph}

	Recall the reduction functor \cite[Definition 3.12]{Specializ} (see also \cite[Definition 2.19]{AGLR22} and \cite[\S 2.2.1]{Gle22a}).
	This functor takes as input a small v-sheaf $X$ and attaches to it the unique scheme-theoretic v-sheaf $X^\red$ whose value on characteristic $p$ perfect affine schemes is 
	\[X^\red(\Spec R)=X(\Spd(R,R)).\]

The following proves that the formation of $\Sht_\mu^{\calK_p}(b)$ is also functorial in tuples $(G,b,\mu,\breve\calK_p)$ and generalizes \Cref{functoriality-Xmub}.
\begin{lemma}
	Let $f:\calK_{1,p}\to \calK_{2,p}$ be a group homomorphism such that $b_2=f(b_1)$, $\mu_2=f\circ \mu_1$ and $f(\calK_{1,p})\subseteq \calK_{2,p}$. 
Then we have a map 
\begin{equation}\label{map-on-integral-Sht}
\Sht^{\calK_{1,p}}_{\mu_1}(b_1)\to \Sht^{\calK_{2,p}}_{\mu_2}(b_2)
\end{equation}
of v-sheaves.
Moreover, taking the reduction functor of the map \eqref{map-on-integral-Sht} induces the map $X^{\calK_{1,p}}_{\mu_1}(b_1)\to X^{\calK_{2,p}}_{\mu_2}(b_2)$ of \Cref{functoriality-Xmub}.
\end{lemma}
\begin{proof}
	Recall that for $i\in \{1,2\}$, $\Sht^{\calK_{i,p}}_{\mu_i}(b_i)$ is defined as a closed subsheaf of another v-sheaf $\Sht_{\calK_{i,p}}(b_i)\times \Spd O_{\breve{E}}$ (\cite[Definition 2.45, 2.51, Proposition 2.54]{Gle22a}).
	Explicitly, $\Sht_{\calK_{i,p}}(b_i)$ parametrizes triples $(\calT,\Phi,\lambda)$ where $(\calT,\Phi)$ is a $\calK_{i,p}$-shtuka and $\lambda$ is an isogeny towards $\calE_{b_i}$, where $\calE_{b_i}$ denotes the vector bundle on the Fargues--Fontaine curve determined by $b_i$ (see \cite[Definition 2.13, Remark 2.16, Definition 2.17]{Gle22a}). 

	Using the formula 
	\[\calT\mapsto \calT\overset{\calK_{1,p}}{\times}\calK_{2,p},\]
	and functoriality, a map of groups $f:\calK_{1,p}\to \calK_{2,p}$ induces a map of v-sheaves $\Sht_{\calK_{1,p}}(b_1)\to \Sht_{\calK_{2,p}}(b_2)$.
	
	We must argue that the composition 
	\[\Sht^{\calK_{1,p}}_{\mu_1}(b_1)\to \Sht_{\calK_{2,p}}(b_2)\times \Spd O_{\breve{E}} \to \Sht_{\calK_{1,p}}(b_1)\times \Spd O_{\breve{E}}\] factors through the closed immersion $\Sht^{\calK_{2,p}}_{\mu_2}(b_2)\subseteq \Sht_{\calK_{2,p}}(b_2)\times \Spd O_{\breve{E}}$.
	This follows from \Cref{functorialityadlv} (see \cite[Remark 2.52]{Gle22a}).

	For the second statement we argue as follows.
	Consider the following commutative diagram 
\begin{equation}
	\label{some-comm-diagr}
\begin{tikzcd}
 \Sht^{\calK_{1,p}}_{\mu_1}(b_1)\arrow{r} \arrow{d}  &  \arrow{d} \Sht^{\calK_{1,p}}_{\mu_2}(b_2) \\
 \Sht_{\calK_{1,p}}(b_1)\times \Spd O_{\breve{E}} \arrow{r} & \Sht_{\calK_{2,p}}(b_2) \times \Spd O_{\breve{E}} 
\end{tikzcd}
\end{equation}

	By \cite[Propositions 2.61]{Gle22a}, we have the identity $X^{\calK_{i,p}}_{\mu_i}(b_i)=\Sht^{\calK_{i,p}}_{\mu_i}(b_i)^{\on{red}}$, with $i\in \{1,2\}$. 
	It follows from \cite[Propositions 2.58, 2.60]{Gle22a}, that after applying the reduction functor to \Cref{some-comm-diagr}, we obtain the following commutative diagram:
\begin{center}
\begin{tikzcd}
	X^{\calK_{1,p}}_{\mu_1}(b_1)\arrow{r} \arrow{d}  & X^{\calK_{2,p}}_{\mu_2}(b_2)   \arrow{d}  \\
  \Fl_{\br{\calK}_{1,p}} \arrow{r} & \Fl_{\br{\calK}_{2,p}} 
\end{tikzcd}
\end{center}
where the vertical arrows are the natural inclusions and the bottoms arrow is the canonical map attached to $f$.
This finishes the proof.
\end{proof}
As a special case, if we fix a datum $(G,b,\mu)$ and two parahorics ${K}_1\subseteq {K}_2$ of $G(\bbQ_p)$, we have a map  
\begin{equation}\label{map-integral-Sht}
\Sht^{\calK_1}_\mu(b)\to \Sht^{\calK_2}_\mu(b)
\end{equation}
of v-sheaves. On the generic fiber, the map \eqref{map-integral-Sht} gives the change-of-level-structures map of \eqref{changeoflevelstructureshtukas}.
After applying the reduction functor  to the map \eqref{map-integral-Sht}, we recover the map $X^{\calK_1}_\mu(b)\to X^{\calK_2}_\mu(b)$ from \Cref{functoriality-Xmub} applied to scenario (1).
\end{numberedparagraph}

\subsection{The Grothendieck--Messing period map}\label{grothendieck-section}
Recall that given a triple $(G,b,\mu)$, there is a quasi-pro-\'etale \textit{Grothendieck--Messing period morphism} (see for example \cite[\S 23]{Ber}): 
\begin{equation}
	\label{Grothendieckmessing}
\pi_{\on{GM}}:\Sht_{(G,b,\mu,\infty)}\times \Spd \bbC_p\to \Gr_\mu \times \Spd \bbC_p.
\end{equation}
Now, the $b$-admissible locus $\Gr^b_\mu\subseteq \Gr_\mu$, can be defined as the image of $\pi_{\on{GM}}$, which is an open subset by \cite[Proposition 23.3.3]{Ber}.
Moreover, there is a (universal)
$\underline{G(\bbQ_p)}$-torsor $\bbL_b$ over $\Gr^b_\mu$, such that for each finite extension $K$ over $\br{E}$ and $x\in \Gr_{\mu}^b(K)$,  $x^*\bbL_{b}$ is a 
crystalline representation associated to the isocrystal with $G$-structure defined by $b$ (for more details see for example \cite[\S 2.2-2.4]{Gle21a}). 
The map in \eqref{Grothendieckmessing} can be then constructed as the geometric $\underline{G(\bbQ_p)}$-torsor attached to $\bbL_b$, i.e.~$\Sht_{(G,b,\mu,\infty)}$ is the moduli space of trivializations of $\bbL_b$. The first author together with Louren\c{c}o prove the following theorem using diamond-theoretic techniques.
\begin{theorem}[\cite{GL22Conn}]
	\label{badmconn}
	Let $(G,b,\mu)$ be a $p$-adic shtuka datum with $\bfb\in B(G,\bfmu)$.
	The $b$-admissible locus $\Gr^b_\mu\times \Spd \bbC_p$ is connected and dense within $\Gr_\mu\times \Spd \bbC_p$.
\end{theorem}
From this we deduce the following. 
\begin{proposition}
	\label{corollary-badmconn}
	The $G(\bbQ_p)$-action on $\pi_0(\Sht_{(G,b,\mu,\infty)}\times \Spd \bbC_p)$ is transitive.
\end{proposition}
\begin{proof}
	This follows from \Cref{general-lemma-profinite-qt} and the identity of v-sheaves 
	\[\Gr^b_\mu\times \Spd \bbC_p\simeq \Sht_{(G,b,\mu,\infty)}/\underline{G(\bbQ_p)}.\] 
	Indeed, on connected components we get a sequence of identifications 
	\begin{align*}
		\pi_0(\Gr^b_\mu\times \Spd \bbC_p) & \simeq  \pi_0(\Sht_{(G,b,\mu,\infty)}/\underline{G(\bbQ_p)}) \\
						   &  \simeq  \pi_0(\Sht_{(G,b,\mu,\infty)})/G(\bbQ_p).
	\end{align*}		
From \Cref{badmconn}, it follows that $\pi_0(\Sht_{(G,b,\mu,\infty)}\times \Spd \bbC_p)$ is acted on transitively by $G(\bbQ_p)$.  
\end{proof}
We also get the following result which from a different perspective is much harder to obtain.
\begin{corollary}
	\label{kisins-surjectivity-argument}
	For any point $x\in \Sht_{(G,b,\mu,\calK_p)}(\bbC_p)$ the map:
	\[G(\bbQ_p)/\calK_p \xrightarrow{g\mapsto g\cdot x} \Sht_{(G,b,\mu,\calK_p)}(\bbC_p) \xrightarrow{\on{sp}}X_\mu^{\calK_p}(b)(\bar{\bbF}_p)\] 
	induces a surjection $G(\bbQ_p)/\calK_p\to \pi_0(X_\mu^{\calK_p}(b))$.
\end{corollary}
\begin{proof}
	Let $\tilde{x}\in \Sht_{(G,b,\mu,\infty)}(\bbC_p)$ be any lift of $x$ and consider the following diagram:	
\begin{center}
\begin{tikzcd}
	G(\bbQ_p) \arrow{r}{g\mapsto g\cdot\tilde{x}} \arrow{d} & \Sht_{(G,b,\mu,\infty)}(\bbC_p) \ar{r} \ar{d} & \pi_0(\Sht_{(G,b,\mu,\infty)} \times \Spd \bbC_p) \arrow{d} \\
	G(\bbQ_p)/\calK_p \arrow{r}{g\cdot \calK_p \mapsto g\cdot {x}}& \Sht_{(G,b,\mu,\calK_p}(\bbC_p) \ar{r} &\pi_0(\Sht_{(G,b,\mu,\calK_p}\times \Spd \bbC_p) 
\end{tikzcd}
\end{center}
the map $G(\bbQ_p)\to \pi_0(\Sht_{(G,b,\mu,\infty)} \times \Spd \bbC_p)$ is $G(\bbQ_p)$-equivariant and by \Cref{corollary-badmconn} surjective. The right arrow is also surjective since the map of spaces $\Sht_{(G,b,\mu,\infty)} \times \Spd \bbC_p\to \Sht_{(G,b,\mu,\calK_p)} \times \Spd \bbC_p$ is a $\underline{\calK_p}$-torsor.
Finally, by \Cref{adlvtosht} the map $G(\bbQ_p)/\calK_p\to \pi_0(X_\mu^{\calK_p}(b))$ is surjective since it is the composition of two surjective maps.
\end{proof}

\subsection{The Bialynicki-Birula map}
Recall the Bialynicki-Birula map (see \cite[Proposition 19.4.2]{Ber}) from the Schubert cell $\Gr_{\mu}^{\circ}$ to the generalized flag variety $\on{Fl}_\mu:=G/P_\mu$ 
\begin{equation}\label{Bia-Bir-map}
	\on{BB}:\Gr^\circ_\mu \to \on{Fl}_\mu.
\end{equation}
In general, the map \eqref{Bia-Bir-map} is not an isomorphism (it is an isomorphism only when $\mu$ is minuscule), but it always induces a bijection on classical points, i.e.~finite extensions $F$ of $\br{E}$ (see for example \cite[Theorem 5.2]{viehmannBdr}). 

Let $\on{Fl}_\mu^{\on{adm}}\subseteq \on{Fl}_\mu$ denote the weakly admissible (or equivalently, semistable) locus inside the flag variety \cite[\S 5]{DOR10}, and let $\Gr_\mu^{\circ,b}:=\Gr_\mu^\circ \cap \Gr_\mu^b$. 
By \cite{ColmFont}, we have a bijection $\on{BB}:\Gr^{\circ,b}_\mu(F)\cong {\on{Fl}}^{\on{adm}}_\mu(F)$ for all finite extensions $F$ of $\br{E}$. 
Moreover, \eqref{Bia-Bir-map} fits in the following commutative diagram:
\begin{equation}\label{Bialynickiadmissible}
	\begin{tikzcd}
		\Gr^{\circ,b}_\mu	\arrow{r} \ar{d}[swap]{\on{BB}} & \Gr^\circ_\mu  \ar[d, "\on{BB}"] \\
		{\on{Fl}}^{\on{adm}}_\mu \arrow[r]  & {\on{Fl}}_\mu.   \\
	\end{tikzcd}
\end{equation}
\subsection{Ad-isomorphisms and z-extensions}\label{ad-isom-z-extn-section}
\begin{definition}\label{defi-ad-isomorphism}\cite[$\mathsection$4.8]{Kottwitz}
A morphism $f:G\to H$ is called an \textit{ad-isomorphism} if $f$ sends the center of $G$ to the center of $H$ and induces an isomorphism of adjoint groups. 
\end{definition}
An important example of ad-isomorphisms are z-extensions. 
\begin{definition}\label{z-extensions}
\cite[$\mathsection 1$]{RatConjug}
A map of connected reductive groups $f:G'\to G$ is a \textit{z-extension} if: $f$ is surjective, $Z=\on{Ker}(f)$ is central in $G'$, $Z$ is isomorphic to a product of tori of the form $\on{Res}_{F_i/\bbQ_p}\bbG_m$ for some finite extensions $F_i\subseteq \ov{\bbQ}_p$, and $G'$ has simply connected derived subgroup.

\end{definition}

\begin{lemma}
	\label{lifting-b-and-mu}
Let $f:\tilde{G}\to G$ be a z-extension and $\bfb\in B(G,\bfmu)$. \\
	(1) There exist a conjugacy class of cocharacters $\tilde{\bfmu}$ and an element $\tilde{\bfb}\in B(\tilde{G},\tilde{\bfmu})$ which, under the map $B(\tilde{G},\tilde{\bfmu})\to B(G,\bfmu)$, map to $\bfmu$ and $\bfb$, respectively.\\
	(2) $c_{\tilde{b},\tilde{\mu}}\pi_1(\tilde{G})^\varphi_I \to c_{b,\mu}\pi_1(G)^\varphi_I$ is surjective. 
\end{lemma}
\begin{proof}
	(1) Let $T\subseteq G$ be a maximal torus and $\tilde{T}\subseteq \tilde{G}$ its preimage under $f$. 
	Let $Z=\on{Ker}(f)$. We have an exact sequence 
	\begin{equation}
	0\to Z \to \tilde{T}\to T \to 0
	\end{equation}
	Since $Z$ is a torus, we have an exact sequence:
	\begin{equation}
	0\to X_*(Z) \to X_*(\tilde{T})\to X_*(T) \to 0
	\end{equation}
	In particular, we can lift $\bfmu$ to an arbitrary $\tilde{\bfmu}\in X_*(\tilde{T})$.
	To lift $\tilde{\bfb}$ compatibly, it suffices to recall from \cite[(6.5.1)]{KottwitzII} that 
	\begin{equation}B(G,\bfmu)\cong B(G^{\ad},\bfmu^{\ad})\cong B(\tilde{G},\tilde{\bfmu}).
	\end{equation}
	
	(2) Recall that the map $G(\bbQ_p)\to \pi_1(G)_{I}^\varphi$ is surjective (see for example \cite[Lemma 5.18]{Zhou20}). 
	Indeed, this follows from the exact sequence
	\begin{equation}
	0 \to \calT(\br{\bbZ}_p)\to T(\br{\bbQ}_p)\to \pi_1(G)_I \to 0
	\end{equation}
	and the group cohomology vanishing $H^1(\bbZ,\calT(\br{\bbZ}_p))=0$, where $\calT$ is the unique parahoric of $T$ and the $\bbZ$-action on $\calT(\br{\bbZ}_p)$ is given by the Frobenius $\varphi$. 
	Consider the following commutative diagram:
\begin{equation}\label{diagram-z-extension-lemma}
	\begin{tikzcd}
		\tilde{G}(\bbQ_p)  \ar[r] \ar[d] & \pi_1(\tilde{G})_I^\varphi \ar[d]\\
		G(\bbQ_p) \ar[r]  & \pi_1(G)^\varphi_I 
	\end{tikzcd}
\end{equation}
The horizontal arrows in \eqref{diagram-z-extension-lemma} are surjective. 
Since $Z$ is an induced torus, $H^1_{\et}(\Spec \bbQ_p, Z)=0$. Thus by the exact sequence of pointed sets that   
\begin{equation}
    	0\to Z \to \tilde{G}\to G \to 0,
\end{equation}
induces, the map $\tilde{G}(\bbQ_p)\to G(\bbQ_p)$ is surjective. 
Therefore $\pi_1(\tilde{G})_I^\varphi \to \pi_1(G)_I^\varphi$ is surjective. 
Finally, since $\tilde{\bfb}$ and $\tilde{\bfmu}$ map to $\bfb$ and $\bfmu$, the coset $c_{\tilde{b},\tilde{\mu}}\pi_1(\tilde{G})_I^\varphi$ also maps to the coset $c_{{b},{\mu}}\pi_1(\tilde{G})_I^\varphi$.
\end{proof}

Assume that $f$ is an ad-isomorphism for the rest of this subsection. 
Let $b_H:=f(b)$ and $\mu_H:=f\circ \mu$. Let $\calK^H_p$ denote the unique parahoric of $H$ that corresponds to the same point in the Bruhat--Tits building as $\calK_p$. 
\begin{proposition}
	\label{adisoforadlv}
The following diagram is Cartesian:	
\begin{equation}\label{diagram-adisoforadlv}
	\begin{tikzcd}
		\pi_0(X^{\calK_p}_\mu(b)) \ar[r] \ar[d] & c_{b,\mu}\pi_1(G)^\varphi_I \ar[d]\\
		\pi_0(X^{\calK^H_p}_\mu(b)) \ar[r] & c_{b_H,\mu_H}\pi_1(H)^\varphi_I 
	\end{tikzcd}
\end{equation}
\end{proposition}
\begin{proof}
	
	This is a consequence of \cite[Lemma 5.4.2]{PapRap22}, which is a generalization of \cite[Corollary 2.4.2]{CKV} for arbitrary parahorics. 
\end{proof}

\section{Hodge--Newton decomposition}
\label{section-HN-decomposition}
We can classify elements in $B(G,\bfmu)$ into two kinds: Hodge-Newton decomposable or indecomposable.
\begin{definition}[Hodge-Newton Decomposability]Assume $\bfb\in B(G,\bfmu)$.
	We say $\bfb$ is \textit{Hodge-Newton decomposable} (with respect to $M$) in $B(G,\bfmu)$ if there exists a $\varphi_0$-stable standard Levi subgroup $M$ containing $M_{\bfnu}$, and 
	\begin{equation}\mu^\diamond-\bfnu\in \mathbb{Q}_{\ge0}\Delta_M^\vee.
	\end{equation}
	If no such $M$ exists, $\bfb$ is said to be \textit{Hodge-Newton indecomposable} in $B(G,\bfmu)$.
\end{definition}

\begin{example}
A basic element $\bfb$ is always HN-indecomposable in $B(G,\bfmu)$ since $M_{\bfnu}=G$.
\end{example}

For a HN-decomposable $\bfb$ in $B(G,\bfmu)$, affine Deligne--Lusztig varieties admit a decomposition theorem (\Cref{GHNThmA}). 
More precisely, suppose $\bfb$ is HN-decomposable with respect to a Levi subgroup $M$. 
Let $P$ be the standard parabolic subgroup containing $M$ and $B$. 
As in \cite[4.4]{GHN19}, let $\frakP^\varphi$ be the set of $\varphi$-stable parabolic subgroups containing the maximal torus $T$ and conjugate to $P$. 
Given $P'\in\frakP^\varphi$, let $N'$ be the unipotent radical, and $M'$ the Levi subgroup containing $T$ such that $P'=M'N'$. 
We let $\calK^{M'}_p$ denote the parahoric group scheme of $M'$ such that $\calK^{M'}_p(\breve{\bbQ}_p)=K_p\cap M'(\breve{\bbQ}_p)$. 
Let $W_K$ be the subgroup of $\widetilde{W}$ generated by the set of simple reflections corresponding to $\calK_p$. 
Let $W_K^\varphi$ be the $\varphi$-invariant elements of $W_K$. 
We will consider the quotient set $\frakP^\varphi/W_K^\varphi$ as in \cite[\S 4.5]{GHN19}.
In such a setup, we let $\mu_{P'}$ denote the unique $P'$-dominant conjugate of $\mu$, and we let $b_{P'}\in M(\br{\bbQ}_p)\cap \bfb$ denote a representative such that ${\nu^{M'}_{b_{P'}}}$ is $P'$-dominant (this exists by \cite[Proposition 4.8]{GHN19}). We have the following. 
\begin{theorem}[{\cite[Theorem A]{GHN19}}]\label{GHNThmA}
    Let $\bfb\in B(G,\bfmu)$ be HN-decomposable with respect to $M\subset G$. Then there is an isomorphism of perfect schemes
    \begin{equation}\label{HN-decomposition-formula}
    X^{\calK_p}_\mu(b)\simeq \bigsqcup_{P'=M'N'} X_{\mu_{P'}}^{\calK_p^{M'}}(b_{P'}),
    \end{equation}
    where $P'$ ranges over the set $\frakP^\varphi/W_K^\varphi$.
\end{theorem}
\begin{remark}
	\label{description-of-embedding}
Note that the natural embedding
\begin{equation}
\label{embeddingparahoric}
 \phi_{P'}: X^{\calK^{M'}_p}_{\mu_{P'}}(b_{P'})\hookto X^{\calK_p}_\mu(b)   
\end{equation}
is the composite of the closed immersion $\Fl_{\calK^{M'}_p}\hookrightarrow\Fl_{\calK_p}$ of affine flag varieties and the map $g\breve{K}_p\mapsto h_{P'}g\breve{K}_p$, where $h_{P'}\in G(\br{\bbQ}_p)$ satisfies $b_{P'}=h_{P'}^{-1}b\sigma(h_{P'})$ (\cite[\S 4.5]{GHN19}). 
Moreover, the isomorphism of \Cref{GHNThmA} expresses $X^{\calK_p}_\mu(b)$ as a disjoint union of open and closed subschemes of the form $\phi_{P'}(X^{\calK^{M'}_p}_{\mu_{P'}}(b_{P'}))\subseteq X^{\calK_p}_\mu(b)$.  
\end{remark}

By the following lemma, we may assume--without loss of generality in the proof of \Cref{kappamapisoHNirrep}--that each $(b_{P'},\mu_{P'})$ is HN-indecomposable.

\begin{lemma}[{\cite[Lemma 5.7]{Zhou20}}]\label{HNdecomptoindecomp}
    There exists a unique $\varphi_0$-stable $M\subset G$ such that, for each $P'$ appearing in decomposition \eqref{HN-decomposition-formula}, $b_{P'}$ is HN-indecomposable in $B(M',\bfmu_{P'})$.
\end{lemma}

\begin{example}\label{HN-irred-example}
When $G^\ad$ is simple, $\bfb$ is basic and $\bf\mu$ is not central, then $\bfb$ is Hodge-Newton irreducible (\Cref{HN-irred-Defn}) in $B(G,\bfmu)$ because if a linear combination of coroots is dominant then all the coefficients are positive.
\end{example}
\Cref{HN-irred-example} shows that, except for the ``central cocharacter'' case, HN-indecomposability is the same as HN-irreducibility whenever $\bfb$ is basic. The general version of this phenomena is \Cref{indecompirred} below, which asserts that the gap between HN-indecomposable and HN-irreducible elements consists only of central elements. 

\begin{proposition}[{cf. \cite[Lemma 5.3]{Zhou20}}]\label{indecompirred}
	Suppose that $G=G^{\on{ad}}$ and that $G$ is $\bbQ_p$-simple. Let $b\in G(\breve{\bbQ}_p)$ and $\mu$ a dominant cocharacter, such that $\bfb\in B(G,\bfmu)$. 
	Suppose $(\bfb,\bfmu)$ is HN-indecomposable. 
	Then either $(\bfb,\bfmu)$ is HN-irreducible or $b$ is $\varphi$-conjugate to some $\dot{t}_{\bar{\mu}}$ with $\bar{\mu}\in X_*(T)_I$ central.
\end{proposition}
 Moreover, when $b$ is $\varphi$-conjugate to $\dot{t}_{\bar{\mu}}$ for a central ${\mu}$, the connected components of affine Deligne--Lusztig varieties have been computed in \Cref{centraladlv} below. 
 Note that if ${\mu}$ is central, there is a unique $\bfb\in B(G,\bfmu)$. Moreover, this $\bfb$ is basic and represented by $\dot{t}_{\bar{\mu}}$, which is a lift of $t_{\bar{\mu}}$ to $N(\br{\bbQ}_p)$. We can then apply the following result. 
\begin{proposition}[{\cite[Theorem 0.1 (1)]{He-Zhou}}]\label{centraladlv}
Suppose that $G^\ad$ is $\bbQ_p$-simple. Let $b\in {G(\br{\bbQ}_p})$ be a representative for a basic element $\bfb\in B(G)$. If ${\mu}$ is central and $\bfb\in B(G,\bfmu)$, then $X^{\calK_p}_\mu(b)$ is discrete and
\begin{equation}
X^{\calK_p}_\mu(b)\simeq G(\bbQ_p)/\calK_p(\bbZ_p).
\end{equation}
\end{proposition}

\begin{numberedparagraph}
Next we show that HN-irreducibility is preserved under ad-isomorphisms and taking projection onto direct factors. Let $f:G\to H$ be an ad-isomorphism. Let $b_H:=f(b)$ and $\mu_H=\mu\circ f$. 
Let $T_H$ denote a maximal torus containing $f(T)$. By functoriality, we have commutative diagrams
\begin{equation}
	\begin{tikzcd}\label{functorialityBgmu1}
		X_*(T) \ar{r}{f_*} \ar{d} & X_*(T_H) \ar{d} \\		
		\pi_1(G)_\Gamma \ar{r}  & \pi_1(H)_\Gamma 		
	\end{tikzcd}
\end{equation}
and 
\begin{equation}
	\begin{tikzcd}\label{functorialityBgmu2}
		B(G) \ar{r} \ar{d} & B(H) \ar{d} &		
		B(G) \ar{r} \ar{d} & B(H) \ar{d} \\		
		\pi_1(G)_\Gamma \ar{r}  & \pi_1(H)_\Gamma & X_*(T)^+_\bbQ \ar{r}  & X_*(T)^+_\bbQ
	\end{tikzcd}
\end{equation}
We have the following.
\begin{proposition}
	\label{checkafteradiso}	
	Let $\bfb\in B(G,\bfmu)$ and let $f$ be an ad-isomorphism. Then $(\bfb_H,\bfmu_H)$ is HN-irreducible if and only if $(\bfb,\bfmu)$ is HN-irreducible.
\end{proposition}
\begin{proof} 
Since $\bfb\in B(G,\bfmu)$, we have $\kappa_G(\bfb)=\mu^\natural$ (see \P\ref{notation-mu-sharp}). By \eqref{functorialityBgmu1} and \eqref{functorialityBgmu2}, we have $\kappa_H(\bfb_H)=\mu_H^\natural$.
	Moreover, we can write 
\begin{equation}
	\mu^\diamond-\bfnu=\sum_{\alpha \in \Delta} c_\alpha \alpha^\vee,
\end{equation}
where $c_\alpha\geq 0$.
On the other hand, note that $f_*(\mu^\diamond-\bfnu)=\mu_H^\diamond -\bfnuH$. Since $f$ is an ad-isomorphism, $f_*(\alpha^\vee)=\alpha^\vee$. Thus we have $\mu_H^\diamond -\bfnuH=\sum_{\alpha \in \Delta} c_\alpha \alpha^\vee$, and hence $\bfb_H\in B(H,\bfmu_H)$. 
	Now, each $(\bfb_H,\bfmu_H)$ is HN-irreducible if and only if $(\bfb,\bfmu)$ is, since this is in turn equivalent to $c_\alpha>0$ for all $\alpha\in \Delta$.
\end{proof}

\end{numberedparagraph}

\begin{numberedparagraph}
Let $G=G_1\times G_2$, then $T=T_1\times T_2$, $B(G)=B(G_1)\times B(G_2)$, $\pi_1(G)_\Gamma=\pi_1(G_1)_\Gamma\times \pi_1(G_2)_\Gamma$ and $X_*(T)=X_*(T_1)\times X_*(T_2)$. In this case, the Kottwitz and Newton maps \P\ref{newtonpointdefi} can be computed coordinatewise.
\begin{proposition}
	\label{checkafterprod}	
	The following hold:
	\begin{enumerate}
	    \item $\bfb\in B(G,\bfmu)$ if and only if each $\bfb_i\in B(G_i,\bfmu_i)$ for $i\in \{1,2\}$. 
	    \item $(\bfb,\bfmu)$ is HN-irreducible if and only if each $(\bfb_i,\mu_i)$ is HN-irreducible for $i\in \{1,2\}$.
	\end{enumerate}
\end{proposition}
\begin{proof}
The condition $\kappa_G(\bfb)=\mu^\natural$ can be checked component-wise. 
Moreover, since $\mu^\diamond-\bfnu=(\mu_1^\diamond-\bfnu_1, \mu_2^\diamond-\bfnu_2)$, verifying whether it is a non-negative (resp.~positive) linear combination of positive coroots (see \Cref{HN-irred-Defn}) can also be done component-wise.  
\end{proof}
\end{numberedparagraph}

\begin{proposition}
	\label{kappamapisoHNirrep}
	Assume that $G$ is quasisplit. If the Kottwitz map $\omega:\pi_0(X^{\calK_p}_\mu(b))\to c_{b,\mu}\pi_1(G)^\varphi_I$ is a bijection, then $(\bfb,\bfmu)$ is HN-irreducible. 
\end{proposition}
\begin{proof}
	By \Cref{checkafteradiso}, \Cref{checkafterprod}, \Cref{productsopen} and \Cref{adisoforadlv}, we may assume without loss of generality that $G$ is adjoint and $\bbQ_p$-simple.
We prove by contradiction and assume that $(\bfb,\bfmu)$ is not HN-irreducible. 

(I) If $\bfb$ is HN-decomposable in $B(G,\bfmu)$, then by \Cref{GHNThmA}, we have 
	\begin{equation}\label{bijection-parahoric-union}
		\pi_0(X^{\calK_p}_\mu(b))=\bigsqcup_{P'\in \frakP^\varphi/W_K^\varphi}\pi_0(X_{\mu_{P'}}^{M'}(b_{P'})).
	\end{equation}
	Thus by \Cref{HNdecomptoindecomp}, we may assume that each $b_{P'}$ is HN-indecomposable in $B(M',\mu_{P'})$. Recall from \eqref{embeddingparahoric} that for each $P'\in \frakP^\varphi/W_K^\varphi$ we have an embedding $\phi_{P'}:X^{M'}_{\mu_{P'}}(b_{P'})\hookto X_\mu(b)$, which induces a map
\begin{equation}\label{pi0-embeddingparahoric}
	\pi_0(\phi_{P'}): \pi_0(X^{M'}_{\mu_{P'}}(b_{P'}))\hookto \pi_0(X^{\calK_p}_\mu(b)).
\end{equation}
The disjoint union over $P'\in \frakP^\varphi/W_K^\varphi$ in \eqref{pi0-embeddingparahoric} gives the bijection \eqref{bijection-parahoric-union}. 

Consider $\iota: M'(\br{F})\to G(\br{F})$. Let $\iota_I: \pi_1(M')_I\to\pi_1(G)_I$ be the induced map, which then induces a map $\iota_I^{\varphi}: \pi_1(M')_I^{\varphi}\to\pi_1(G)_I^{\varphi}$. 
By \cite[7.4]{KottwitzII}, the following diagram commutes: 
\begin{equation}\label{Kottwitz-map-diagram}
        \begin{tikzcd}
            M'(\br{\bbQ}_p)\arrow{r}{\kappa_{M'}}\arrow{d}[swap]{\iota}&\pi_1(M')_I\arrow{d}{\iota_I}\\
            G(\br{\bbQ}_p)\arrow{r}{\kappa_G}&\pi_1(G)_I
        \end{tikzcd}
\end{equation}
Recall from \Cref{description-of-embedding} that there is an element $h_{P'}$ inducing the embedding $\phi_{P'}$.
Denote by $+_{h_{P'}}:\pi_1(G)_I\to\pi_1(G)_I$ the addition-by-$\kappa_G(h_{P'})$ map.
    Then \eqref{Kottwitz-map-diagram} shows that $+_{h_{P'}}\circ\iota_I$ sends $c_{b_{P'},\mu_{P'}}\pi_1(M')_I^\varphi$ to $c_{b,\mu}\pi_1(G)_I^\varphi$. Moreover, we have the following commutative diagram
    
\begin{equation}
        \begin{tikzcd}
           \pi_0(X^{M'}_{\mu_{P'}}(b_{P'}))\arrow[hookrightarrow]{d}[swap]{\pi_0(\phi_{P'})}\arrow[twoheadrightarrow]{r}{\omega_{M'}}&c_{b_{P'},\mu_{P'}}\pi_1(M')_I^\varphi\arrow{d}{+_{h_{P'}}\circ\iota_I}\\\pi_0(X_\mu(b))\arrow{r}{\omega_G}[swap]{\cong}&c_{b,\mu}\pi_1(G)_I^\varphi
        \end{tikzcd}
\end{equation}

    Here the surjectivity of $\omega_{M'}$ follows from \cite[Lemma 6.1]{He-Zhou}. Now, if the lower horizontal arrow $\omega_G$ is a bijection, then the upper horizontal arrow $\omega_{M'}$ should also be a bijection. Moreover, this implies that $+_{h_{P'}}\circ\iota_I$ is injective, which then implies that
    $\iota_I^\varphi:\pi_1(M')_I^\varphi\to\pi_1(G)_I^\varphi$ is injective. 
    Since we assumed that $G$ is quasisplit, we may find a $\bbQ_p$-rational parabolic subgroup $M\subseteq P\subseteq G$ having $M$ as semi-direct factor.
    In this situation $\iota_I^\varphi:\pi_1(M')_I^\varphi\to\pi_1(G)_I^\varphi$ cannot be injective, since this would contradict \Cref{Levipione}.

(II) If $\bfb$ is HN-indecomposable in $B(G,\bfmu)$, 
by \Cref{indecompirred}, we may assume that $\mu$ is central and $b=\dot{t}_{\bar{\mu}}$. We now show that $\pi_0(X^{\calK_p}_\mu(b))\to\pi_1(G)_I^\varphi$ is not bijective. 
    
By \Cref{centraladlv}, there is a bijection $\pi_0(X^{\calK_p}_\mu(b))\simeq G(\bbQ_p)/{K}_p$. Since $G$ is not anisotropic (even quasisplit), there exists a non-trivial $\bbQ_p$-split torus $S$, and we can consider the composition of maps 
\begin{equation}
S(\bbQ_p)\hookrightarrow G(\bbQ_p) \twoheadrightarrow G(\bbQ_p)/\mathcal{K}_p(\bbZ_p).
\end{equation}
Since $S(\bbQ_p)\cap \calK_p(\bbZ_p)$ is compact, we have $S(\bbQ_p)\cap \mathcal{K}_p(\bbZ_p)\subseteq S(\bbZ_p)$. Therefore, we obtain an injective homomorphism
\begin{equation} X_*(S)\cong S(\bbQ_p)/S(\bbZ_p)\hookrightarrow G(\bbQ_p)/\mathcal{K}_p(\bbZ_p).
\end{equation}
Since $G$ is adjoint, $\pi_1(G)_I^\varphi$ is finite. However, $X_*(S)$ is infinite, thus the map $\omega_G:\pi_0(X_\mu(b))\to\pi_1(G)_I^\varphi$ cannot be bijective. We have a contradiction. 
\end{proof}

Now we finish the proof of \Cref{kappamapisoHNirrep} by proving the following lemma. 
\begin{lemma}\label{Levipione}
Let $G$ be adjoint and $\bbQ_p$-simple. Let $P\subseteq G$ be a proper parabolic defined over $\bbQ_p$ with Levi factor $M$. The natural map $\iota_I^{\varphi}:\pi_1(M)_I^\varphi \to \pi_1(G)_I^\varphi$ is not injective.
\end{lemma}
\begin{proof}
	Recall that $\pi_1(G)_\Gamma$, being the group of co-invariants for $\Gamma$, can be rewritten as $(\pi_1(G)_I)_{\hat{\bbZ}}$ where we form the group of co-invariants in two steps. 
We prove by contradiction and assume that the natural map $\iota_I^{\varphi}:\pi_1(M)_I^\varphi\to \pi_1(G)_I^\varphi$ is injective. In particular, 
$\pi_1(M)_I^\varphi\otimes\bbQ \hookto \pi_1(G)_I^\varphi\otimes\bbQ$ is also injective.
Via the ``average map'' under $\varphi$-action, we have 
\begin{equation}
    \pi_1(-)_I^\varphi\otimes\bbQ\simeq (\pi_1(-)_I)_{\langle\varphi\rangle}\otimes\bbQ\simeq \pi_1(-)_\Gamma\otimes\bbQ\simeq\pi_1(-)^\Gamma\otimes\bbQ.
\end{equation} 
If $\iota_I^\varphi$ is injective, we deduce that the natural map 
\begin{equation}\label{gamma-invariant-Levi-to-G}
    \pi_1(M)^\Gamma\otimes\bbQ\to\pi_1(G)^\Gamma\otimes\bbQ
\end{equation} 
is injective. 
Let $M\subseteq P\subseteq G$ be the corresponding parabolic subgroup. Let $\theta_P=\sum \limits_{\alpha\in \Phi_P}\alpha^{\vee}\in X_*(T)$ denote the sum of coroots of $P$.
Now, $\theta_P$ is $\Gamma$-stable since $P$ is defined over ${\bbQ}_p$. 
Moreover, its image under the natural projection map $q_M:X_*(T)\to \pi_1(M)$ is $\Gamma$-stable. One can check that $q_M(\theta_P)\neq 0$ in $\pi_1(M)^\Gamma\otimes \bbQ$. 
Since $q_G(\theta_P)=0$ in $\pi_1(G)$, the map in \eqref{gamma-invariant-Levi-to-G} is not injective. We have a contradiction, this proves that $\iota_I^\varphi$ is not injective.   
\end{proof}

\section{Generic Mumford--Tate groups}\label{generic-MT-section}

\subsection{Mumford--Tate groups of crystalline representations}\label{generic-MTgp-section}
We will use the theory of \textit{crystalline representations with $G$-structures} (see for example \cite{DOR10}). 
Let $\Rep_G$ be the category of algebraic representations of $G$ in $\bbQ_p$-vector spaces.
Let $\Isoc$ be the category of isocrystals over $\bar{\bbF}_p$. 

Fix a finite extension $K$ of $\br{\bbQ}_p$. Let $\Rep^\cris_{\Gamma_K}$ be the category of crystalline representations of $\Gamma_K$ on finite-dimensional $\bbQ_p$-vector spaces. 
Let $\omega:\Rep^\cris_{\Gamma_K}\to \vs_{\bbQ_p}$ be the forgetful fibre functor. 
Let $\FilI_{K/\br{\bbQ}_p}$ be the category of filtered isocrystals whose objects are pairs of an isocrystal $N$ and a decreasing filtration of $N\otimes K$. Furthermore, let $\FilI^\ad_{K/\br{\bbQ}_p}$ be Fontaine's subcategory of weakly admissible filtered isocrystals \cite{Fontaine-wa}. 
This is a $\bbQ_p$-linear Tannakian category, which is equivalent to $\Rep^\cris_{\Gamma_K}$ through Fontaine's functor $V_\cris$ \cite{ColmFont}.

\begin{numberedparagraph}
Fix a pair $(b,\mu^\eta)$ with $b\in G(\br{\bbQ}_p)$ and $\mu^\eta:\bbG_m\to G_K$ a group homomorphism over $K$.
This defines a $\otimes$-functor 
\begin{equation}\label{funtor-Gbmueta}
\calG_{(b,\mu^\eta)}:\Rep_G\to \FilI_{K/\br{\bbQ}_p}
\end{equation} 
sending $\rho:G\to \GL(V)$ to the filtered isocrystal $(V\otimes\br{\bbQ}_p,\rho(b)\sigma, \on{Fil}^{\bullet}_{\mu^\eta}V\otimes K)$, where the filtration on $V\otimes K$ is the one induced by $\mu^\eta$. 
The pair $(b,\mu^\eta)$ is called \textit{admissible} \cite[Definition 1.18]{RZ96}, if the image of $\calG_{(b,\mu^\eta)}$ lies in $\FilI^\ad_{K/\br{\bbQ}_p}$. 
Moreover, when $\bfb\in B(G,\bfmu)$, $V_\cris\circ \calG_{(b,\mu^\eta)}$ defines a conjugacy class of crystalline representations $\xi_{(b,\mu^\eta)}:\Gamma_K\to G(\bbQ_p)$ (see \cite[Proposition 11.4.3]{DOR10} and the paragraph preceding it). 
\begin{definition}\label{MT-defn}
	With notation as above, let $\MT_{(b,\mu^\eta)}$ denote the identity component of the Zariski closure of $\xi_{(b,\mu^\eta)}(\Gamma_K)$ in $G(\bbQ_p)$. This is the \textit{Mumford--Tate group} attached to $(b,\mu^\eta)$.
\end{definition}
\begin{theorem}
	\label{remarkserresen}(\cite[Th\'eor\`eme 1]{Serre79},\cite[\S 4, Th\'eor\`eme 1]{Sen73}, \cite[Proposition 3.2.1]{Chen})
The image of $\xi_{(b,\mu^\eta)}$ contains an open subgroup of $\MT_{(b,\mu^\eta)}$. 
\end{theorem}

\end{numberedparagraph}

\begin{numberedparagraph}
As in \cite[\S 3]{Chen}, we let $\scrT^\cris_{(b,\mu^\eta)}:=\calG_{(b,\mu^\eta)}(\Rep_G)$ and $\scrT_{(b,\mu^\eta)}:=V_\cris\circ \calG_{(b,\mu^\eta)}(\Rep_G)$ be the images of $\Rep_G$.
Then $\MT_{(b,\mu^\eta)}$ is the Tannakian group for the fiber functor $\omega:\scrT_{(b,\mu^\eta)}\to \vs_{\bbQ_p}$ by \cite[Proposition 3.2.3]{Chen}.

In \cite[$\mathsection$3]{Chen}, there is a fiber functor $\omega_s:\scrT^\cris_{(b,\mu^\eta)}\to \vs_{\bbQ_{p^s}}$ for $s$ sufficiently large\footnote{Note that our notation $\omega_s$ differs from the notations \textit{loc.cit.}, where the notation $\omega_{b,\mu}^{\cris,s}$ is used instead.}, with Tannakian group $\MT^{\cris,s}_{(b,\mu^\eta)}:=\Aut^{\otimes}\omega_s$ as in \cite[D\'{e}finition 3.3.1]{Chen}.
When $b$ is $s$-decent (see \Cref{decent-definition}), there is a canonical embedding $\MT^{\cris,s}_{(b,\mu^\eta)}\subseteq G_{\bbQ_{p^s}}$ \cite[Lemme 3.3.2]{Chen}. Moreover,  $\MT^{\cris,s}_{(b,\mu^\eta)}$ and $\MT_{(b,\mu^\eta)}\otimes_{\Q_p}\Q_{p^s}$ are pure inner forms of each other \cite[Proposition 3.3.3]{Chen}. Both claims follow immediately using Tannakian formalism. 
In particular, to prove that $\MT_{(b,\mu^\eta)}$ contains $G^\der$, it suffices to prove that $\MT^{\cris,s}_{(b,\mu^\eta)}$ contains $G^\der_{\bbQ_{p^s}}$ (since $G^{\der}$ is normal). 
\end{numberedparagraph}

\begin{numberedparagraph}\label{concrete-MT}
In fact, there is a more concrete description of $\MT^{\cris,s}_{(b,\mu^\eta)}$ given as follows. 

Let $(V,\rho)\in \Rep_G$. The $\mu^\eta$-filtration of $V_K$ induces a degree function 
\begin{equation}
\dg_{\mu^\eta}:V\setminus \{0\}\to \bbZ,
\end{equation}
where $\dg_{\mu^\eta}(v)=i$ if $v\in \Fil_{\mu^\eta}^i V\setminus \Fil_{\mu^\eta}^{i+1}V$. 
We shall consider a subset $V^{s,k}_{(b,\mu^\eta)}\subseteq V\otimes {\bbQ_{p^s}}$ of elements that satisfy a certain ``Newton equation" \eqref{Newton-eqn} and a certain ``Hodge equation" \eqref{Hodge-eqn} with respect to $k$ for some $k\in \bbZ$.

Let $T^{s\cdot \nu_b}_{\rho}:V\otimes \bbQ_{p^s}\to V\otimes \bbQ_{p^s}$ be the operator with formula 
\begin{equation}\label{Newton-eqn}
T^{s\cdot \nu_b}_{\rho}:=\rho\circ [s\cdot \nu_b](p).
\end{equation}
Consider also the function $d^s_{\rho,\mu^\eta}:V\otimes \bbQ_{p^s}\setminus \{0\}\to \bbZ$ where
\begin{equation}\label{Hodge-eqn}
d^s_{\rho,\mu^\eta}(v)=\sum\limits_{i=0}^{s-1}\dg_{\mu^\eta}([\rho(b)\varphi]^i(v)).
\end{equation}
We consider the following subset of $V\otimes\bbQ_{p^s}$ given by
\begin{center}
\begin{align}
	\label{MTdescription}
	V^{s,k}_{(b,{\mu^\eta})}:= \{v\in V\otimes \bbQ_{p^s} \mid   T^{s\cdot \nu_b}_{\rho}(v)=p^kv, \,  d^s_{\rho,{\mu^\eta}}(v)= k \}   	
\end{align}
\end{center}
By \cite[Proposition 3.3.6]{Chen},  $\MT^{\cris,s}_{(b,{\mu^\eta})}$ consists of those elements $g\in G_{\bbQ_{p^s}}$ such that: for any $(V,\rho)\in \Rep_G$ and any $k\in \bbZ$, all of the elements $v\in V^{s,k}_{(b,{\mu^\eta})}$ are eigenvectors of $\rho(g)$. 
In particular, to prove $G_{\bbQ_{p^s}}^\der\subseteq \MT^{\cris,s}_{(b,{\mu^\eta})}$, it suffices to prove that $G_{\bbQ_{p^s}}^\der$ acts trivially on $V^{s,k}_{(b,{\mu^\eta})}$ for all $V$ and $k$. 
\end{numberedparagraph}

\subsection{Generic filtrations}
In this subsection and \S \ref{sect:mumford-tate} we give a representation theoretic formula for $d^s_{\rho,{\mu^\eta}}$ when ${\mu^\eta}$ is generic (see \Cref{equality-2-2}). 
We closely follows \cite[\S 4]{Chen}, except that we do small modifications to the argument \textit{loc.cit.} to deal with general reductive groups $G$. In particular, our version of the argument does not assume that $G$ is unramified or even quasisplit. 

\begin{numberedparagraph}
Before we move on to our goal, we recall some generalities, which we will apply later to $G_{\bbQ_{p^s}}$ for $s$-sufficiently large such that $G_{\bbQ_{p^s}}$ is quasisplit. 
Until further notice, $K$ will denote an arbitrary field of characteristic $0$, $G$ a quasisplit reductive group over $K$, and $\bfmu$ a conjugacy class of group homomorphisms $\mu:\bbG_m\to G_{\bar{K}}$. 
Let $E/K$ be the reflex field of $\bfmu$. Since $G$ is quasisplit, we can choose a representative $\mu\in \bfmu$ defined over $E$ such that it is dominant for a choice of $K$-rational Borel $B\subseteq G$. 
To this data, we can associate a flag variety $\on{Fl}_\mu:=G_E/P_\mu$ over $\Spec (E)$ as in \eqref{Bia-Bir-map}. 
It parametrizes filtrations of $\Rep_G$ of type $\bfmu$. Given a field extension $K'/K$, $x\in \on{Fl}_\mu(K')$ and $(V,\rho)\in \Rep_G$, we obtain a filtration $\Fil^\bullet_x V_{K'}$ as in \cite[D\'efinition 4.1.1]{Chen}.
\begin{definition}[{\cite[D\'efinition 4.2.1]{Chen}}]
	\label{defini-of-filt}
	With the setup as above, let 
	
\begin{align}\label{genfil-VE}
	&\overline{\Fil}_\mu^\bullet V_E:=\left(\bigcap \limits_{x\in \on{Fl}_\mu(E)}\Fil^\bullet_x V_E\right)\\
	&\overline{\Fil}_\bfmu^\bullet V:=V\cap \left(\bigcap \limits_{x\in \on{Fl}_\mu(E)}\Fil^\bullet_x V_E\right).\label{genfil-bfmu}
\end{align}
	We refer to \eqref{genfil-VE} (resp.~\eqref{genfil-bfmu}) as the \textit{generic filtration} of $V_E$ (resp.~$V$) attached to $\mu$ (resp.~$\bfmu$).
\end{definition}

\end{numberedparagraph}

\begin{numberedparagraph}
	Each step of the filtration, $\overline{\Fil}_\bfmu^i V\subseteq V$, is a subrepresentation of $V$. 
Indeed, since $G(E)$ acts transitively on $\on{Fl}_\mu(E)$ we obtain the identity
\begin{equation}\label{alternative-genfil-bfmu}
\overline{\Fil}_\bfmu^\bullet V=V\cap\left(\bigcap \limits_{g\in G(E)}\rho(g)\Fil^\bullet_\mu V_E\right).
\end{equation}
This filtration $\overline{\Fil}_\bfmu^i V$ gives rise to a degree function $\overline{\dg}_\bfmu:V\setminus \{0\}\to \bbZ$ which can be computed as:  
\begin{equation}
	\overline{\dg}_\bfmu(v)=\on{inf}_{g\in G(E)} \dg_\mu(\rho(g)\cdot v).
\end{equation}
 
Moreover we have the following

\begin{proposition}
	\label{generic-agree}
	For all $\tau\in \on{Gal}(E/K)$, the generic filtration attached to $\bfmu$ and $\tau(\bfmu)$ agree, i.e. we have 
	$\overline{\Fil}_\bfmu^\bullet V=\overline{\Fil}_{\tau(\bfmu)}^\bullet V.$
\end{proposition}
\begin{proof}
	We compute
	\begin{align}
		\overline{\Fil}_{\tau(\mu)}^\bullet V_E & =  \left(\bigcap \limits_{g\in G(E)}\rho(g)\Fil^\bullet_{\tau(\mu)} V_E\right)  \\
							&  =  \left(\bigcap \limits_{g\in G(E)}\tau[ \rho(\tau^{-1}(g)) \Fil^\bullet_{\mu} V_E]\right)  \\ 
							&  =  \tau[ \left(\bigcap \limits_{g\in G(E)}\rho(\tau^{-1}(g)) \Fil^\bullet_{\mu} V_E\right)  ]  \\
							&  =  \tau[  \overline{\Fil}_{\mu}^\bullet V_E ]  
	\end{align}
	In particular, $v\in \overline{\Fil}_\bfmu^\bullet V=V\cap \overline{\Fil}_\mu^\bullet V_E$, if and only if $\tau(v)\in \overline{\Fil}_{\tau(\bfmu)}^\bullet V$. 
	Since for such $v$, $v=\tau(v)$, we must have $\overline{\Fil}_{\bfmu}^\bullet V= \overline{\Fil}_{\tau(\bfmu)}^\bullet V$.
\end{proof}

Let $K'$ be an arbitrary extension of $K$. 
\begin{definition}
\label{generic-definition}
We say that a map $\on{Spec}(K')\to \on{Fl}_\mu$ is \textit{generic} if, at the level of topological spaces $|\Spec(K')|\to |\on{Fl}_\mu|$, the image of the unique point on the left is the generic point of $\on{Fl}_\mu$. 
\end{definition}
\end{numberedparagraph}
The following statement relates $\overline{\Fil}_\bfmu^\bullet V$ (see \eqref{genfil-bfmu} or \eqref{alternative-genfil-bfmu}) to the generic points of $\on{Fl}_\mu$ in the sense of \Cref{generic-definition}. 
\begin{proposition}[{\cite[Lemme 4.2.2]{Chen}}]
	\label{genericfilisgenericpoint} 
	
	Let $\mu^\eta:\Spec (K')\to \on{Fl}_\mu$ be generic (in the sense of \Cref{generic-definition}). 
	Then for all $i\in \bbZ$, we have 
\begin{equation}
\overline{\Fil}_\bfmu^i V=V\cap \Fil^i_{\mu^\eta} V_{K'},
\end{equation}
where the inclusion $V\subseteq V\otimes_K E \subseteq V\otimes_K K'$ is the natural one.
\end{proposition}
\begin{proof} 
The following proof is in \cite[4.2.2]{Chen}. 
We recall the argument for the convenience of the reader. 
Note that we do not assume $G$ split over $K$, which is the running assumptions in \textit{loc.cit}. 

Let $\widetilde{\calY}_{\mu}$ be the universal $P_\mu$-bundle over  $\on{Fl}_\mu=G_E/P_{\mu}$ coming from the natural map to $[\ast/P_{\mu}]$. Consider the vector bundle $\mathcal{E}:=\widetilde{\calY}_{\mu}\times_{P_{\mu,\rho}}V$, with a filtration $$\dots \supseteq \Fil^0\mathcal{E}\supseteq \Fil^1\mathcal{E}\supseteq \cdots\supseteq \Fil^n\mathcal{E}\supseteq \dots$$ of locally free locally direct factors of the form $\widetilde{\calY}_{\mu}\times_{P_{\mu},\rho}\Fil^{\bullet}V$, where $\Fil^{\bullet}V$ is the natural filtration of $V$ by subrepresentations of $P_{\mu}$. 

We may regard elements $v\in V$ as global sections of $\mathcal{E}$, and we have that $$v\in\Fil_x^iV\Leftrightarrow v\in\ker\left(\Gamma(\on{Fl}_\mu,\mathcal{E}/\Fil^i\mathcal{E})\to  \Gamma(\Spec \kappa(x),\mathcal{E}/\Fil^i\mathcal{E})\right).$$ 

The vanishing locus of such an element is a Zariski closed subset, and it contains the generic point if and only if it contains all the $E$-rational points. Thus $\overline{\Fil}^i_{\bfmu} V=V\cap \Fil^i_{\mu^\eta}V_{K'}$. 
\end{proof}

\begin{numberedparagraph}
We need a more easily computable description of $\overline{\Fil}_\bfmu^\bullet V$.
In \cite[Proposition 4.3.2]{Chen}, there is such a description assuming that $G$ is split over $K$. We now prove a generalization in the quasisplit case.

Let $\Gamma_K$ denote the Galois group of $K$. We fix $K$-rational tori $S\subseteq T\subseteq B\subseteq G$ where $S$ is maximally split and $T$ is the centralizer of $S$. 
Recall that, by combining the theory of highest weights and Galois theory, one can classify all irreducible representations of a quasisplit group by the Galois orbits $\calO\subseteq X^*(T)^+$ of dominant weights. 
More precisely, attached to $\calO\subseteq X^*(T)^+$ there is a unique irreducible representation of $G$, which we denote by $W_\calO$, such that 
\[W_\calO\otimes_K \bar{K}=\oplus_{\lambda\in \calO} W_\lambda\]
where $W_\lambda$ is the unique irreducible representation of $G_{\bar{K}}$ that has highest weight $\lambda$.
Given $\lambda \in X^*(T)$, we let $\calO_\lambda:=\Gamma_K\cdot \lambda$ denote its Galois orbit. 
We let $X^*(T)/\Gamma_K$ denote the set of Galois orbits and $X^*(T)^+/\Gamma_K$ denote the subset of dominant ones.
We define $\calO^E_\lambda:=\Gamma_E\cdot \lambda$, $X^*(T)/\Gamma_E$ and $X^*(T)^+/\Gamma_E$ analogously.
Given a dominant $\Gamma_K$-Galois orbit $\calO\in X^*(T)^+/\Gamma_K$ (resp.~$\Gamma_E$-Galois orbit  $\calO^E\in X^*(T)^+/\Gamma_E$), let $V_\calO$ (resp.~$V_{\calO^E}$) denote the $W_\calO$-isotypic (resp.~$W_{\calO^E}$-isotypic) direct summand of $V$ (resp.~$V_E$).
We have 
\begin{align}
	V_\calO\otimes_K \bar{K}&=\bigoplus \limits_{\lambda \in \calO} V^\lambda_{\bar{K}}.\\
	V_{\calO^E}\otimes_E \bar{K}&=\bigoplus \limits_{\lambda \in \calO^E} V^\lambda_{\bar{K}},
\end{align}
where $V^\lambda_{\bar{K}}$ denotes the isotypic component attached to $W_\lambda$ (the unique irreducible representation of $G_{\bar{K}}$ with highest weight $\lambda$).
Observe that by Galois descent $V_\calO=(\bigoplus \limits_{\lambda \in \calO} V^\lambda_{\bar{K}})^{\Gamma_K}$, this will be used in the proof of \Cref{genericfiltration} below.
For $\calO\in X^*(T)/\Gamma_K$ a Galois orbit, we let $\underline{\calO}\in (X^*(T)_\bbQ)^{\Gamma_K}$ be given by $\underline{\calO}:=\frac{1}{|\calO|}\sum\limits_{\lambda\in \calO} \lambda$.
If $\calO_\lambda=\Gamma_K\cdot \lambda$ and $\Gamma_\lambda$ denotes the stabilizer of $\lambda$ in $\Gamma_K$ we have
\begin{equation}
\underline{\calO_\lambda}=\frac{1}{[\Gamma_K:\Gamma_\lambda]}\sum\limits_{\gamma\in \Gamma_K/\Gamma_\lambda} \gamma(\lambda).
\end{equation}
Analogously, we have 
$ \underline{\calO^E_\lambda}\in X^*(T)_\bbQ^+)^{\Gamma_E}$ with 
\[\underline{\calO^E_\lambda}=\frac{1}{[\Gamma_E:\Gamma_\lambda\cap\Gamma_E]}\sum\limits_{\gamma\in \Gamma_E/\Gamma_\lambda\cap \Gamma_E} \gamma(\lambda).\]
Let $\mathcal{W}$ denote the absolute Weyl group of $G$. Let $w_0\in \mathcal{W}$ be the longest element, which is $\Gamma_K$-invariant.

\begin{proposition}
	\label{genericfiltration}
	Let the setup be as above. For any $(V,\rho)\in \Rep_G$, the generic filtration attached to $\mu$ is given by the formula:
	\begin{center}
		\begin{align}
			\overline{\Fil}^k_\bfmu V =\bigoplus\limits_{\substack{\calO_\lambda\in X^*(T)^+/\Gamma_K\\ \text{such that } \langle\underline{\calO^E_{\tau(\lambda)}},w_0.\mu\rangle\geq k, \\
			\forall \tau\in \on{Gal}(E/K)}} V_{\calO_{\lambda}}
		\end{align}
	\end{center}
\end{proposition}
\begin{proof}
	Since $\overline{\Fil}^k_\bfmu V$ consists of subrepresentations, it suffices to show that 
\begin{equation}
	V_{\calO_\lambda}\subseteq \overline{\Fil}^k_\bfmu V\Longleftrightarrow k\leq \langle \underline{\calO^E_{\tau(\lambda)}}, w_0.  \mu\rangle  \,\,\,\, \forall \tau\in \on{Gal}(E/K).
\end{equation}
	
Let us first prove ``$\Longrightarrow$''. 
Let $V=\bigoplus\limits_{\sigma\in \Irrep(T)}V_{\sigma}$ be the decomposition into isotypic pieces with respect to the $T$-action. 
Observe that $\Irrep(T)$ is naturally in bijection with $\Gamma_K$-orbits in $X^*(T)$.
Over an algebraic closure, each $V_{\sigma}$ decomposes as $V_{\sigma}\otimes_K \bar{K}=\bigoplus\limits_{\chi'\in\calO_{\sigma}}V_{\chi'}$, where $\calO_\sigma\subseteq X^*(T)$ denotes the $\Gamma_K$-orbit attached to $\sigma$. 

Observe that $\overline{\Fil}^k_\bfmu V\subseteq \Fil^k_{\tau(\mu)} V_E$ for all $\tau\in \on{Gal}(E/K)$. 
Indeed, this follows from \Cref{generic-agree}.
Moreover, we can express $\Fil^k_{\tau(\mu)} V_E$ in terms of the grading that $\tau(\mu)$ induces on $V_E$ as follows  
\begin{equation}
	\Fil^k_{\tau(\mu)} V_{\bar{K}}=\bigoplus\limits_{\substack{\chi\in X^*(T),\\ \langle{{\chi}},\tau(\mu)\rangle\geq k  }}V_{\bar{K},{\chi}}.
\end{equation}
In particular, when $V_{\calO_\lambda}\subseteq \overline{\Fil}^k_\bfmu V\subseteq \Fil^k_{\tau(\mu)} V_{\bar{K}}$, the anti-dominant weights appearing in $V_{\calO_\lambda}$ pair with $\tau(\mu)$ to a number greater than or equal to $k$. 
In other words, $k\leq \langle w_0.\xi, \tau(\mu)\rangle$ for all $w_0.\xi \in \calO_{w_0.\lambda}$, but then pairing $\tau(\mu)$ with their $\Gamma_E$-average $w_0.\underline{\calO^E_{\lambda}}$ will still be greater than or equal to $k$, i.e.~
\[k\leq \langle w_0.\underline{\calO^{E}_\lambda}, \tau(\mu)\rangle=\langle \underline{\calO^E_{\tau(\lambda)}}, w_0.\mu\rangle.\] 
The above argument does not depend on $\tau$, thus
\begin{equation}
	\overline{\Fil}_{\bfmu}^k V\subset\bigoplus\limits_{\substack{ \lambda\in X^*(T)^+/\Gamma_K\\ \text{such that } \langle \underline{\calO^E_{\tau(\lambda)}},w_0.\mu\rangle\geq k \\ \forall \tau\in \on{Gal}(E/K)}} V_{\calO_{\lambda}}.
\end{equation}

Let us now prove ``$\Longleftarrow$''. 
Suppose $k\leq \langle w_0.\underline{\calO^E_\lambda}, \tau(\mu)\rangle$ for all $\tau\in \on{Gal}(E/K)$, this implies that for at least one $w_0.\xi\in \calO^E_{w_0.\lambda}$, we have $k\leq \langle w_0.\xi, \tau(\mu)\rangle$. 
Since $\langle \cdot, \cdot \rangle$ is $\Gamma_E$-equivariant, $k\leq \langle w_0.\xi, \tau(\mu)\rangle$ for all $w_0.\xi\in \calO^E_{w_0.\lambda}$.
We can view $w_0.\xi$ as a character of $T_{\bar{K}}$ and $V_{\bar{K},w_0.\xi}\subseteq \Fil^k_{\tau(\mu)}V_{\bar{K}}$. 
Consider $V_{\bar{K}}^{\xi}$, the $W_\xi$-isotypic part of $V_{\bar{K}}$, where again $W_\xi$ is the irreducible representation of $G_{\bar{K}}$ of highest weight $\xi$. 
If $\chi$ is a weight appearing in $V_{\bar{K}}^\xi$, then $k\leq \langle w_0.\xi, \tau(\mu)\rangle\leq  \langle \chi, \tau(\mu)\rangle$, since $\tau(\mu)$ is dominant and $\xi$ is the highest weight of $W_\xi$. 
Thus $V_{\bar{K}}^\xi \subseteq \Fil^k_{\tau(\mu)}V_{\bar{K}}$. 
Since $\Fil^k_{\tau(\mu)}V_{\bar{K}}$ is $\Gamma_E$-stable, the same follows for every $\xi\in \calO_\lambda$.
We deduce that 
\begin{equation}
	V_{\calO^E_{\lambda}}:=\bigg(\,\bigoplus\limits_{\xi\in \calO^E_{\lambda}} V_{\bar{K}}^\xi\bigg)^{\Gamma_E}
\end{equation}
is a subrepresentation of $\Fil^k_{\tau(\mu)}V_E$ defined over $E$. 
By \Cref{alternative-genfil-bfmu}, this implies that 
\[V_{\calO^E_\lambda}\subseteq \bigcap _{g\in G(E)} \rho(g)\cdot  {\Fil}^k_{\tau(\mu)}V_E = \overline{\Fil}^k_{\tau(\bfmu)}V_E,\]
and since the above argument does not depend on $\tau$, this holds for all $\tau\in \on{Gal}(E/K)$.
Equivalently, $\tau^{-1}(V_{\calO^E_\lambda})\subseteq \overline{\Fil}^k_{\bfmu}V_E$ for all $\tau\in \on{Gal}(E/K)$.
Recall that
\[V_{\calO_\lambda}\otimes_K E:=\Sigma_{\tau\in \on{Gal}(E/K)} \tau(V_{\calO^E_\lambda})\subseteq V_E,\] 
by the above $V_{\calO_\lambda}\otimes_K E\subseteq \overline{\Fil}^k_{\bfmu}V_E$. 
This shows that 
\[V_{\calO_\lambda}=V\cap V_{\calO_\lambda}\otimes_K E\subseteq V\cap \overline{\Fil}^k_{\bfmu}V_E = \overline{\Fil}^k_{\bfmu}V\]
as desired.
\end{proof}
\end{numberedparagraph}

\subsection{Mumford--Tate group computations}
\label{sect:mumford-tate}
The goal of this section is to prove \Cref{enhancedChen-maintxt} (or \Cref{enhancedchen} in the introduction).

Let $G$ be a reductive group over $\bbQ_p$. Let $K'$ be a finite extension of $\br{\bbQ}_p$. 
Let $b\in G(\br{\bbQ}_p)$ be decent (\Cref{decent-definition}) and $\mu^\eta:\bbG_m\to G_{K'}$ be generic (\Cref{generic-definition}) with $\mu^\eta\in \bfmu$. 

\begin{theorem}\label{enhancedChen-maintxt}
	Suppose that $b$ is decent,  $\mu^\eta$ is generic and $\bfb\in B(G,\bfmu)$.
	The following hold: 
	\begin{enumerate}
		\item $(b,\mu^\eta)$ is admissible.
		\item If $(\bfb,\bfmu)$ is HN-irreducible, then $\on{MT}_{(b,\mu^\eta)}$ contains $G^\der$. 
	\end{enumerate}
\end{theorem}
\begin{proof}

We fix $s$ large enough so that $b$ is $s$-decent, $G$ is quasisplit over $\bbQ_{p^s}$ and splits over a totally ramified extension of $\bbQ_{p^s}$ that we denote by $L$. 
After fixing a Borel $B\subseteq G_{\bbQ_{p^s}}$, we let $\mu\in X^*(T)^+$ denote the dominant representative of $\bfmu$ and we let $\bbQ_{p^s}\subseteq E\subseteq L$ denote the field of definition of $\bfmu$.
Recall that replacing $b$ by $g^{-1}b\varphi(g)$ and $\mu^\eta$ by $g^{-1}\mu^\eta g$ gives isomorphic fiber functors $\calG_{(b,\mu^\eta)}$ (see \eqref{funtor-Gbmueta}). 
Moreover, via this kind of replacement, we can arrange that $\nu_b=\bfnu$ (see \P\ref{decencydominant}). 
Note that this replacement preserves genericity of $\mu^\eta$.

(1) The argument in \cite[Th\'eor\`eme 5.0.6.(1)]{Chen} goes through in our setting. Indeed, the only part in the proof \textit{loc.cit.} using that $G$ is unramified is to justify that $\on{Fl}_\mu^{\on{ad}}\neq \emptyset$ whenever $\bfb\in B(G,\bfmu)$, but this is true by \cite[Theorem 9.5.10]{DOR10} in full generality.  

(2) We adapt the argument given for \cite[Th\'eor\`eme 5.0.6.(2)]{Chen} to our setup. 
Let $(V,\rho)\in \Rep_G$ and let $v\in V^{s,k}_{(b,\mu^\eta)}$ as in \eqref{MTdescription}. 
By \P\ref{concrete-MT}, it suffices to show that $\rho(g)v=v$ for all $g\in G^\der_{\bbQ_{p^s}}$. 
Recall that $L$ denotes the splitting field of $G$.
Over $L$, we can write $v=\sum\limits_{\lambda\in \Lambda_v} v_\lambda$ where $\Lambda_v \subseteq X^*(T)^+$, $v_\lambda\in V_L^\lambda$ and $v_\lambda\neq 0$ (here $V_L^\lambda$ denotes the $W_\lambda$-isotypic part of $V_L$ where $W_\lambda$ is the highest weight representation of $G_L$ attached to $\lambda$).  
Since $v$ is defined over $\bbQ_{p^s}$, we have $\gamma(v_\lambda)=v_{\gamma(\lambda)}$ for $\gamma\in \Gamma_{\bbQ_{p^s}}$.

Consider $\calO\subseteq X^*(T)^+$ a Galois orbit for $\Gamma_{\bbQ_{p^s}}$ giving rise to $V_\calO \in \Irrep_{G_{\bbQ_{p^s}}}$ that appears in the isotypic decomposition of $V_{\bbQ_{p^s}}$. 
We let $v_\calO=\sum\limits_{\lambda\in \calO\subseteq \Lambda_v}v_\lambda$.
We have $v_\calO\in V_{\bbQ_{p^s}}$.
Fix $\lambda\in \Lambda_v$, by \Cref{genericfiltration} and \Cref{genericfilisgenericpoint}, we can write
	\begin{align}\label{first-eqn}
		\dg_{\mu^\eta}(v) & =  \overline{\dg}_\bfmu(v) \\
			   & =  \Inf_{\lambda\in \Lambda_v} \overline{\dg}_\bfmu(v_{\calO_\lambda})\label{second-eqn} \\
			   & \leq  \overline{\dg}_\bfmu(v_{\calO_\lambda})\label{3rd-eqn} \\
			   & =  \Inf_{\tau \in \on{Gal}(E/\bbQ_{p^s})} \langle \underline{\calO^E_{\tau(\lambda)}},w_0\cdot \mu \rangle \label{4th-eqn} \\
			   & \leq   \langle \underline{\calO_\lambda},w_0\cdot \mu \rangle \label{5th-eqn}
			   \\
			   & =  \langle w_0\cdot\lambda, \underline{\mu} \rangle, \label{last-equality}
	\end{align}
Here \eqref{first-eqn} follows from \Cref{genericfilisgenericpoint}. Since each step of $\overline{\on{Fil}}^\bullet_{\bfmu} V$ is a subrepresentation of $V$, in order for $v\in \overline{\on{Fil}}^k_{\bfmu} V$, each $v_{\calO_\lambda}$ has to be in $\overline{\on{Fil}}^\bullet_{\bfmu} V$, and hence \eqref{second-eqn}. Inequality \eqref{3rd-eqn} follows from the definition of infimum and our assumption $\lambda\in \Lambda_v$. \eqref{4th-eqn} follows from \Cref{genericfiltration}, and the fact that 
\begin{equation}
v_{\calO_\lambda}=\sum \limits_{\tau\in \on{Gal}(E/K)} v_{\calO^E_{\tau(\lambda)}}.
\end{equation}
Since the average is larger than the infimum, \eqref{5th-eqn} follows. Finally, \eqref{last-equality} follows from equivariance of the pairing $\langle \cdot, \cdot \rangle$ with respect to the $\Gamma_K$-action, and invariance of the pairing under the $w_0$-action. 

Observe that $\varphi\in \on{Gal}(\bbQ_{p^s}/\bbQ_p)$ acts on the set of orbits $X^*(T)/\Gamma_{\bbQ_{p^s}}$. 
Moreover, if $\calO_\lambda\subseteq X^*(T)^+$ is a $\Gamma_{\bbQ_{p^s}}$-orbit corresponding to an irreducible representation $W_\calO$ of $G_{\bbQ_{p^s}}$, then the $\varphi$-twist of $W_\calO$ corresponds to $\varphi_0(\calO)\subseteq X^*(T)^+$, with the definition of $\varphi_0$ as in \P\ref{phi0action}.
We deduce that $\varphi(v_\calO)=\varphi(v)_{\varphi_0(\calO)}$.
Write $v^i=(\rho(b)\varphi)^i[v]$, it follows that $v^i_{\varphi^i_0(\calO)}=(\rho(b)\varphi)^i[v_{\calO}]$. 
In particular, $\Lambda_{v^i}=\varphi^i_0(\Lambda_v)$.

We have the following formula
	\begin{align}
	d^s_{\rho,\mu^\eta}(v) & =\sum\limits_{i=0}^{s-1}\dg_{\mu^\eta}((\rho(b)\varphi)^iv)\label{equality-2-1} \\
	& = \sum\limits_{i=0}^{s-1} \Inf_{\lambda\in \Lambda_{v^i}} \overline{\dg}_\bfmu(v^i_{\calO_\lambda})\label{equality-2-2} \\
	& \leq \sum\limits_{i=0}^{s-1} \langle \underline{\calO_{\varphi_0^i(\lambda)}},w_0\cdot \mu \rangle \label{inequality-2-3}\\
	& = \sum\limits_{i=0}^{s-1} \langle {w_0\cdot {\varphi_0^i(\lambda)}},\underline{\mu} \rangle \label{equality-2-4}\\
	& = \sum\limits_{i=0}^{s-1} \langle w_0\cdot\lambda, \varphi^i_0(\underline{\mu}) \rangle \label{equality-2-5}
	\\
	& =s\cdot \langle w_0\cdot\lambda, \mu^\diamond \rangle \label{equality-2-7}
	\end{align}

Equality \eqref{equality-2-1} follows from the definition in \eqref{Hodge-eqn}. 
By \Cref{genericfiltration}, we obtain \eqref{equality-2-2}. 
Inequality \eqref{inequality-2-3} follows from the inequalites \eqref{first-eqn} through \eqref{last-equality} above, and $\lambda\in \Lambda_v \implies \varphi_0^i(\lambda)\in \Lambda_{v^i}$. Equality
\eqref{equality-2-4} follows from equivariance of $\langle , \rangle$ under the Galois action and $w_0$-action.
Since $T$ is $\varphi_0$-stable,
\eqref{equality-2-5} follows from equivariance of $\langle, \rangle$ under the $\varphi_0$-action.
Equality \eqref{equality-2-7} follows from the definition of $\mu^\diamond$ (see \eqref{defn-mu-diamond}).

Since $v\in V^{s,k}_{(b,\mu^\eta)}$, by \eqref{equality-2-1} through \eqref{equality-2-7}, we have 
\begin{equation}
	\label{first-key-ineq}
\frac{k}{s}\leq \langle w_0\cdot\lambda, \mu^\diamond \rangle
\end{equation}
for all $\lambda\in \Lambda_v$. 
On the other hand, over $L$ (the splitting field of $G$), we have a decomposition $v=\sum\limits_{\chi\in X^*(T)} u_\chi$ in terms of weight spaces with respect to the $T$-action. 
Since we have arranged that $\nu_b=\bfnu$, by \eqref{Newton-eqn} and \eqref{MTdescription} we have

	\begin{equation}
		T^{s\cdot \nu_b}_{\rho}(v)  =\sum\limits_{\chi\in X^*(T)} T^{s\cdot \nu_b}_{\rho}(u_\chi)  =\sum\limits_{\chi\in X^*(T)}p^{\langle \chi,s\cdot \bfnu\rangle}u_\chi.
	\end{equation}

The assumption $v\in V^{s,k}_{(b,\mu^\eta)}$ forces $\chi$ to satisfy $\langle \chi,s\cdot \bfnu\rangle=k$ for all $\chi$ where $u_\chi\neq 0$. 
In particular, since $w_0\cdot \lambda\leq \chi$ when $\chi$ is a weight for $T$ in $V^\lambda_L$, we deduce that 
\begin{equation}
	\label{second-key-eq}
\langle w_0\cdot\lambda, \bfnu\rangle \leq \frac{k}{s}
\end{equation}
for all $\lambda\in \Lambda_v$. 
It follows \Cref{second-key-eq} and \Cref{first-key-ineq} that 
\[\langle w_0\cdot \lambda, \mu^\diamond-\bfnu\rangle \geq 0.\] 
Moreover, since $(\bfb,\bfmu)$ is HN-irreducible, we may write $\mu^\diamond-\bfnu=\Sigma_{\alpha\in \Delta} c_\alpha \alpha^\vee$ with $c_\alpha> 0$. 
Since $w_0\cdot \lambda$ is anti-dominant, this can only happen if $\langle w_0\cdot \lambda, \alpha^\vee\rangle = 0$ for all $\alpha\in \Delta$. 
This shows that $\lambda$ is orthogonal to the coroot lattice and can only be the highest weight of a $1$-dimensional representation of $G_L$.
This shows that the action of $G^\der_{L}$ on $V^\lambda$ is trivial for all $\lambda\in \Lambda_v$. 
In particular, $G^\der_{L}$ acts trivially on $v$ as we wanted to show.
This is the end of the proof.
\end{proof}

\begin{proposition}
	\label{cor:MainChen}
	Let $(G,b,\mu)$ be a local shtuka datum over $\bbQ_p$ with $(\bfb,\bfmu)$ HN-irreducible. There exists a finite extension $K$ over $\br{\bbQ}_p$ containing the reflex field of $\bfmu$, and a point $x\in \Gr^b_\mu(K)$ whose induced (conjugacy class of) crystalline representation(s) 
	$$\rho_x:\Gamma_K\to G(\bbQ_p)$$
	satisfies that $\rho_x(\Gamma_K)\cap G^\der(\bbQ_p)$ is open in $G^\der(\bbQ_p)$.
\end{proposition}
\begin{proof}
	Recall from \cite[Proposition 2.12]{Gle22a} (see also \cite[Theorem 5.2]{viehmannBdr}) that the Bialynicki-Birula map BB in \eqref{Bia-Bir-map} induces a bijection of classical points. Therefore it suffices to construct the image $\on{BB}(x)\in \on{Fl}_{\mu}$, which corresponds to constructing a weakly admissible filtered isocrystal with $G$-structure.

	By \Cref{existence-generic-pt}, we can take $\on{BB}(x)=\mu^{\eta}$ to be generic (\Cref{generic-definition}). 
	By \Cref{enhancedChen-maintxt}(2), $\MT_{(b,\mu^{\eta})}$ contains $G^{\der}$. By \Cref{remarkserresen}, the image of the generic crystalline representation $\xi_{(b,\mu^{\eta})}$ contains an open subgroup of $\MT_{(b,\mu^{\eta})}$, thus containing an open subgroup of $G^{\der}$. 
\end{proof}
\begin{lemma}\label{existence-generic-pt}
There exist a finite extension $K$ over $\br{\bbQ}_p$ and a map $\mu^\eta:\on{Spec}(K)\to \on{Fl}_\mu$ such that $|\mu^\eta|:\{\ast\}\to |\on{Fl}_\mu|$ maps to the generic point.
\end{lemma}
\begin{proof}
Recall from \cite[Proposition 2.0.3]{Chen} that the transcendence degree of $\br{\bbQ}_p$ over $\bbQ_p$ is infinite. 
By the structure theorem of smooth morphisms \cite[Tag 054L]{Stacks}, one can find an open neighborhood $U\to \on{Fl}_\mu$ that is \'etale over $\bbA^n_{\bbQ_p}$. On the other hand, one can always find a map $\Spec(\br{\bbQ}_p)\to \bbA^n_{\bbQ_p}$ mapping to the generic point by choosing $n$ transcendentally independent elements of $\br{\bbQ}_p$. Its pullback to $U$ is an \'etale neighborhood of $\Spec(\br{\bbQ}_p)$ that consists of a finite disjoint union of finite extensions $K$ of $\br{\bbQ}_p$. Any of these components will give a map to the generic point of $\on{Fl}_\mu$.
\end{proof}

The following is a partial converse to \Cref{enhancedChen-maintxt}, and it follows directly from \Cref{secondthmmain}.
\begin{proposition}
\label{partial-converse-Chensection}
\label{chenconverse}
		Assume that $G$ is quasisplit. If $\on{MT}_{(b,\mu^\eta)}$ contains $G^\der$, then $(\bfb,\bfmu)$ is HN-irreducible. 
\end{proposition}
\begin{proof}
	If $\Gder\subseteq \on{MT}_{(b,\mu^\eta)}$, then by \Cref{remarkserresen} there exists a finite field extension $K$ over $\br{\bbQ}_p$, and a crystalline representation $\xi:\Gamma_K\to G(\bbQ_p)$ with invariants $(\bfb,\bfmu)$ whose image in $G^{\der}(\bbQ_p)$ is open. Indeed, we can let $K$ be generic as in \Cref{generic-definition}. The result follows from the equivalence $(3)\iff (4)$ in \Cref{secondthmmain}.
\end{proof}

\section{Proof of main theorems}\label{proof-main-thm-section}
The first goal in this section is to prove the following main theorem:
\begin{theorem}
		\label{secondthmmain} 
		Suppose that $\bfb\in B(G,\bfmu)$.   
		\begin{enumerate}[label=\alph*)]
			\item If we assume that $G$ is quasisplit, then the following statements are equivalent: 
				\begin{enumerate}[label=(\arabic*)]
			\item The map $\omega_G:\pi_0(X_\mu(b))\to c_{b,\mu}\pi_1(G)_I^\varphi$ is bijective.  
			\item The map $\omega_G:\pi_0(X^{\calK_p}_\mu(b))\to c_{b,\mu}\pi_1(G)_I^\varphi$ is bijective.  
			\item The pair $(\bfb,\bfmu)$ is HN-irreducible.
			\item There exists a field extension $K$ of finite index over $\br{\bbQ}_p$, and a crystalline representation $\xi:\Gamma_K\to G(\bbQ_p)$ with invariants $(\bfb,\bfmu)$ for which $G^{\der}(\bbQ_p)\cap \xi(\Gamma_K)\subseteq G^{\der}(\bbQ_p)$ is open. 
			\item The action of $G(\bbQ_p)$ on $\Sht_{(G,b,\mu,\infty)}$ makes $\pi_0(\Sht_{(G,b,\mu,\infty)}\times \Spd \bbC_p)$ into a $G^\circ$-torsor.
\end{enumerate}
\item If we assume that $G^\ad$ does not have anisotropic factors, then the implications 
	\[(3)\implies (4)\implies (5) \implies (1) \implies (2)\]
	still hold.
\item If we assume that $G$ is arbitrary, then $(5)\implies (1) \implies (2)$ and $(3)\implies (4)$ still hold.
\end{enumerate}
\end{theorem}
	The second goal in this section is to prove the following corollary of \Cref{secondthmmain}. 

	\begin{theorem}
		\label{mainthmmain}
		Let $G$ be arbitrary. Suppose that $(\bfb,\bfmu)$ is HN-irreducible, then the Kottwitz map
		$\omega_G:\pi_0(X^{\calK_p}_\mu(b))\to c_{b,\mu}\pi_1(G)_I^\varphi$ is bijective. 
	\end{theorem}

	The proof of the above main theorems 
	proceeds as follows and will occupy the rest of section \ref{proof-main-thm-section}. We first prove a modified version of the statement in the case of tori (see $\mathsection$\ref{tori-case-section}, \Cref{torisecondmain}, \Cref{lemmagabwhatever}). We then use z-extensions and ad-isomorphisms to reduce the proof of \Cref{secondthmmain} and \Cref{mainthmmain} to the case where $\Gder=\Gsc$ (see \Cref{reductionsofsecondthm}). 
	We prove the circle of implications of \Cref{secondthmmain} in this case. Then, we deduce \Cref{mainthmmain} from \Cref{secondthmmain} whenever $G$ has no anisotropic factors. Finally, we deduce \Cref{mainthmmain} in the anisotropic case.
	
Before we dive into the proof of \Cref{mainthmmain}, we deduce \Cref{Pappas-Rapoport-uniformization-corollary} below.
We choose our notation to resemble that of \cite{pappas2021padic}.
Let $(p,\mathbf{G},X,\mathbf{K})$ be a tuple of global Hodge type \cite[\S 1.3]{pappas2021padic}, let $\scrS_{\bfK}$ denote the integral model of \cite[Theorem 1.3.2]{pappas2021padic}, let $k$ be an algebraically closed field in characteristic $p$ and let $x_0\in \scrS_{\bfK}(k)$.
Let $G=\mathbf{G}_{\bbQ_p}$, suppose that $\mathbf{K}=K_p\bf{K}^p$ where $K_p$ denotes the level at $p$ and $\bf{K}^p$ denotes the level away from $p$.
Suppose that $\calG$ is a parahoric group scheme with $\calG(\bbZ_p)=K_p$.
Pappas and Rapoport consider a map of v-sheaves $c:\on{RZ}_{\calG,\mu,x_0}^\diamondsuit\to \calM^{\on{int}}_{\calG,b,\mu}$ \cite[Lemma 4.1.0.2]{pappas2021padic}, where the source is a Rapoport--Zink space and the target is another name for $\Sht^\calG_\mu(b)$ i.e. $\calM^{\on{int}}_{\calG,b,\mu}=\Sht^\calG_\mu(b)$.
Let the notations be as in \cite[Theorem 4.10.6, \S 4.10.2]{pappas2021padic}. 
Fix $x_0\in \scrS_{\bfK}(k)$, and recall that by construction $\on{RZ}_{\calG,\mu,x_0}$ is a formal scheme equipped with a uniformization map 
\[\Theta^{\on{RZ}}_{\calG,x_0}:\on{RZ}_{\calG,\mu,x_0}\to \hat{\scrS}_{\bfK}\]
of formal schemes over $\on{Spf} O_{\breve{E}}$ (see \cite[Equation 4.10.3]{pappas2021padic}).
Using the action of $\bfG(\bbA^f_p)$, we obtain a morphism 
\begin{equation}
	\label{unifo-equation}
	\Theta^{\on{RZ}}_{\calG,x_0}:\on{RZ}_{\calG,\mu,x_0}\times \bfG(\bbA_f^p)\to \hat{\scrS}_{\bfK}
\end{equation}
The image under this map $\calI(x_0)$ after taking $k$-points is called the isogeny class of $x_0$. 

Let us recall the algebraic group $I_{x_0}\subseteq \on{Aut}_\bbQ(\calA_{x_0})$.
This is analogous to the situation considered in the introduction (see \Cref{isog-inject}) with $\bar{x}$ replaced by $x_0$.
Attached to ${x_0}\in \scrS_{\bfK}(k)$ we obtain an abelian variety $\calA_{{x_0}}$ with associated $p$-divisible group $\calG_{{x_0}}$. 
These are equipped with tensors $s_{\alpha, \ell, {x_0}}\in \on{H}^1_{\acute{e}t}(\calA_{{x_0}}, \bbQ_\ell)^\otimes$ for all prime $\ell\neq p$ and with crystalline tensors on the rational Dieudonn\'e module $s_{\alpha,\on{crys},{x_0}}\in \bbD(\calG_{{x_0}})^\otimes$ (see \cite[\S 4.10.3]{pappas2021padic}).\footnote{We denote these crystalline tensors as \cite{pappas2021padic} do. This is in contrast with the introduction, where we follow the notation used in \cite{Zhou20}.}
	Recall that associated to such an abelian variety we can consider an algebraic group $\on{Aut}_\bbQ(\calA_{{x_0}})$, defined over $\bbQ$, parametrizing automorphisms of $\calA_{{x_0}}$ treated as an abelian variety up-to-isogeny.
	This has as $R$-points $\on{Aut}_\bbQ(\calA_{{x_0}})(R)=(\on{End}(\calA_{{x_0}})\otimes_{\bbZ} R)^\times$. 
	We consider the algebraic subgroup $I_{{x_0}}\subseteq \on{Aut}_{\bbQ}(\calA_{{x_0}})$ of those automorphisms of the abelian variety $\calA_{{x_0}}$ up-to-isogeny that preserve the tensors $s_{\alpha,\ell,{x_0}}$ and $s_{\alpha,\on{crys},{x_0}}$.

\begin{corollary}\label{Pappas-Rapoport-uniformization-corollary}
	The map $c:\on{RZ}_{\calG,\mu,x_0}^\diamondsuit\to \calM^{\on{int}}_{\calG,b,\mu}$ is an isomorphism. Thus, $\calM^{\on{int}}_{\calG,b,\mu}$ is representable by a formal scheme $\scrM_{\calG,b,\mu}$,
	and we obtain a $p$-adic uniformization isomorphism of $O_{\br{E}}$-formal schemes 
	\begin{equation}
		I_{x_0}(\bbQ)\backslash (\scrM_{\calG,b,\mu}\times \bfG(\bbA^p_f)/\bfK^p) \to (\widehat{\scrS_{\bfK}\otimes_{O_E}O_{\br{E}}})_{/\calI(x_0)}.
	\end{equation}
\end{corollary}

\begin{proof}
	Throughout this argument, we let the notations be as in \cite[\S 4.8]{pappas2021padic}.  
	In particular, we have fixed an appropriate Hodge embedding $\bfG\hookrightarrow \on{GSp}(V,\psi)$ into a symplectic similitude group, inducing an embedding $G\hookrightarrow H$ with $H=\on{GSp}(V_{\bbQ_p},\psi_{\bbQ_p})$.
	Further, we are assuming that this embedding comes from a dilated immersion of parahoric group schemes $\calG\to \calH$ (see \cite[\S 4.5.2]{pappas2021padic}).
	Our chosen Hodge embedding produces a map 
	\[\iota:\scrS_{\bfK}\to \calA_{\bfK^\flat}\otimes_{\bbZ_{(p)}}O_E.\]
		where $\calA_{\bfK^\flat}$ denotes a Siegel moduli scheme for an appropriately chosen level satisfying $\bfK^\flat\cap G(\bbA_f)=\bfK$ (see \cite[\S 4.5.2]{pappas2021padic}).

	By the proof of \cite[Theorem 4.10.6]{pappas2021padic}, it suffices to verify condition ($\mathrm{U}_{x_0}$) in \cite[\S 4.10.2]{pappas2021padic} which is the running assumption in the reference. 
	This condition poses the existence of a commutative diagram of perfect schemes
	\begin{center}
	\begin{tikzcd}
		X^{\calG}_\mu(b_{x_0})	\arrow{r} \arrow[dotted]{d}{\exists}  & X^{\calH}_\mu(b_{\iota(x_0)})  \arrow{d} \\
		\scrS^{perf}_{\mathbf{K},\bar{\bbF}_p} \arrow{r} & \calA^{perf}_{\mathbf{K}^\flat,\bar{\bbF}_p},
	\end{tikzcd}
	\end{center}
	where the solid arrows have been fixed (we refer the reader to \cite{pappas2021padic} for the precise definition of the solid arrows) and condition $\mathrm{U}_{x_0}$ asks for the existence of the dotted arrow.
	As explained in \cite[\S 4.10.2]{pappas2021padic}, the condition $\mathrm{U}_{x_0}$ follows from showing that the map of v-sheaves
	\[c:\on{RZ}_{\calG,\mu,x_0}^\diamondsuit\to \calM^{\on{int}}_{\calG,b,\mu}\]
	is an isomorphism.
	Further, as explained in \textit{loc.cit.}, this is also equivalent to showing that the map of sets
	\[c:\on{RZ}_{\calG,\mu,x_0}^\diamondsuit(\bar{\bbF}_p)\to \calM^{\on{int}}_{\calG,b,\mu}(\bar{\bbF}_p)\]
	is surjective, and this is what we will show.

By \cite[Proposition 4.10.3 and Lemma 4.10.2.b)]{pappas2021padic}, $c:\on{RZ}_{\calG,\mu,x_0}^\diamondsuit\subseteq \calM^{\on{int}}_{\calG,b,\mu}$ is an open and closed immersion, and it suffices to check that for every $x\in \pi_0(\calM^{\on{int}}_{\calG,b,\mu})$ there is $y\in \pi_0(\on{RZ}_{\calG,\mu,x_0}^\diamondsuit)$ with $c(y)=x$. 

Let $\tilde{x}_0\in {\scrS_{\bfK}}(\breve{E}_{\tilde{x}_0})$ be a $\breve{E}_{\tilde{x}_0}$-valued point of ${\scrS_{\bfK}}$ specializing to $x_0$ with $[\breve{E}_{\tilde{x}_0}:\breve{E}]<\infty$. 
Such a point exists by flatness of ${\scrS_{\bfK}}$ over $\bbZ_{(p)}$. 
By Serre--Tate theory, $\tilde{x}_0$ induces a canonical point in $\tilde{y}_0\in \on{RZ}_{\calH,\iota(x_0)}(O_{\breve{E}_{\tilde{x}_0}})$, which overall induces a point $\tilde{z}_0\in \on{RZ}_{\calG,\mu,x_0}^\diamondsuit(\breve{E}_{\tilde{x}_0})$. 
Indeed, by definition (see \cite[\S 4.10.1]{pappas2021padic}) $\on{RZ}_{\calG,\mu,x_0}$ is an open and closed formal subscheme of $S$ where $S$ fits in the following Cartesian diagram of formal schemes
\begin{center}
\begin{tikzcd}
				 S \arrow{r} \arrow{d}  & \on{RZ}_{\calH,\mu,\iota(x_0)} \arrow{d} \\
				 \widehat{\scrS}_{\mathbf{K}} \arrow{r} & \widehat{\calA_{\mathbf{K}^\flat}\otimes_{\bbZ_{(p)}} O_E}.
\end{tikzcd}
\end{center}
It is clear that $\tilde{x}_0$ and $\tilde{y}_0$ induce a unique point $\tilde{z}_0\in S(O_{\breve{E}_{\tilde{x}_0}})$ that specializes to $x_0$, and, since $x_0\in \on{RZ}_{\calG,\mu,x_0}$, $\tilde{z}_0\in \on{RZ}_{\calG,\mu,x_0}(O_{\breve{E}_{\tilde{x}_0}})$.
Passing to diamonds, we get a point in $\on{RZ}^\diamondsuit_{\calG,\mu,x_0}(O_{\breve{E}_{\tilde{x}_0}})$ which we may restrict to a point in $\on{RZ}^\diamondsuit_{\calG,\mu,x_0}({\breve{E}_{\tilde{x}_0}})$. By abuse of notation we will still call this $\tilde{z_0}$.

Recall that to any element of $g\cdot \calG(\bbZ_p)\in G(\bbQ_p)/\calG(\bbZ_p)$ we may attach a point in $g\cdot \tilde{x}_0\in {\scrS_{\bfK}}(\bbC_p)$ by acting through at-$p$ $G$-isogenies. Analogously, for every $h\cdot \calH(\bbZ_p)\in \calH(\bbQ_p)/\calH(\bbZ_p)$ we get an element $h\cdot \tilde{y}_0\in \on{RZ}_{\calH,\iota(x_0)}(\bbC_p)$, and we get a commutative diagram:
\begin{center}
\begin{tikzcd}
	G(\bbQ_p)/\calG(\bbZ_p) \arrow{r}{(-)\cdot \tilde{z}_0} \arrow[bend left=20]{rr}{(-)\cdot \tilde{x}_0} \ar{d} & \on{RZ}_{\calG,\mu,x_0}(\bbC_p) \ar{r} \ar{d} \arrow{d} &   \scrS_{\bfK}(\bbC_p) \arrow{d} \\
	H(\bbQ_p)/\calH(\bbZ_p) \arrow{r}{(-)\cdot \tilde{y}_0} & \on{RZ}_{\calH,\iota(x_0)}(\bbC_p) \ar{r} & \calA_{\bfK^\flat}(\bbC_p) .
\end{tikzcd}
\end{center}
Moreover, we get a further compatibility
\begin{center}
\begin{tikzcd}
	G(\bbQ_p)/\calG(\bbZ_p)\arrow{r}{\on{GM}_{\tilde{z}_0}} \arrow{d}{(-)\cdot \tilde{z}_0}  & \calM^{\on{int}}_{\calG,b,\mu}  \\
 \on{RZ}_{\calG,\mu,x_0}^\diamondsuit \arrow{ru}{c} & 
\end{tikzcd}
\end{center}
where $\on{GM}_{\tilde{z}_0}$ is the map $G(\bbQ_p)/\calG(\bbZ_p)\to \calM^{\on{int}}_{\calG,b,\mu}(\bbC_p)=\Sht_{(G,b,\mu,\calG(\bbZ_p))}(\bbC_p)$ induced from choosing an identification of the fibers of the Grothendieck--Messing period morphism $G(\bbQ_p)=\pi_{GM}^{-1}(\pi_{GM}(\tilde{z}_0))\subseteq \Sht_{(G,b,\mu,\infty)}(\bbC_p)$ (see \S \ref{grothendieck-section}).
It suffices to prove that $\on{GM}_{\tilde{z}_0}:G(\bbQ_p)/\calG(\bbZ_p)\to \pi_0(\calM^{\on{int}}_{\calG,b,\mu})$ is surjective, but this is precisely the content of \Cref{kisins-surjectivity-argument}.
\end{proof}

\subsection{The tori case}\label{tori-case-section}
When $G=T$ is a torus, there is only one parahoric model that we denote by $\calT$. 
The tori analogue of \Cref{secondthmmain} is as follows. 
 \begin{proposition}
		\label{torisecondmain} 
		  Suppose that $\bfb\in B(T,\bfmu)$. The following hold: 
		\begin{enumerate}
			\item The map $\omega_T:\pi_0(X^\calT_\mu(b))\to c_{b,\mu}\pi_1(T)_I^\varphi$ is bijective.  
			\item The action of $T(\bbQ_p)$ on $\Sht_{(T,b,\mu,\infty)}$ makes $\pi_0(\Sht_{(T,b,\mu,\infty)}\times \Spd \bbC_p)$ into a $T^\circ$-torsor.
\end{enumerate}
 \end{proposition}
In this case, both $X^{\calT}_\mu(b)$ and $\Sht_{(T,b,\mu,\calT)}\times \Spd \bbC_p$ are zero-dimensional. Since we are working over algebraically closed fields, they are of the form $\coprod\limits_J \Spec\bar{\bbF}_p$ and $\coprod\limits_I \Spd\bbC_p$ for some index sets $I$ and $J$, respectively.
Moreover, by \Cref{adlvtosht}, the specialization map  \eqref{specializationmap} induces a bijection $\pi_0(\on{sp}):I\cong J$. 
Also,  $T^\circ=T(\bbQ_p)$ and $\Sht_{(T,b,\mu,\infty)}\times \Spd \bbC_p$ is a $\underline{T(\bbQ_p)}$-torsor over $\Spd \bbC_p$ (see for example \cite[Lemma 3.10]{Gle22a}).
In particular, $\pi_0(\Sht_{(T,b,\mu,\infty)}\times \Spd \bbC_p)$ is a $T(\bbQ_p)$-torsor and \Cref{torisecondmain} (2) holds. 
The content of \Cref{torisecondmain} (1) becomes the following lemma. 
\begin{lemma}
\label{lemmagabwhatever}
Let $T$ be a torus. We have a $T(\bbQ_p)$-equivariant commutative diagram, where the horizontal arrows are isomorphisms: 
\begin{center}
	\begin{tikzcd}
		\pi_0(\Sht_{(T,b,\mu,\calT)}\times \Spd \bbC_p )\arrow[r, "\cong"]&	\pi_0(X^{\calT}_\mu(b)) \ar[r, "\cong"] \ar[d] & c_{b,\mu}\pi_1(T)^\varphi_I \ar[d]\\
	&	\pi_0(\Fl_{\calT}) \arrow[r, "\cong"]  & \pi_1(T)_I\\
	\end{tikzcd}
\end{center}
\end{lemma}

\begin{proof}

Upon fixing an element of $\Sht_{(T,b,\mu,\infty)}\times \Spd \bbC_p$, we can identify $\pi_0(\Sht_{(T,b,\mu,\infty)}\times \Spd \bbC_p)\cong T(\bbQ_p)$ (see for example \cite[Lemma 3.10]{Gle22a}), which then gives an identification 
\[\Sht_{(T,b,\mu,\calT)}\times \Spd \bbC_p\cong T(\bbQ_p)/\calT(\bbZ_p)\cong T(\br{\bbQ}_p)^{\varphi=\on{id}}/\calT(\br{\bbZ}_p)^{\varphi=\on{id}}.\] 
Since ${H}^1_{\acute{e}t}(\Spec \bbZ_p,\calT)$ vanishes, we can write 
\begin{equation}
    T(\br{\bbQ}_p)^{\varphi=\on{id}}/\calT(\br{\bbZ}_p)^{\varphi=\on{id}}\cong (T(\br{\bbQ}_p)/\calT(\br{\bbZ}_p))^{\varphi=\on{id}},
\end{equation}
where the right-hand side is $X_*(T)^\varphi_I=\pi_1(T)_I^\varphi$.  
Therefore the $T(\bbQ_p)$-action makes $\pi_0(X^{\calT}_\mu(b))$ and $\pi_0(\Sht_{(T,b,\mu,\calT)}\times \Spd \bbC_p)$ into $\pi_1(T)^\varphi_I$-torsors (via the specialization map \eqref{specializationmap}). Thus by equivariance of $\pi_1(T)_I^{\varphi}$-action,  $\pi_0(\Sht_{(T,b,\mu,\calT)}\times \Spd \bbC_p )$ and $\pi_0(X^{\calT}_\mu(b))$ can be identified with a unique coset $c_{b,\mu}\pi_1(T)^\varphi_I\subseteq \pi_1(T)_I$ (by the definition of $c_{b,\mu}$). 
\end{proof}

\subsection{Reduction to the $G^{\der}=G^{\on{sc}}$ case}
 For the rest of this subsection, assume that $f:G\to H$ is an ad-isomorphism. 
Let $b_H:=f(b)$ and $\mu_H:=f\circ \mu$. Let $\calK^H_p$ denote the unique parahoric of $H$ that corresponds to the same point in the Bruhat--Tits building as $\calK_p$.

\begin{proposition}
	\label{pro:centralpreservecool}
	
(1) We have a canonical identification of 
diamonds
\begin{equation}\label{ad-iso-sht}
	\Sht_{(H,b_H,\mu_H,\infty)}\cong\Sht_{(G,b,\mu,\infty)}{\times^{\underline{G(\bbQ_p)}}} \underline{H(\bbQ_p)}.
\end{equation}
(2) In particular, $\pi_0(\Sht_{(G,b,\mu,\infty)}\times \Spd \bbC_p)$ is a $G^\circ$-torsor, if and only if
\begin{equation}
\pi_0(\Sht_{(H,b_H,\mu_H,\infty)}\times \Spd \bbC_p)
\end{equation}
is a $H^\circ$-torsor.
\end{proposition}
\begin{proof}
	(1) A version of \eqref{ad-iso-sht} was proven in \cite[Proposition 4.15]{Gle21a}, where the result is phrased in terms of the torsor $\bbL_b$ from $\mathsection$\ref{grothendieck-section}.
	\footnote{Although \cite[Proposition 4.15]{Gle21a} only considers unramified groups $G$ (since this was the ongoing assumption in \textit{loc.cit.}), the proof goes through without this assumption. 
	A more detailed proof of \Cref{pro:centralpreservecool} (1) can also be found in \cite[Proposition 5.2.1]{PapRap22}, which was obtained independently 
	as \textit{loc.cit.}}
We sketch the proof for the reader's convenience:

\textit{Step 1.}  $\Gr_\mu=\Gr_{\mu_H}$: there is an obvious proper map $\Gr_\mu\to \Gr_{\mu_H}$ of spatial diamonds. 
Therefore, to prove that it is an isomorphism, it suffices to prove bijectivity on points, which can be done as in the classical Grassmannian case (see \cite[Proposition 4.16]{AGLR22} for a stronger statement).

\textit{Step 2.}  $\Gr_\mu^b=\Gr_{\mu_H}^{b_H}$: the $b$-admissible and $b_H$-admissible loci are open subsets of $\Gr_\mu=\Gr_{\mu_H}$. To prove that they agree, we can prove it on geometric points. 
	This ultimately boils down to the fact that an element $e\in B(G)$ is basic if and only if $f(e)\in B(H)$ is basic, which holds because centrality of the Newton point $\nu_e$ can be checked after applying an ad-isomorphism.

\textit{Step 3.}  $\Sht_{(H,b_H,\mu_H,\infty)}\cong\Sht_{(G,b,\mu,\infty)}{\times^{\underline{G(\bbQ_p)}}} \underline{H(\bbQ_p)}$: recall from $\mathsection$\ref{grothendieck-section} that the Grothendieck--Messing period map \eqref{Grothendieckmessing} realizes $\Sht_{(G,b,\mu,\infty)}$ (respectively $\Sht_{(H,b_H,\mu_H,\infty)}$) as a $\underline{G(\bbQ_p)}$-torsor (respectively an $\underline{H(\bbQ_p)}$-torsor) over $\Gr_\mu^b=\Gr_{\mu_H}^{b_H}$. 
Since the $\underline{G(\bbQ_p)}$-equivariant map 
$\Sht_{(G,b,\mu,\infty)}\to \Sht_{(H,b_H,\mu_H,\infty)}$
extends to a map of $\underline{H(\bbQ_p)}$-torsors
	\begin{equation}
	\Sht_{(G,b,\mu,\infty)}\times^{\underline{G(\bbQ_p)}} \underline{H(\bbQ_p)} \to \Sht_{(H,b_H,\mu_H,\infty)},
	\end{equation}
and any map of torsors is an isomorphism, the conclusion follows.

(2) Recall that since $G\to H$ is an ad-isomorphism, we have an isomorphism $\Gsc\to H^{\on{sc}}$. 
	Recall $G^\circ:=G(\bbQ_p)/\on{Im}(\Gsc(\bbQ_p))$ and $H^\circ:=H(\bbQ_p)/\on{Im}(\Gsc(\bbQ_p))$. 
	By \eqref{ad-iso-sht} combined with \Cref{general-lemma-profinite-qt}, we have a canonical isomorphism
	\begin{equation}\label{pi0-ad-iso-sht}
	\pi_0(\Sht_{(H,b_H,\mu_H,\infty)}\times \Spd \bbC_p)\cong \pi_0(\Sht_{(G,b,\mu,\infty)}\times \Spd \bbC_p){\times^{{G(\bbQ_p)}}} {H(\bbQ_p)}.
	\end{equation}
	The right-hand side of \eqref{pi0-ad-iso-sht} is by definition 
	\begin{equation}\label{diagonal-quotient}
	    \left(\pi_0(\Sht_{(G,b,\mu,\infty)}\times \Spd \bbC_p){\times} {H(\bbQ_p)}\right)/G(\bbQ_p),
	\end{equation}
	where the quotient is via the diagonal action.
	Since $\Gsc(\bbQ_p)$ acts trivially on $\pi_0(\Sht_{(G,b,\mu,\infty)}\times \Spd \bbC_p)$, quotienting \eqref{diagonal-quotient} by $\Gsc(\bbQ_p)$ first gives
	\begin{align}
	    &\pi_0(\Sht_{(G,b,\mu,\infty)}\times \Spd \bbC_p){\times^{{G(\bbQ_p)}}} {H(\bbQ_p)}\\
	    &\cong \left(\pi_0(\Sht_{(G,b,\mu,\infty)}\times \Spd \bbC_p){\times} ({H(\bbQ_p)}/\on{Im}\Gsc(\bbQ_p))\right)/G^\circ,
	\end{align}
which simplifies, via \eqref{pi0-ad-iso-sht} and since $\Gsc(\bbQ_p)=H^{\mathrm{sc}}(\bbQ_p)$, to	
\begin{equation}\label{Hcirc-torsor}
	\pi_0(\Sht_{(H,b_H,\mu_H,\infty)}\times \Spd \bbC_p)=\pi_0(\Sht_{(G,b,\mu,\infty)}\times \Spd \bbC_p)\times^{G^\circ} {H^\circ}.
	\end{equation}
	From this expression, it is clear that $\pi_0(\Sht_{(G,b,\mu,\infty)} \times \Spd \bbC_p)$ is a $G^\circ$-torsor if and only if $\pi_0(\Sht_{(H,b_H,\mu_H,\infty)}\times \Spd \bbC_p)$ is an $H^\circ$-torsor.
\end{proof}

\begin{proposition}
	\label{reductionsofsecondthm}
	If \Cref{secondthmmain} holds for $\Gder=\Gsc$, then it holds in general as well. 
\end{proposition}
\begin{proof}
	Consider an arbitrary z-extension $\tilde{G}\to G$ (see \Cref{z-extensions}). 
	By \Cref{lifting-b-and-mu} (1), we may choose a conjugacy class of cocharacters $\tilde{\bfmu}$ and an element $\tilde{\bfb}\in B(\tilde{G},\tilde{\bfmu})$ that map to $\bfmu$ and $\bfb$, respectively, under the map $B(\tilde{G},\tilde{\bfmu})\to B(G,\bfmu)$.

	For each item $i\in\{1,\dots,5\}$, we show that $(i)$ holds for $\tilde{G}$ if and only if $(i)$ holds for $G$. 
	Since by definition of z-extensions, $\tilde{G}^\der=\tilde{G}^{\on{sc}}$, whenever we obtain an implication of the form $(i)\implies (i+1)$ or $(5)\implies (1)$ in the $\Gder=\Gsc$ case, we automatically deduce the same implication in the general case.
	Indeed, $\tilde{G}$ is quasisplit if and only if $G$ is, so the hypothesis \Cref{secondthmmain}.(a) are preserved. 
	Analogously, $\tilde{G}^{\on{ad}}=G^{\on{ad}}$ so they have the same $\bbQ_p$-simple factors and the hypothesis in \Cref{secondthmmain}.(b) are preserved.

	We first justify $(1)_{\tilde{G}}\iff (1)_{{G}}$ and $(2)_{\tilde{G}}\iff (2)_{{G}}$ of \Cref{secondthmmain}.  
	Recall that by \Cref{lifting-b-and-mu}.(2),  $c_{\tilde{b},\tilde{\mu}}\pi_1(\tilde{G})^\varphi_I \to c_{b,\mu}\pi_1(G)^\varphi_I$ is surjective. 
	We apply \Cref{adisoforadlv} to the ad-isomorphism $\tilde{G}\to G$. 
	The top horizontal arrow in \eqref{diagram-adisoforadlv} is a bijection (of sets) if and only if the bottom horizontal arrow in \eqref{diagram-adisoforadlv} is also a bijection of sets. 
	Now, $(3)_{\tilde{G}}\iff (3)_{{G}}$ of \Cref{secondthmmain} is a direct consequence of \Cref{checkafteradiso}.

	For $(4)_{\tilde{G}}\iff (4)_{{G}}$ recall that the map $\tilde{G}^\der\to G^\der$ is surjective with finite kernel. 
	In particular, $\xi(\Gamma_K)\subseteq \tilde{G}^\der$ is open if and only if its image in $G^\der$ is open.
	Further, after replacing $K$ by a finite extension $L$ containing $\tilde{E}$, the field of definition of $\tilde{\bfmu}$, every crystalline representation $\Gamma_K\to G(\bbQ_p)$ of type $(\bfb,\bfmu)$ can be lifted to a crystalline representation $\Gamma_L\to \tilde{G}(\bbQ_p)$ of type $(\tilde{\bfb},\tilde{\bfmu})$.
	Indeed, these crystalline representations are obtained from $K$-points in $\Gr_\mu^b$ (see \Cref{cor:MainChen}), and we have an identification (see proof of \Cref{pro:centralpreservecool}) 
	\[\Gr_\mu^b\times_{E} \Spd \tilde{E}=\Gr_{\tilde{\mu}}^{\tilde{b}}.\]
	Finally $(5)_{\tilde{G}}\iff (5)_{{G}}$ follows from \Cref{pro:centralpreservecool}.(2). 
\end{proof}
\subsection{Argument for $(1)\implies (2)$}
We start by giving a new proof to \cite[Theorem 7.1]{He}. 
\begin{theorem}[He]
	\label{thmHe}
	The map $X^\calI_\mu(b)\to X^{\calK_p}_\mu(b)$ is surjective. 
\end{theorem}
\begin{proof}
By functoriality of the specialization map \cite[Proposition 4.14]{Specializ} applied to 
$\Sht^{\calI}_\mu(b)\to \Sht^{\calK}_\mu(b)$ from \eqref{map-integral-Sht}, we get a commutative diagram:
\begin{center}
	\begin{tikzcd}
		\mid\Sht_{(G,b,\mu,{\calI(\bbZ_p)})}\mid \arrow{r} \arrow{d}[swap]{\on{sp}} & \mid \Sht_{(G,b,\mu,\calK_p)}\mid  \arrow{d}{\on{sp}}\\
		\mid X^\calI_\mu(b)\mid  \arrow{r}  & \mid X^{\calK_p}_\mu(b)\mid  
	\end{tikzcd}
\end{center}
The top arrow is given by \eqref{changeoflevelstructureshtukas}. 
By \cite[Proposition 23.3.1]{Ber}, the map $\Sht_{(G,b,\mu,{\calI(\bbZ_p)})}\to\Sht_{(G,b,\mu,\calK_p)}$ is finite \'etale and the fibers are isomorphic to $\calK_p/{\calI(\bbZ_p)}$.
In particular, the map of topological spaces is surjective. 
It then suffices to prove that the specialization map is surjective, which follows directly from \cite[Theorem 2 b)]{Gle22a}. 
\end{proof}
Now, \Cref{thmHe} implies the $(1)\implies (2)$ part of \Cref{secondthmmain}: 
by \Cref{functoriality-Xmub}, we have the following commutative diagram:
\begin{equation}
	\begin{tikzcd}
		\pi_0(X^\calI_\mu(b)) \arrow{r} \arrow{d} & \pi_0(X^{\calK_p}_\mu(b)) \arrow{d} \\
		\pi_0(\Fl_\brevI) \arrow{r}{\cong}  & \pi_0(\Fl_{\breve{\calK}_p})  & 
	\end{tikzcd}
\end{equation}
For the bijection of the lower horizontal arrow, see for example \cite[Lemma 4.17]{AGLR22}. 
The left downward arrow is injective by assumption (1), and the top arrow is surjective by \Cref{thmHe}. Thus the right downward arrow is also injective.
\subsection{Argument for $(2) \overset{\on{q.split}}{\implies} (3)$}
This is the content of \Cref{kappamapisoHNirrep}.
\subsection{Argument for $(3) \implies (4)$}
This is the content of \Cref{cor:MainChen}.

\subsection{Argument for $(5)\implies (1)$}\label{secondthm-4-to-1}

\begin{proposition}
	\label{5-implies-1}
   $(5)\implies (1)$ in \Cref{secondthmmain}. 
\end{proposition}
\begin{proof}
Recall that by \Cref{reductionsofsecondthm}, we may and do assume that $\Gder=\Gsc$. 
Consider the map ${\on{det}}:G\to \Gab$ where $\Gab=G/\Gder$. 
Let $\calI^\der$ denote the Iwahori subgroup of $G^{\on{der}}$ attached to our alcove $\bf{a}$ (see $\mathsection$\ref{preliminaries-section}). Let $\calG^\ab$ be the unique parahoric group scheme of $\Gab$. 
By \cite[proof of Proposition 3]{HR08} combined with \Cref{maps-of-points-vs-groups}, we have an exact sequence of parahoric group schemes:
\begin{equation}
e\to \calI^\der \to \calI \to \calG^\ab	\to e,
\end{equation}
which induces maps $\Sht^\calI_\mu(b)\to \Sht^{\calG^\ab}_{\mu^\ab}(b^\ab)$ and $X_\mu(b)\to X_{\mu^{\on{ab}}}(b^{\on{ab}})$ by \eqref{map-on-integral-Sht} and \Cref{functoriality-Xmub}, respectively.
Since $\Gder=\Gsc$, we automatically have $G^\circ=\Gab(\bbQ_p)$ and  $\pi_1(G)=X_*(\Gab)$, which induces an isomorphism $\pi_1(G)_I=X_*(\Gab)_I$. In this case, by functoriality of the Kottwitz map ${\kappa}$, we have the following commutative diagram
\begin{equation}\label{commutative-diagram-kottwitz-map-functoriality}
    \begin{tikzcd}
		\pi_0(X_\mu(b)) \arrow{r}{{\omega_G}}{} \arrow{d}[swap]{\pi_0(\on{det})} &  \pi_1(G)_I \arrow{d}{\cong} \\
		X_{\mu^{\on{ab}}}(b^{\on{ab}})	 \arrow{r}{{\omega_{G^{\on{ab}}}}}{} &  X_*(G^{\on{ab}})_I.
	\end{tikzcd}
\end{equation}
This fits into the following commutative diagram (sse \Cref{torisecondmain})
\begin{equation}
	\begin{tikzcd}\label{commutative-diagram-kottwitz-map-factorthrough}
		\pi_0(X_\mu(b)) \arrow{r}{\omega_G}[swap]{} \arrow{d}[swap]{\pi_0(\on{det})} &  c_{b,\mu}\pi_1(G)_I^\varphi \arrow{d}{\cong}\arrow[hook]{r}{} &\pi_1(G)_I\arrow{d}{\cong}\\
	X_{\mu^{\on{ab}}}(b^{\on{ab}})	 \arrow{r}{\omega_{G^{\on{ab}}}}[swap]{\cong} &  c_{b^{\on{ab}},\mu^{\on{ab}}}X_*(G^{\on{ab}})_I^\varphi \arrow[hook]{r}{} & X_*(G^{\ab})_I.
	\end{tikzcd}
\end{equation}
In particular, it suffices to prove that left-hand side vertical arrow is a bijection.
By functoriality of the specialization map \cite[Proposition 4.14]{Specializ} and \Cref{adlvtosht}, we have
\begin{equation}
	\label{specializationdiag}
	\begin{tikzcd}
		 \pi_0(\Sht_{(G,b,\mu,{\calI(\bbZ_p)})}\times \Spd \bbC_p) \arrow{r}{\pi_0(\on{sp})}[swap]{\cong} \arrow{d}[swap]{\pi_0(\on{det})} &  \pi_0(X_\mu(b)) \arrow{d}{\pi_0(\on{det})} \\ 
		  \Sht_{(\Gab,b^{\on{ab}},\mu^{\on{ab}},\calG^{\on{ab}})} \times \Spd \bbC_p \arrow{r} {\pi_0(\on{sp})}[swap]{\cong} & X_{\mu^{\on{ab}}}(b^{\on{ab}}) 
	\end{tikzcd}
\end{equation}
Note that by \Cref{general-lemma-profinite-qt} we have the following identification
\begin{align*}
\pi_0(\Sht_{(G,b,\mu,{\calI(\bbZ_p)})}\times \Spd \bbC_p)&=\pi_0\left(\Sht_{(G,b,\mu,\infty)}\times \Spd \bbC_p\right/\underline{{\calI(\bbZ_p)}})\\
&=\pi_0(\Sht_{(G,b,\mu,\infty)}\times\Spd\bbC_p)/{\calI(\bbZ_p).}
\end{align*}
Since $\pi_0(\Sht_{(G,b,\mu,\infty)}\times \Spd \bbC_p)$ is a $G^{\ab}(\bbQ_p)$-torsor 
(i.e.~assumption (5) of \Cref{secondthmmain}), 
up to choosing an $x\in\pi_0(\Sht_{(G,b,\mu,\infty)}\times \Spd \bbC_p)$, we have compatible identifications $\pi_0(\Sht_{(G,b,\mu,\infty)}\times \Spd \bbC_p)\cong G^{\on{ab}}(\bbQ_p)$ and 
\begin{equation}
\pi_0(\Sht_{(G,b,\mu,{\calI(\bbZ_p)})}\times \Spd \bbC_p)\cong G^{\on{ab}}(\bbQ_p)/\on{det}(\calI(\bbZ_p)).
\end{equation}
Analogously, taking $x^{\on{ab}}\in\pi_0(\Sht_{(G^{\on{ab}},b^{\on{ab}},\mu^{\on{ab}},\infty)}\times \Spd \bbC_p)$ as $x^{\on{ab}}=\pi_0(\on{det}(x))$, we obtain a compatible identification $\pi_0(\Sht_{(\Gab,b^{\on{ab}},\mu^{\on{ab}},\calG^{\on{ab}})}\times \Spd \bbC_p)=G^{\on{ab}}(\bbQ_p)/\calG^{\on{ab}}(\bbZ_p)$ by \Cref{lemmagabwhatever}. Moreover, the map $\det: \pi_0(\Sht_{G,b,\mu,\infty}\times \Spd \bbC_p)\to \pi_0(\Sht_{(\Gab,b^{\on{ab}},\mu^{\on{ab}},\infty)}\times \Spd \bbC_p)$ is equivariant with respect to the $G(\bbQ_p)$-action on the left and the $G^{\ab}(\bbQ_p)$-action on the right. Thus we have the following commutative diagram:
\begin{center}
	\begin{tikzcd}
		\Gab(\bbQ_p)/\on{det}(\calI(\bbZ_p)) \arrow{r}{\cong} \arrow{d}[swap]{\on{det}} & \pi_0(\Sht_{(G,b,\mu,{\calI(\bbZ_p)})}\times \Spd \bbC_p) \arrow{r}{\cong}[swap]{\pi_0(\on{sp})} \arrow{d}{\on{det}} &  \pi_0(X_\mu(b)) \arrow{d}{\on{det}} \\ 
		\Gab(\bbQ_p)/\calG^{\on{ab}}(\bbZ_p) \arrow{r}{\cong} & \Sht_{(\Gab,b^{\on{ab}},\mu^{\on{ab}},\calG^{\on{ab}})} \times \Spd \bbC_p \arrow{r} {\cong}[swap]{\pi_0(\on{sp})} & X_{\mu^{\on{ab}}}(b^{\on{ab}}) 
	\end{tikzcd}
\end{center}
Thus in order to prove that the vertical arrow on the left-hand side is a bijection, it suffices to show that $\calI\to \calG^{\on{ab}}$ is surjective on the level of $\bbZ_p$-points. 
Now, the map $\calI\to \calG^\ab$ is smooth. 
In particular, the surjectivity of $\bbZ_p$-points follows from the surjectivity of $\bbF_p$-points. 
On the other hand, the $\bbF_p$-fibers of the map $\calI\to \calG^\ab$ are $\calI^\der_{\bbF_p}$-torsors.
Since $\calI^\der$ is parahoric, $\calI^\der_{\bbF_p}$ is smooth and connected. 
An application of Lang's theorem shows that all $\calI^\der_{\bbF_p}$-torsors are trivial.
\end{proof}

\begin{remark}
	Note that the proof of \Cref{5-implies-1} does not use that $\calI$ is an Iwahori, so one could have shown that implication $(5)\implies (2)$ holds directly for $\calK_p$ any parahoric group scheme.
\end{remark}

\subsection{Argument for $(4)\overset{\on{iso.fact.}}{\implies} (5)$}
\label{subsection3implies4}

\begin{proposition}
	If $G^{\on{ad}}$ only has isotropic $\bbQ_p$-simple factors then $(4)\implies (5)$ in \Cref{secondthmmain}.
\end{proposition}
\begin{proof}
As seen earlier (for example in $\mathsection$\ref{secondthm-4-to-1}), the map 
\begin{equation}\det: \pi_0(\Sht_{G,b,\mu,\infty}\times \Spd \bbC_p)\to \pi_0(\Sht_{(\Gab,b^{\on{ab}},\mu^{\on{ab}},\infty)}\times \Spd \bbC_p)
\end{equation}
is equivariant with respect to the $G(\bbQ_p)$-action on the source and the $G^{\ab}(\bbQ_p)$-action on the target. 
By the assumption that $G^\der=\Gsc$ (in particular $G^\circ=\Gab(\bbQ_p)$), it suffices to show that
\begin{equation}\label{det-pi0}
\pi_0(\on{det}):\pi_0(\Sht_{(G,b,\mu,\infty)}\times \Spd \bbC_p)\to \Sht_{(G^\ab,b^\ab,\mu^\ab,\infty)}\times \Spd \bbC_p
\end{equation}
is bijective. 
Since the map $G(\bbQ_p)\to G^{\ab}(\bbQ_p)$ is surjective, by equivariance of the respective group actions, the map \eqref{det-pi0} is always surjective. 
By \Cref{corollary-badmconn}, $G(\bbQ_p)$ acts transitively on $\pi_0(\Sht_{(G,b,\mu,\infty)}\times \Spd \bbC_p)$, thus up to picking a $t\in \pi_0(\Sht_{(G,b,\mu,\infty)}\times \Spd \bbC_p)$ we have an identification of sets $\pi_0(\Sht_{(G,b,\mu,\infty)}\times \Spd \bbC_p)\cong G(\bbQ_p)/H_t$ for some subgroup $H_t:=\mathrm{Stab}(t)$. To prove (5), it suffices to show that $H_t=G^{\der}(\bbQ_p)$. Firstly, it is easy to see that $H_t\subseteq G^{\der}(\bbQ_p)$: take any $g\in H_t$, we have $g\cdot t=t$; thus $\deg(g)\cdot\det(t)=\det(g\cdot t)=\det(t)$; by the tori case (see $\mathsection$\ref{tori-case-section}), $\det(g)$ is trivial, thus $g\in G^{\der}(\bbQ_p)$.

We now prove the other inclusion, i.e.~that $G^{\der}(\bbQ_p)$ acts trivially on $\pi_0(\Sht_{(G,b,\mu,\infty)}\times \Spd \bbC_p)$. We argue instead over finite extensions of $\br{\bbQ}_p$. 
Recall from \cite[Lemma 12.17]{Et} that, the underlying topological space of a cofiltered inverse limit of locally spatial diamonds along qcqs\footnote{i.e. quasicompact quasiseparated} transition maps is the limit of the underlying topological spaces. 
Moreover, 
\[\Sht_{(G,b,\mu,\infty)}\times \Spd \bbC_p=\varprojlim_{\breve{\bbQ}_p\subseteq K \subseteq \bbC_p} \Sht_{(G,b,\mu,\infty)}\times \Spd K \]
as $K$ ranges over Galois extensions that are finite over $\breve{\bbQ}_p$. 
Indeed, the similar claim $\Spd \bbC_p=\varprojlim_{\breve{\bbQ}_p\subseteq K \subseteq \bbC_p}\Spd K$ holds and limits commute with products. 
Therefore, it suffices to prove that $\Gder(\bbQ_p)$ acts trivially on 
\[\pi_0(\Sht_{(G,b,\mu,\infty)}\times \Spd K)\]
for all finite degree extensions $K$ over $\br{\bbQ}_p$ (or equivalently, for a cofinal family of such). 
For any fixed $x\in \pi_0(\Sht_{(G,b,\mu,\infty)}\times \Spd K)$, we denote by $G_x\subseteq G(\bbQ_p)$ the stabilizer of $x$. Let $G^\der_x:=G_x\cap \Gder(\bbQ_p)$. It suffices to prove that $G_x^{\der}=G^{\der}(\bbQ_p)$, which is shown in \Cref{stab-equal-Gder} after some preparations in \Cref{open-stab-lem} and \Cref{normalizer-fin-index}. 
\end{proof}

\begin{lemma}\label{open-stab-lem}
	$G^\der_x$ is open in $\Gder(\bbQ_p)$.	
\end{lemma}
\begin{proof}
	For any choice of $y\in \Gr^b_\mu(K)$, let $\mathcal{T}_y:=\Sht_{(G,b,\mu,\infty)}\times_{\Gr^b_\mu} \Spd K$ be the fiber of $y$ under the Grothendieck--Messing period morphism. 
	Take an arbitrary $w\in \pi_0(\mathcal{T}_y)$. 
	Since $\mathcal{T}_y\to \Spd K$ is a $\underline{G(\bbQ_p)}$-torsor and $\Spd K$ is connected, the $G(\bbQ_p)$-action on $\pi_0(\mathcal{T}_y)$ is transitive (see \Cref{general-lemma-profinite-qt}).
Similarly, by \Cref{corollary-badmconn}, the $G(\bbQ_p)$-action on $\pi_0(\Sht_{(G,b,\mu,\infty)}\times\Spd K)$ is also transitive.
	Observe that the map 
	\[\pi_0(\mathcal{T}_y)\to \pi_0(\Sht_{(G,b,\mu,\infty)}\times\Spd K)\]
	is $G(\bbQ_p)$-equivariant. 
	Using this, we may assume without loss of generality that $w\mapsto x$ under the surjection $\pi_0(\mathcal{T}_y)\to \pi_0(\Sht_{(G,b,\mu,\infty)}\times\Spd K)$. 
	Consider $G_w^{\der}:=G_w\cap G^{\der}(\bbQ_p)$, and the inclusion of groups $G^\der_w\subseteq G^\der_x\subseteq G^\der(\bbQ_p)$. 
	It suffices to find a $y\in \Gr^b_{\mu}(K)$, such that 
\[(*)\quad\text{there exists a }w\in \pi_0(\mathcal{T}_y)\text{ with  }G^\der_w\text{ open in } G^\der(\bbQ_p).\] 
	
Recall the $\underline{G(\bbQ_p)}$-torsor $\bbL_b$ over $\Gr^b_{\mu}$ from $\mathsection$ \ref{grothendieck-section}. 
Let $y^*\bbL_b$ be the corresponding torsor over $\Spd K$, which induces a crystalline representation $\rho_y:\Gamma_K\to G(\bbQ_p)$, well-defined up to conjugacy. We claim that $G_w$ is equal to $\rho_y(\Gamma_K)$ up to $G(\bbQ_p)$-conjugacy. 
We now justify the claim. Consider the pullback $\mathcal{T}_t$ of $\mathcal{T}_y$ under the geometric point $t:\Spd\bbC_p\to \Spd K$. Thus $\mathcal{T}_t$ is a trivial torsor that gives a section $s:\Spd\bbC_p\to \mathcal{T}_t$. 
The Galois action of $\Gamma_K$ on $\mathcal{T}_t$ defines a representative of the crystalline representation $\rho_y$. 
We claim that the orbit $\Gamma_K\cdot s$ descends to a unique connected component $w_s\in \pi_0(\mathcal{T}_y)$. 
Indeed, the $\Gamma_K$-action encodes descent datum along the v-cover $\Spd \bbC_p\to \Spd K$, and any $\Gamma_K$-stable closed subsheaf $\calF_t\subseteq \calT_t$ gives rise, by pro-\'etale (or v-)descent to a unique closed subsheaf $\calF_y\subseteq \calT_y$. 
This will further satisfy $\pi_0(\calF_y)=\pi_0(\calF_t)/\Gamma_K$ by \Cref{general-lemma-profinite-qt}.
Applying this reasoning to a $\Gamma_K$-orbit we obtain our claim.
Therefore, for any $g\in G(\bbQ_p)$ such that $g\cdot s\in \Gamma_K\cdot s$, we have $g\cdot w_s=w_s$. 
This gives us the desired claim. 
As we have seen, to find a point $y\in \Gr^b_{\mu}(K)$ satisfying $(*)$ it suffices to find a crystalline representation with $G$-structure $\rho_y:\Gamma_K\to G(\bbQ_p)$ that has invariants $(\bfb,\bfmu)$ and for which $\rho_y(\Gamma_K)\cap G^\der(\bbQ_p)$ is open in $G^\der(\bbQ_p)$. 
But hypothesis $(4)$ of \Cref{secondthmmain} ensures that such a crystalline representation exists.  
\end{proof}

\begin{lemma}\label{normalizer-fin-index}  
	Let $N_x$ denote the normalizer of $G_x$ in $G(\bbQ_p)$, then $N_x$ has finite index in $G(\bbQ_p)$. In particular, $N_x$ contains $\Gder(\bbQ_p)$.
\end{lemma}
\begin{proof}	
	Let $S$ be the set of orbits of $\pi_0(\Sht_{(G,b,\mu,{\calI(\bbZ_p)})}\times \Spd K)$ under the $J_b(\bbQ_p)$-action. 
	By \cite[Theorem 1.2]{HV20}, $S$ is finite. 
	For each $s\in S$, we choose a representative $x_s\in \pi_0(\Sht_{(G,b,\mu,\infty)}\times \Spd K)$ that maps to $s$ under the map $\pi_0(\Sht_{(G,b,\mu,\infty)}\times \Spd K)\to \pi_0(\Sht_{(G,b,\mu,{\calI(\bbZ_p)})}\times \Spd K)\to S$. 
	We can always arrange that $x$ is in this set of representatives for some $s$. 
	By \Cref{corollary-badmconn}, and since we have an identification
	\[\pi_0(\Sht_{(G,b,\mu,\infty)}\times \Spd K)=\pi_0(\Sht_{(G,b,\mu,\infty)}\times \Spd \bbC_p)/\Gamma_K\]
	(using \Cref{general-lemma-profinite-qt}) that is compatible with $G(\bbQ_p)$-action,
	we can find an element $h_s\in G(\bbQ_p)$ such that $x_s\cdot h_s=x$, for each $s\in S$.  
	We construct a surjection   
		$\coprod_{s\in S} {\calI(\bbZ_p)}\cdot h_s \twoheadrightarrow G(\bbQ_p)/N_x$.	
	We do this in two steps. 
The first step is to construct, for any $g\in G(\bbQ_p)$, a triple $(i,j,s)$ where $i\in {\calI(\bbZ_p)}$, $j\in J_b(\bbQ_p)$ and $s\in S$ such that 
\begin{equation}\label{desired-triple}
j\cdot (x \cdot g)\cdot i=x_s
\end{equation}
(note that $s$ is uniquely determined by $x$ and $g$). 
We do this by choosing $j$ so that $j\cdot (x\cdot g)$ and $x_s$ map to the same element in $\pi_0(\Sht_{(G,b,\mu,{\calI(\bbZ_p)})}\times \Spd K)$. 
	Since ${\calI(\bbZ_p)}$ acts transitively on the fibers of the map $\pi_0(\Sht_{(G,b,\mu,\infty)}\times \Spd K) \to \pi_0(\Sht_{(G,b,\mu,{\calI(\bbZ_p)})}\times \Spd K)$, there exists an $i$ satisfying \eqref{desired-triple}. Thus we have
	\begin{equation}\label{bad-j-equation}
	j\cdot x \cdot (gih_s)=x 
	\end{equation}
The second step	is to eliminate $j$ from \eqref{bad-j-equation}. 
By \Cref{corollary-badmconn}, there exists an $n\in G(\bbQ_p)$ such that $x\cdot n=j\cdot x$. We now show that $n\in N_x$. 
	Indeed, $n^{-1}G_x n=G_{x\cdot n}=G_{j\cdot x}=G_x$ since the actions of $J_b(\bbQ_p)$ and $G(\bbQ_p)$ commute. Thus we have $(x\cdot n) \cdot (gih_s)=x$. 
	Since $G_x\subseteq N_x$, in particular $n\cdot (gih_s)\in G_x\subseteq N_x$. Thus $g\cdot i\cdot h_s\in N_x$, and we have a natural surjection:
	\begin{equation}\label{desired-surjection}
		\begin{aligned}
			\coprod_{s\in S} {\calI(\bbZ_p)}\cdot h_s & \twoheadrightarrow G(\bbQ_p)/N_x 	 \\
			(s,i) & \mapsto i\cdot h_sN_x= g^{-1}N_x.
		\end{aligned}
	\end{equation}
The target of \eqref{desired-surjection} is discrete, and the source is compact. Thus the index of $N_x$ in $G(\bbQ_p)$ is finite. 

Recall that $\Gder=\Gsc$. 
Since $G^{\der}$ only has isotropic $\bbQ_p$-simple factors, and $N_x\cap G^\der(\bbQ_p)$ has finite index in $G^\der(\bbQ_p)$, we have $N_x\cap G^\der(\bbQ_p)=G^\der(\bbQ_p)$. 
Indeed, it is a standard fact that $\Gder(\bbQ_p)$ has no open subgroups of finite index \cite[Chapter II, Theorem 5.1]{Marg}, thus we are done.   
\end{proof}
\begin{lemma}\label{stab-equal-Gder}
	Assume hypothesis (4) in \Cref{secondthmmain} for $\breve{\bbQ}_p\subseteq K$. 
	Let $K\subseteq K'$ be any finite Galois extension of $K$. 
	Pick any $x\in \pi_0(\Sht_{(G,b,\mu,\infty)}\times \Spd K')$.
Then, $G^\der_x=G^\der(\bbQ_p)$.
\end{lemma}
\begin{proof}
	Without loss of generality we may assume $K'=K$.
	Indeed, given a crystalline representation $\xi:\Gamma_K\to G(\bbQ_p)$ satisfying the conditions of hypothesis (4) in \Cref{secondthmmain}, its restriction $\xi_{\Gamma_{K'}}$ is also a crystalline representation satisfying the same hypothesis.
By \Cref{open-stab-lem} and \Cref{normalizer-fin-index},  $G^\der_x\subseteq \Gder(\bbQ_p)$ is open and normal. 
This already implies $G^\der_x= \Gder(\bbQ_p)$, since $\Gder(\bbQ_p)$ does not have open normal subgroups (see \Cref{lemma-G--hasno-nrom}). 
\end{proof}

\begin{lemma}
	\label{lemma-G--hasno-nrom}
If $G=\Gder=\Gsc$ is a reductive group with isotropic $\bbQ_p$-simple factors.	
Then $G(\bbQ_p)$ has no open normal subgroup.
\end{lemma}
\begin{proof}
We can argue on $\bbQ_p$-simple factors and assume without loss of generality that $G$ is $\bbQ_p$-simple.
If $K\subseteq \Gder(\bbQ_p)$ is an arbitrary open and normal subgroup, it follows from \cite[Proposition 2.2.16]{KP23} that $K$ must be unbounded since its normalizer is unbounded. 
Let $G(\bbQ_p)^+\subseteq G(\bbQ_p)$ be the normal subgroup generated by the $\bbQ_p$-rational elements of the unipotent radicals of parabolic $\bbQ_p$-subgroups of $G$.
By \cite[Proposition 2.2.14]{KP23}, $K$ contains $G(\bbQ_p)^+$. Since $G$ is simply connected, by the positive solution of the Kneser--Tits problem over local fields, $G(\bbQ_p)^+=G(\bbQ_p)$ (see for example \cite[\S 2]{MR787664}).
\end{proof}
This finishes the argument for $(4)\implies (5)$.

\begin{proof}[Proof of \Cref{secondthmmain}]
	We have now justified the circle of implications 
	\[(1)\implies (2)\overset{\on{q.split}}{\implies} (3) \implies (4) \overset{\on{iso.fact.}}{\implies} (5)\implies (1),\]
	 i.e.~\Cref{secondthmmain} holds in the case $\Gder=\Gsc$. But by \Cref{reductionsofsecondthm} the general case follows.   
\end{proof}

	\begin{proof}[Proof of \Cref{mainthmmain}] \textit{(i.e. also Proof of \Cref{mainthm})} 
Using z-extension ad-isomorphisms and decomposition into products  (\Cref{checkafteradiso}, \Cref{checkafterprod}, \Cref{adisoforadlv} and \Cref{productsopen}), we may assume without loss of generality that $\Gder=\Gsc$ and that $G^{\on{ad}}$ is $\bbQ_p$-simple.
We split our analysis into two cases: (1) when $G^{\on{ad}}$ is isotropic, and (2) when $G^{\on{ad}}$ is anisotropic. 
The first case holds from the implication $(3)\implies (2)$ of \Cref{secondthmmain}.(b). 

We now consider the case where $G^{\on{ad}}$ is anisotropic. Recall that
\begin{equation}\Gr^b_\mu\times \Spd \bbC_p=\Sht_{(G,b,\mu,\infty)}\times \Spd \bbC_p/\underline{G(\bbQ_p)}.
\end{equation}
When $G^{\on{ad}}$ is $\bbQ_p$-simple and anisotropic, its Bruhat--Tits building consists of one point \cite[Proposition 9.3.9]{KP23}. 
In particular, there is a unique Iwahori $\calI(\bbZ_p)\subseteq G(\bbQ_p)$. 
Recall from Bruhat--Tits theory, that the set of Iwahori subgroups forms the set of minimal parabolic subgroups for a saturated Tits system over the component group $G(\bbQ_p)^0$ defined as in \cite[Definition 2.6.23]{KP23} (see \cite[Axiom 4.1.9]{KP23}). 
Uniqueness of the Iwahori implies that $\calI(\bbZ_p)=G(\bbQ_p)^0$.
By \cite[Remark 4.1.13]{KP23}, it also follows that $\Gder(\bbQ_p)=\Gsc(\bbQ_p)\subseteq \calI(\bbZ_p)$.
Moreover, it follows from \cite[Corollary 11.7.6]{KP23} that $\calI(\bbZ_p)\subseteq G(\bbQ_p)$ is normal and that the Kottwitz map gives us identifications
\[G(\bbQ_p)/{\calI(\bbZ_p)}=\pi_1(G)_I^\varphi=G^{\on{ab}}(\bbQ_p)/\calG^{\on{ab}}(\bbZ_p)\] (see $\mathsection$ \ref{tori-case-section} for the last identification).
Since ${\calI(\bbZ_p)}$ is normal in $G(\bbQ_p)$, $\Sht_{(G,b,\mu,{\calI(\bbZ_p)})}\times \Spd \bbC_p$ becomes a $\pi_1(G)_I^\varphi$-torsor over $\Gr^b_\mu\times \Spd \bbC_p$. Moreover, the map 
\begin{equation}
\on{det}:\Sht_{(G,b,\mu,{\calI(\bbZ_p)})}\times \Spd \bbC_p \to \Sht_{(G^{\on{ab}},b^{\on{ab}},\mu^{\on{ab}},\calG^{\on{ab}}(\bbZ_p))} \times \Spd \bbC_p
\end{equation}
is $\pi_1(G)^\varphi_I$-equivariant. Since $\Sht_{(G,b^{\on{ab}},\mu^{\on{ab}},\calG^{\on{ab}}(\bbZ_p))} \times \Spd \bbC_p$ is a $\pi_1(G)^\varphi_I$-torsor over $\Spd \bbC_p$, 
$\Sht_{(G,b,\mu,{\calI(\bbZ_p)})}\times \Spd \bbC_p$ is the trivial $\pi_1(G)_I^\varphi$-torsor over $\Gr_\mu^b\times \Spd \bbC_p$. That is
\begin{equation}\label{Gcirc-copies-of-Grbmu}
\Sht_{(G,b,\mu,{\calI(\bbZ_p)})}\times \Spd \bbC_p\cong (\Gr^b_\mu\times \Spd \bbC_p)\times \underline{\pi_1(G)_I^\varphi}.
\end{equation}

Taking $\pi_0$ in \eqref{Gcirc-copies-of-Grbmu}, we have 
\begin{equation}
	\label{equation-aniso}
\pi_0(\Sht_{(G,b,\mu,{\calI(\bbZ_p)})}\times \Spd\bbC_p)\cong \pi_0(\Gr^b_\mu\times \Spd \bbC_p)\times \pi_1(G)_I^\varphi.
\end{equation}
By \Cref{badmconn}, $\pi_0(\Sht_{(G,b,\mu,{\calI(\bbZ_p)})}\times \Spd\bbC_p)\cong \pi_1(G)_I^\varphi$ and the map 
\[\omega_G\circ \pi_0(\on{sp}):\pi_0(\Sht_{(G,b,\mu,{\calI(\bbZ_p)})})\times \Spd \bbC_p) \to c_{b,\mu}\pi_1(G)^\varphi_I\] is an isomorphism as we needed to show.  
We can finish by recalling that the map of \eqref{bijectivityofspec} is bijective. 
\end{proof}

\bibliography{biblio.bib}
\bibliographystyle{alpha}
\end{document}